\documentclass[a4paper,12pt]{amsart}

\usepackage{amsmath,amssymb,amsthm,graphicx,pstricks,comment,accents}

\usepackage[all]{xy}

\usepackage{lscape}
\usepackage{hyperref}
\usepackage{tikz}
\usetikzlibrary{arrows}
\usetikzlibrary{decorations}
\usetikzlibrary{shapes}

\author{Claire Amiot}
\address{Institut Fourier, 100 rue des maths, 38402 Saint Martin d'H\`eres}
\email{claire.amiot@univ-grenoble-alpes.fr}

\author{Thomas Br\"ustle}
\address{Bishop's University, 2600 College Street,
Sherbrooke Qu\'ebec J1M 1Z7}
\email{tbruestl@ubishops.ca}
\address{Universit\'e de Sherbrooke, 2500, boul. de l'Universit\'e,
Sherbrooke Qu\'bec J1K 2R1}
\email{Thomas.Brustle@USherbrooke.ca}

\thanks{Most of this work was done while the first author was in UMI 3457 Centre de Recherche Math\'ematiques de Montr\'eal, D\'epartement de math\'ematique de l'Universit\'e de Sherbrooke. She would like to thank the CNRS and UMI 3457 for financial support and Sherbrooke University for warm hospitality. She was also supported by the French ANR grant SC3A (ANR-15-CE40-0004-01).
The second author was supported by Bishop's University and NSERC of Canada.}
\numberwithin{equation}{section}

\newtheorem{theorem}{Theorem}[section]

\newtheorem{lemma}[theorem]{Lemma}
\newtheorem{corollary}[theorem]{Corollary}
\newtheorem{proposition}[theorem]{Proposition}
\theoremstyle{remark}
\newtheorem{remark}[theorem]{Remark}
\newtheorem{example}[theorem]{Example}
\theoremstyle{definition}
\newtheorem{definition}[theorem]{Definition}

\DeclareMathOperator{\add}{add}

\DeclareMathOperator{\Hom}{Hom}
\DeclareMathOperator{\End}{End}
\DeclareMathOperator{\thick}{thick}

\renewcommand{\mod}{\operatorname{mod}}

\newcommand{\surf}{\mathcal{S}}
\newcommand{\orbifold}{\overline{\mathcal S}}
\newcommand{\dual}{\widehat{G}}
\newlength{\dhatheight}
\newcommand{\doublehat}[1]{%
    \settoheight{\dhatheight}{\ensuremath{\widehat{#1}}}%
    \addtolength{\dhatheight}{-0.05ex}%
    \widehat{\vphantom{\rule{1pt}{\dhatheight}}%
    \smash{\!\widehat{#1}}}}

\newcommand{\lten}[1]{\underset{#1}{\overset{\textbf{L}}{\otimes}}}
\newcommand{\ten}[1]{\underset{#1}{\otimes}}
\newcommand{\smallten}{{\scriptstyle \otimes}}

\definecolor{dark-green}{RGB}{14,150,2}
\newcommand{\gpoint}{\color{dark-green}{\circ}}
\newcommand{\rpoint}{{\red \bullet}}
\newcommand{\cross}{{\sf \red x}}
\newcommand{\Z}{\mathbb Z}
\newcommand{\grading}{\mathbf{n}}

\title[Derived equivalences between skew-gentle algebras]{Derived equivalences between skew-gentle algebras using orbifolds}

\begin{document}
\maketitle
\begin{abstract}
Skew-gentle algebras are skew-group algebras of gentle algebras equipped with  a certain $\Z_2$-action. Building on the bijective correspondence between gentle algebras and dissected surfaces, we obtain in this paper a bijection between skew-gentle algebras and certain dissected orbifolds that admit a double cover. 

We prove the compatibility of the $\Z_2$-action on the double cover with the skew-group algebra construction.  This allows us to investigate the derived equivalence relation between skew-gentle algebras in geometric terms: We associate to each skew-gentle algebra a line field on the orbifold, and on its double cover, and interpret different kinds of derived equivalences of skew-gentle algebras in terms of diffeomorphisms respecting the homotopy class of the line fields associated to the algebras. 
\end{abstract}

\tableofcontents
\section{Introduction}
Gentle algebras, introduced in the 80's \cite{AsSk}, provide an example of a class of algebras  whose derived category can be described explicitly (\cite{BM} and \cite{BuDr}).  The class of gentle algebras  contains all finite dimensional path algebras of type $\mathbb A$ and $\widetilde{\mathbb{A}}$ and has been shown to be stable under derived equivalences \cite{SchZi}. 
More recently, gentle algebras have been found to be deeply and surprisingly connected to the combinatorics and geometry of marked surfaces: 
The Jacobian algebra of a triangulation of an unpunctured surface $(\surf,M)$ is a gentle algebra \cite{ABCP,LF}. Thus certain gentle algebras appear as endomorphism ring of  cluster-tilting objects in the cluster category $C(\surf,M)$ associated in \cite{Am} to the cluster algebra of a marked surface $(\surf,M)$ without punctures defined in \cite{FST}.  Building on this, \cite{BZ} provide a geometric model for the objects in the cluster category $C(\surf,M)$ associating strings and bands with curves and closed curves. 

Obviously, triangulations of surfaces yield only certain gentle algebras. This shortcoming has been overcome in \cite{BCS} and \cite{OPS} by relating every gentle algebra to a dissection of a marked surface, cutting $(\surf,M)$ into polygons. Using this correspondence \cite{BCS} give a geometric description of the module category of a gentle algebra, while \cite{OPS} provide a description of its derived category. Note that a link between gentle algebras and ribbon graphs, thus again surfaces, already appeared in \cite{Sch}.

Independently, \cite{HKK} establish a description of the (partially wrapped) Fukaya category of a surface $\surf$ with stops using the derived category of a (graded) gentle algebra associated to these data, also given by a dissection of $\surf$, see also \cite{LP} for discussion of the derived equivalences. 

Combining results in \cite{OPS} and \cite{LP}, a geometric interpretation of the derived equivalence relation for gentle algebras is given in \cite{APS} and \cite{Opper}.

\medskip
We aim in this paper to extend these results to orbifolds $\bar{\surf}$ admitting a two-fold cover. The two-fold cover $\surf$ corresponds to a gentle algebra which comes equipped with a $\Z_2$-action. The corresponding skew-group algebra is studied in \cite{GePe}, called skew-gentle algebra. This class of algebras contains in particular all path algebras of type $\mathbb D$ and $\widetilde{\mathbb D}$. In fact, these algebras had been studied earlier under the name clannish algebra in \cite{CB}, motivated by a matrix problem notion of clan, see also  \cite{De}, but the viewpoint of skew group algebra allows to use general results from \cite{ReitenRiedtmann}.
We employ this point of view, where a description of the derived category of a skew-gentle algebra can be obtained using the $\Z_2$-action, and the known results for gentle algebras.

Looking back to the cluster algebra of a triangulated surface, the orbifold points correspond to punctures, and the fact that the Jacobian algebra admits a $\Z_2$-action corresponds to having all orbifold points lying in a self-folded triangle. This case has been studied in \cite{GL-FS}, including a deformation argument similar to the one employed in \cite{Br2} which reduces the study to a gentle algebra.
The description of the cluster category for punctured surfaces with skew-gentle algebras has been given in \cite{QZ} using orbifolds, and in \cite{AP} using a $\Z_2$-action on the category and on the surface. 
We follow in this paper a similar approach to the one in \cite{AP}, generalizing it to study the derived category in the case of an orbifold allowing a dissection such that all orbifold points are uniquely connected by an arc to the boundary (this is the polygonal equivalent of the self-folded triangle in the cluster situation) 

Of course, the class of skew-gentle algebras is not stable under derived equivalences, not even the simplest case of type ${\mathbb D}$ satisfies this. It is however natural to ask the following  question:

\emph{ What is the geometric interpretation of the derived equivalence relation for skew-gentle algebras ? }

Furthermore, keeping track of the $\mathbb Z_2$-action, we can refine the question to $\mathbb Z_2$-derived equivalence relations. These are the two main questions we address in this paper. 

\medskip

{\em Organization and main results of the paper:} 
We first study some general properties of $G$-invariant objects in the derived category of an algebra $\Lambda$, for some finite group $G$ acting on $\Lambda$. More precisely, we study the $G$-invariant tilting objects in the derived category of $\Lambda$, and relate them with the $\dual$-invariant tilting objects of the derived category of the skew group algebra $\Lambda G$.

In section 3 we introduce the class of skew-gentle algebras, describing their quiver and relations and various properties. We then provide a geometric model for skew-gentle algebras using certain dissections of a surface that we call $\cross$-dissections. This simultaneously generalizes results from \cite{OPS} for the gentle case, and  from \cite{LF} for triangulations (where each puncture is in a selfolded triangle) of a punctured surface. 

In section 4, we study the $\Z_2$-action, both on the algebraic side of the skew-gentle algebras, and on their geometric realizations. 
To any dissected surface which is invariant under the action of an order-$2$ diffeomorphism (with finitely many fixed points), we associate

\begin{itemize}
\item a gentle algebra $\Lambda$ together with a $\Z_2$ action;
\item and an orbifold together with a $\cross$-dissection.
\end{itemize}
We then show that the skew-gentle algebra corresponding to the $\cross$-dissection is Morita equivalent to the skew-group algebra $\Lambda \Z_2$.
Conversely, given a skew-gentle algebra, we construct a $2$-folded cover of the corresponding orbifold that satisfies the above properties. This construction combined with the results of section 2 permits us to prove that two skew-gentle algebras are $\Z_2$-derived equivalent if and only if their corresponding gentle algebras are $\Z_2$-derived equivalent.

Section 5 generalizes results from \cite{APS} to the setting of orbifolds with a $\Z_2$-cover.
We equip the $2$-folded cover $(\surf,\sigma)$ associated with a skew-gentle algebra $\bar{\Lambda}$ with a $\sigma$-invariant line field $\eta$. We then adapt the results in \cite{APS} to the $\Z_2$-action setting and give a complete answer to the second question asked above:
\begin{theorem}(\ref{theorem::dual-derived equivalence}) 
Two skew-gentle algebras $\bar{\Lambda}$ and $\bar{\Lambda}'$ are $\Z_2$-derived equivalent if and only if there exists a diffeomorphism  between their corresponding $2$-folded covers commuting with the $\Z_2$-action and sending $\eta$ to $\eta'$ up to homotopy.
\end{theorem}


Finally, we give a geometric interpretation of the derived equivalence relation for skew-gentle algebras when the equivalence is given by a $\Z_2$-invariant tilting object. The $\Z_2$-invariant line field $\eta$ of the double cover induces a line field $\bar{\eta}$ on the orbifold, and we have the following characterization:

\begin{theorem}(\ref{thm::derived-skew-gentle})
Two skew-gentle algebras $\bar{\Lambda}$ and $\bar{\Lambda} '$ are derived equivalent via a $\Z_2$-invariant tilting object if and only if there exists a diffeomorphism between their corresponding orbifolds sending $\bar{\eta}$ to $\bar{\eta}'$ up to homotopy.
\end{theorem}

We finish by giving examples showing the subtle differences between these two results.

\section{{\em G}-derived equivalence between {\em G}-algebras and skew-group algebras}

Throughout this section $G$ is a finite abelian group, and $k$ is a field whose characteristic does not divide $|G|$. We denote  the dual (or character) group of  $G$ by $\widehat{G} =\Hom(G,k^{\times})$.

\subsection{Skew-group algebras}

We recall from \cite{ReitenRiedtmann} the notion and some properties of skew-group algebras.
By a $G$-algebra, we mean a finite dimensional $k$-algebra $\Lambda$ with an action of $G$ by automorphisms. Two $G$-algebras $\Lambda$ and $\Lambda'$ are said to be $G$-isomorphic if there exists an isomorphism $\varphi: \Lambda\to \Lambda'$ commuting with the action of $G$.

For $g\in G$, we denote by $\Lambda_g$ the $\Lambda$-bimodule which is $\Lambda$ as a left $\Lambda$-module, and whose action on the right is twisted by $g$, that is, the map $\lambda\mapsto g(\lambda)$ is an isomorphism of right $\Lambda$-modules $\Lambda\to \Lambda_g$. Likewise for the twisted left $\Lambda$-module $_g\Lambda$.

\begin{definition}
Let $\Lambda$ be a $G$-algebra. Then the skew-group algebra $\Lambda G$ is defined as follows:

\begin{itemize}
\item as $k$-vector space we have $\Lambda G=\Lambda \ten{k} kG$ ;
\item the multiplication is given by $(\lambda\smallten g).(\mu\smallten h)=\lambda g(\mu)\smallten gh$ extended by linearity and distributivity.

\end{itemize}
\end{definition}

The map $\lambda\mapsto \lambda\smallten 1_{G}$ is an algebra monomorphism $\Lambda \to \Lambda G$, and so $\Lambda G$ is naturally a $\Lambda-$bimodule, which decomposes as $\Lambda G \cong\bigoplus_{g\in G}\Lambda_g$. Moreover, $\Lambda G$ can be endowed with a $\widehat{G}$-action, which allows to consider the group algebra $\Lambda G \dual := (\Lambda G) \dual$ as follows:

\begin{proposition}\cite[Prop 5.1]{ReitenRiedtmann}\label{prop::isomorphismRR}
Let $\Lambda$ be a $G$-algebra, then $\Lambda G$ is a $\widehat{G}$-algebra with $\widehat{G}$-action given by $$\chi(\lambda\smallten g):=\chi(g)\lambda\smallten g \quad\mbox{ for all }\chi\in \widehat{G},\   \lambda\in \Lambda,\ g\in G.$$
The map $\Lambda G \dual \longrightarrow \End_{\Lambda}(\Lambda G)$ given by 
\begin{equation}\label{iso} \lambda\smallten g\smallten \chi\mapsto (\mu\smallten h\mapsto \chi(h) (\lambda\smallten g).(\mu\smallten h))
\end{equation}
is an isomorphism of algebras.
\end{proposition}

\begin{remark}
Since  $\Lambda G$ is isomorphic to the sum of $|G|$ copies of $\Lambda$ as a right $\Lambda$-module, the proposition above implies that $\Lambda$ is Morita equivalent to $\Lambda G\dual$.  
\end{remark}

\subsection{{\em G}-invariant objects}

An action of $G$ on $\Lambda$ induces an action on the category $\mathcal{D}^b(\mod \Lambda)$ on the right in the sense of \cite[3.1]{Elagin} as follows: For all $g \in G$ we set 

$$X^g:=X\lten{\Lambda} \Lambda_g $$
for all objects $X \in \mathcal{D}^b(\mod \Lambda)$, and for $f:X\to Y$, $$f^g:=f\lten{} 1_{\Lambda_g}.$$

\begin{definition}
An object $X$ in $\mathcal{D}^b(\mod \Lambda)$ is called $G$-invariant (or $G$-equivariant) if there exist isomorphisms $\iota_g:X^{g^{-1}}\to X$ for all $g\in G$ such that 
$$  \iota_{gh}=\iota_g\circ (\iota_{h})^{g^{-1}}$$
holds for all $g,h\in G$.
\end{definition}

With this definition, it is immediate to check the following (compare \cite{KrSo,Elagin}):

\begin{lemma}\label{G-End} If $X\in \mathcal{D}^b(\mod \Lambda)$ is $G$-invariant, then $G$ acts on $\End_{\mathcal{D}^b(\Lambda)}(X)$ by 
$$g.f:=\iota_g \circ f^{g^{-1}}\circ (\iota_g)^{-1}.$$
\end{lemma}
\begin{proof}
The definitions imply
 \begin{align*}
g.(h.f) &= \iota_g \circ [\iota_h \circ f^{h^{-1}}\circ (\iota_h)^{-1}]^{g^{-1}}\circ (\iota_g)^{-1}\\
       &=\iota_g \circ (\iota_h)^{g^{-1}} \circ f^{h^{-1}g^{-1}}\circ ((\iota_h)^{-1})^{g^{-1}}\circ (\iota_g)^{-1}\\
       &=gh.f
\end{align*}
Note that we had to define the action on $f$ using the shift $f^{g^{-1}}$ by the inverse of $g$ in order to obtain a left action of the group $G$.
The action of the neutral element $e \in G$ can be identified with the identity, see \cite[Remark 3.6]{Elagin}  for details.
\end{proof}
\begin{example}\label{example::G-invariant}
\
\begin{enumerate}

\item

The object $\Lambda$ in $\mathcal{D}^b(\mod \Lambda)$ is $G$-invariant, with isomorphisms $\iota_g:\Lambda_{g^{-1}}\to \Lambda$  given by $\lambda\mapsto g(\lambda)$. By Lemma \ref{G-End}, the group $G$ acts on $\End_{\mathcal{D}^b(\Lambda)}(\Lambda)$
and it is easy to see that the isomorphism $\End_{\mathcal{D}^b(\Lambda)}(\Lambda)\simeq \Lambda$ is a $G$-isomorphism.
\medskip
 
 \item For any $X\in \mod \Lambda$ and $\chi \in \dual $ the map $x\smallten g\mapsto \chi(g) x\smallten g$ induces an isomorphism in $\mod \Lambda G$
 $$\iota_\chi: (X\ten{\Lambda} \Lambda G )^{\chi^{-1}}  \overset{\sim}{\longrightarrow} X\ten{\Lambda} \Lambda G,$$ 
 which turns $X\ten{\Lambda} \Lambda G$ into a $ \dual$-invariant $\Lambda G$-module.
 Similarly, any object in $\mathcal{D}^b(\mod \Lambda G)$ of the form $X\lten{\Lambda} \Lambda G$ is $\dual$-invariant.
 \medskip
 
 \item Conversely, any object $X$ in $\mod \Lambda G$  is $G$-invariant when viewed as a $\Lambda$-module. Indeed, let us define
\[
 \iota_g: X_\Lambda^{g^{-1}}  \overset{\sim}{\longrightarrow} X_\Lambda,
 \quad
x \longmapsto x . ( 1 \otimes g^{-1})
\]
 Then $ \iota_g$ is a morphism of $\Lambda$-modules
 \begin{align*}
\iota_g(x.\lambda) = \iota_g(xg^{-1}(\lambda))
&= xg^{-1}(\lambda). ( 1 \otimes g^{-1})\\
       &= x.(1\otimes g^{-1}).(\lambda \otimes 1)\\
       &=\iota_g(x).\lambda
\end{align*}
 and one verifies that
 $$  \iota_{gh}=\iota_g\circ (\iota_{h})^{g^{-1}}$$
 holds for all $g,h \in G.$
 \end{enumerate}
\end{example}
\medskip

\begin{remark}
Since $(\Lambda_g)^{g'}=\Lambda_{gg'}$, the object $\Lambda G$ viewed as a $\Lambda$-module admits a realization as $G$-invariant object which is different from the one given in example \ref{example::G-invariant}(3), namely with the isomorphisms $\iota_g$ given by the permutation of the summands of $\Lambda G$. This induces by Lemma \ref{G-End} an action of $G$ on $\End_{\Lambda}(\Lambda G)$. 
Like the dual group $\dual$ acts on $\Lambda G$, the double dual group $G=\doublehat{G}$ acts on $\Lambda G\dual$, and the isomorphism $\Lambda G \dual \cong \End_{\Lambda}(\Lambda G)$ 
described in Proposition \ref{prop::isomorphismRR} is a $G$-isomorphism. However, if we denote by $e:\Lambda G\to \Lambda G$ the projection to the component $\Lambda \cong \Lambda\otimes 1_G$ of $\Lambda G$, this idempotent of $\End_{\Lambda}(\Lambda G)$ is not stable under the action of $G$. It is not clear in general how to construct an idempotent $e$ of $\Lambda G \dual$ together with a $G$-isomorphism $e\Lambda G\dual e\to \Lambda$.
\end{remark}

We recall from \cite[Lemma 2.3.1]{LeMeur} that the triangle functors 
$$\xymatrix{\mathcal{D}^b(\Lambda) \ar@<.5ex>[rr]^{-\lten{\Lambda}\Lambda G}& & \mathcal{D}^b(\Lambda G)\ar@<.5ex>[ll]^{Res}}$$
form adjoint pairs in both directions, and the unit of adjunction splits.
In particular, we have for all $X \in \mathcal{D}^b(\Lambda)$  a functorial isomorphism 
\begin{equation}\label{eq::split adjunction}
Res(X \lten{\Lambda}\Lambda G) \cong \oplus_{g \in G}X^g. 
\end{equation}
It is shown in Proposition 5.2.3 of \cite{LeMeur} that the skew-group ring $\End_{\mathcal{D}^b(\Lambda)}(X)G$ of a $G$-invariant object $X$ is Morita equivalent to the endomorphism ring of $X\lten{\Lambda} \Lambda G$. We show that they are actually isomorphic:

\begin{proposition}\label{prop::G-hat-iso}
Let $\Lambda$ be a $G$-algebra, and $X\in \mathcal{D}^b(\Lambda)$ be a $G$-invariant object. Then we have a $\dual$-isomorphism $$\End_{\mathcal{D}^b(\Lambda)}(X)G\cong \End_{\mathcal{D}^b(\Lambda G)}(X\lten{\Lambda} \Lambda G).$$
\end{proposition}

\begin{proof} 

Left multiplication with $1 \otimes g$ yields an isomorphism of $\Lambda-\Lambda G$ bimodules $ \Lambda G \to  {_g\Lambda G}$. This induces an isomorphism which is functorial in $X\in \mathcal{D}^b(\Lambda)$ 
\[L^X_g : X \lten{\Lambda} \Lambda G \to X^{g^{-1}} \lten{\Lambda} \Lambda G \]
 since we have $\Lambda\otimes _{g}\Lambda G=\Lambda_{g^{-1}}\otimes \Lambda G$.
Then one easily checks that 
\begin{align}\label{lambdarules}
L_{gh}^X=L_g^{X^{h^{-1}}}\circ L^X_h \mbox{ and } L_g^Y\circ (u\otimes 1)=(u^{g^{-1}}\otimes 1)\circ L^X_g
 \end{align} for any $g,h\in G$ and any $u\in \Hom_{\mathcal{D}^b(\Lambda)}(X, Y)$.
 
Now let $X$ be a $G$-invariant object in $\mathcal{D}^b(\Lambda)$ and $i_g$ the corresponding isomorphism. Define a map
$$ \phi: \End_{\mathcal{D}^b(\Lambda)}(X)G\to \End_{\mathcal{D}^b(\Lambda G)}(X\lten{\Lambda} \Lambda G)$$
by
\[ \phi(u \otimes g) = 
((u \circ i_g) \otimes 1_{\Lambda G}) \circ L^X_{g}.
\]
We first verify that $\phi$ is a morphism of algebras: Using the properties for  $i_g$ and (\ref{lambdarules}), one sees that  the product
$$ (u \otimes g)\cdot (v \otimes h) = u\circ i_g \circ v^{g^{-1}} \circ i_g^{-1} \otimes gh $$
is mapped to 
\begin{align*} (u\circ i_g\circ v^{g^{-1}}\circ i_g^{-1}\circ i_{gh}\otimes 1_{\Lambda G})\circ L^X_{gh} & = (u\circ i_g \otimes 1)\circ ((v\circ i_h)^{g^{-1}}\otimes 1)\circ L^{X^{h^{-1}}_g}\circ L^X_h\\
 & = (u\circ i_g \otimes 1)\circ L^X_g\circ (v\circ i_h\otimes 1)\circ L^X_h \\ &= \phi(u \otimes g) \circ \phi(v \otimes h)
\end{align*}

Next, the adjunction formula and equation (\ref{eq::split adjunction}) yields isomorphisms of vector spaces
 \begin{align*}
 \End_{\mathcal{D}^b(\Lambda G)}(X \otimes \Lambda G) &\cong
  \Hom_{\mathcal{D}^b(\Lambda)}(X, \Hom_{\mathcal{D}^b(\Lambda G)}(\Lambda G, X \otimes \Lambda G))
  \\ &\cong \Hom_{\mathcal{D}^b(\Lambda)}(X,  X \otimes \Lambda G)\\
  &\cong \Hom_{\mathcal{D}^b(\Lambda)}(X,  \bigoplus_{g \in G} X^g )\\
  &\cong  \End_{\mathcal{D}^b(\Lambda)}(X) G
 \end{align*}
under which the element 
$(u  \circ i_g \otimes 1) \circ L^X_{g}$ is sent to the element $u \otimes g$. Therefore $\phi$ is an isomorphism, which can be verified to be compatible with the action of $\dual$.
\end{proof}

\subsection{Invariant tilting objects in $\mathcal{D}^b(\Lambda)$ and in $\mathcal{D}^b(\Lambda G)$}
We study now tilting objects in a derived category with a group action. Let us recall that  an object $T$ of $D^b(\Lambda)$ is \emph{tilting} if
   thick$(T)= D^b(\Lambda)$  and $T$ is {\em rigid}, that is, $\Hom_{D^b(\Lambda)}(T, T[i]) =0$ for any integer~$i \neq0$.

\begin{definition} For $\Lambda$ a $G$-algebra, and $T$ a $G$-invariant tilting object of $\mathcal{D}^b(\Lambda)$, the category ${\rm add}(T)$ will be called a \emph{$G$-tilting subcategory of }$\mathcal{D}^b(\Lambda)$.
\end{definition}
 
 The following result has been partially shown in \cite[Corollary 5.2.2]{LeMeur} in the context of cluster-tilting subcategories.
Note however that we consider $\dual$-invariance instead of invariance under a composition of functors.

\begin{theorem}\label{thm::bijection-tilting}
Let $\Lambda$ be a $G$-algebra. Then the functors 
$$\xymatrix{\mathcal{D}^b(\Lambda) \ar@<.5ex>[rr]^{-\lten{\Lambda}\Lambda G}& & \mathcal{D}^b(\Lambda G)\ar@<.5ex>[ll]^{Res}}$$
 induce a bijection $$\{G\textrm{-tilting subcategories of }\mathcal{D}^b(\Lambda)\} \leftrightarrow \{ \dual\textrm{-tilting subcategories of }\mathcal{D}^b(\Lambda G)\}.$$
\end{theorem}

For the proof we need the following lemma:
\begin{lemma}\label{lemma:G-bimodule}
There is an isomorphism of $\Lambda G$-bimodules $$\Lambda G \dual \simeq \Lambda G\ten{\Lambda}\Lambda G.$$
\end{lemma}

\begin{proof}
We construct two isomorphisms of $\Lambda G$-bimodules $$ \xymatrix{\Lambda G \dual \ar[rr]^{\Phi_1} && \End_{\Lambda}(\Lambda G) && \Lambda G\ten{\Lambda} \Lambda G\ar[ll]_{\Phi_2}}.$$
The map $\Phi_1$ is the one given in \eqref{iso}. This is an isomorphism, and clearly a left $\Lambda G$-module map. So it remains to show that it is a morphism of right $\Lambda G$-modules.

The right $\Lambda G$-module structure of $\Lambda G \dual$ is induced by the embedding $\Lambda G\to \Lambda G \dual$, while the right $\Lambda G$-module structure of $\End_{\Lambda}(\Lambda G)$ comes from the left $\Lambda G$-module structure of $\Lambda G$. A direct computation yields
$$\Phi_1((\lambda\smallten g\smallten\chi).(\lambda'\smallten g'\smallten 1_{\dual}))(\mu\smallten h)=\Phi_1(\lambda\smallten g\smallten \chi)((\lambda'\smallten g').(\mu\smallten h)),$$
thus $\Phi_1$ is an isomorphism of $\Lambda G$-bimodules.

The left $\Lambda$-module $\Lambda G\ten{\Lambda}\Lambda G$ is a free module with basis given by the elements $(1\smallten g_1)\otimes (1\smallten g_2)$, $g_1,g_2\in G$. We define $\Phi_2$ on this basis and extend it by left $\Lambda$-linearity:
We set  $\Phi_2((1\smallten g_1)\otimes (1\smallten g_2))$ to be the map $$\varphi_{g_1,g_2}:
\mu\smallten h\mapsto (1\smallten g_1).\delta_{g_2,h^{-1}} g_2(\mu)\smallten 1_G,$$
where $\delta_{i,j}$ is the Kronecker symbol.

First, a direct computation gives 
$$\varphi_{g_1,g_2}((\mu\smallten h).(\lambda\smallten 1_{G}))=(\varphi_{g_1,g_2} (\mu \smallten h)).(\lambda\smallten 1_{G}),$$
so $\varphi_{g_1,g_2}$ is indeed a map of right $\Lambda$-modules.

Next, note that the elements $\varphi_{g_1,g_2}$ form a $\Lambda$-basis of $\End_{\Lambda}(\Lambda G)$, so $\Phi_2$ is an isomorphism of left $\Lambda$-modules. Moreover $\Phi_2$ is clearly a left $\Lambda G$-morphism. Finally by a direct computation we get that 
$$\Phi_2((1\smallten g_1)\otimes ((1\smallten g_2).(\lambda\smallten g))(\mu\smallten h)= \varphi_{g_1,g_2}((\lambda\smallten g).(\mu\smallten h)),$$
hence $\Phi_2$ is a right $\Lambda G$-module morphism.
\end{proof}

\noindent{\it Proof of Theorem \ref{thm::bijection-tilting}.}
Let $T\in \mathcal{D}^b(\Lambda)$ be a $G$-tilting object. Then $T\lten{\Lambda} \Lambda G$ is $\dual$-invariant by example \ref{example::G-invariant}(2).
As in the proof of proposition \ref{prop::G-hat-iso}, one sees that 
\[
 \Hom_{\mathcal{D}^b(\Lambda G)}(T \lten{\Lambda} \Lambda G, T \lten{\Lambda} \Lambda G [i]) \cong
\Hom_{\mathcal{D}^b(\Lambda)}(T,  \bigoplus_{g \in G} T^g [i])
 \]
and therefore the object $T\lten{\Lambda} \Lambda G$ is rigid since $T$ is so.
To show that thick$(T \lten{\Lambda} \Lambda G) = \mathcal{D}^b(\Lambda  G)$, consider an object 
$X \in \mathcal{D}^b(\Lambda  G).$
Since $T$ is tilting, we have
\[ X_\Lambda \in \mathcal{D}^b\Lambda  = \thick (T),
\]
hence $X \lten{\Lambda} \Lambda G \in \thick(T \lten\Lambda  G).$
Now we use the fact that $\Lambda G$ is projective as $\Lambda$-module, and lemma \ref{lemma:G-bimodule} to obtain
 \begin{align*}
 X \lten{\Lambda} \Lambda G =
 (X \lten{\Lambda G} \Lambda G)\lten{\Lambda} \Lambda G  &\cong
 X \lten{\Lambda G}(\Lambda G \lten{\Lambda} \Lambda G)\\
   &\cong 
   X \lten{\Lambda G}\Lambda G \dual\\
    &\cong 
  \bigoplus_{\chi \in \dual} X^\chi
 \end{align*}
 Since a thick subcategory is closed under direct factors, we conclude $X \in \thick(T \lten{\Lambda} \Lambda G).$
 \medskip
 
 Conversely, let $U\in \mathcal{D}^b(\Lambda G)$ be a $\dual$-tilting object. Then $U_\Lambda$ is $G$-invariant by example \ref{example::G-invariant}(3).
To show that $U$ is rigid, we verify
 \begin{align*}
 \Hom_{\mathcal{D}^b\Lambda}(U_\Lambda, U_\Lambda  [i]) &\cong
\Hom_{\mathcal{D}^b(\Lambda G)}
(U_\Lambda\lten{\Lambda}\Lambda G, U[i])\\
& \cong 
\Hom_{\mathcal{D}^b(\Lambda G)}
(U_\Lambda\lten{\Lambda G}(\Lambda G \otimes_\Lambda \Lambda G), U[i])\\
& \cong 
\Hom_{\mathcal{D}^b(\Lambda G)}(\bigoplus_{\chi \in \dual}U^\chi, U [i]) =0.
 \end{align*}
Consider an object $X \in \mathcal{D}^b\Lambda$. Then 
$(X\lten{\Lambda}\Lambda G)_\Lambda \in \thick(U_\Lambda)$ since $X\lten{\Lambda}\Lambda G \in \mathcal{D}^b\Lambda G = \thick(U).$
As before, this implies $X \in \thick(U_\Lambda)$ since $X$ is a direct factor of 
$(X\lten{\Lambda}\Lambda G)_\Lambda \cong \bigoplus_{g \in G}X^g.$
\medskip

We have so far verified that the functors $-\lten{\Lambda}\Lambda G$ and $Res$ induce maps 
\[ \add(T) \longmapsto \add(T\lten{\Lambda}\Lambda G)
\]
and 
\[ \add(U) \longmapsto \add(U_{\Lambda})
\]
between $G$-tilting subcategories of $\mathcal{D}^b\Lambda$ and $\dual$-tilting subcategories of $\mathcal{D}^b\Lambda G.$
To verify that these maps are inverse to each other, observe that
\[  \add((T\lten{\Lambda}\Lambda G)_\Lambda) = \add(\bigoplus_{g \in G}T^g) =   \add(T)
\]
since $T$ is $G$-invariant.
Likewise,
\[  \add(U_\Lambda\lten{\Lambda}\Lambda G) = \add(\bigoplus_{\chi \in \dual}U^\chi) =   \add(U).
\]
\hfill$\qed$\bigskip

Note that as a consequence of this result, if $T\in \mathcal{D}^b(\Lambda)$ is a $G$-invariant tilting object, and if $\Lambda'=\End_{\mathcal{D}^b(\Lambda)}(T)$ is the corresponding $G$-algebra, then we have the following commutative diagram 

$$\xymatrix{\mathcal{D}^b(\Lambda)\ar[dd]^{-\lten{\Lambda} \Lambda G} &&& \mathcal{D}^b(\Lambda')\ar[lll]^{-\lten{\Lambda'} T}\ar[dd]^{-\lten{\Lambda'} \Lambda' G} \\ &&&\\
\mathcal{D}^b(\Lambda G)&&& \mathcal{D}^b(\Lambda'G)\ar[lll]^{-\lten{\Lambda' G} (T\lten{\Lambda}\Lambda G)} 
}$$ where both horizontal maps are equivalences. This leads us to state the following:

\begin{definition} Two $G$-algebras $\Lambda$ and $\Lambda'$ are called $G$-derived equivalent if there exists a $G$-invariant tilting object $T\in \mathcal{D}^b(\Lambda)$ together with a $G$-isomorphism $\End_{\mathcal{D}^b(\Lambda)}(T)\simeq \Lambda'$.
We denote it by $\mathcal{D}^b(\Lambda)\underset{G}{\sim}\mathcal{D}^b(\Lambda')$.
\end{definition}

Therefore we have 

\begin{corollary}\label{cor::Geq=G-hat-eq}
Let $\Lambda$ and $\Lambda'$ be $G$-algebras, then we have 
$$\mathcal{D}^b(\Lambda)\underset{G}{\sim}\mathcal{D}^b(\Lambda ') \Rightarrow \mathcal{D}^b(\Lambda G )\underset{\dual}{\sim}\mathcal{D}^b(\Lambda 'G).$$
If moreover there exists a $G$-invariant idempotent $\theta$ of $\Lambda G\dual$ and $\theta'$ of $\Lambda' G \dual$ together with $G$-isomorphisms $\Lambda \simeq \theta \Lambda G \dual \theta$ and $\Lambda' \simeq \theta' \Lambda' G \dual \theta'$, then we have 

$$\mathcal{D}^b(\Lambda)\underset{G}{\sim}\mathcal{D}^b(\Lambda ') \Leftrightarrow \mathcal{D}^b(\Lambda G )\underset{\dual}{\sim}\mathcal{D}^b(\Lambda 'G).$$

\end{corollary}

\section{Skew-gentle algebras and dissections}

\subsection{Skew-gentle algebras}
We first recall from \cite{GePe} the concept of skew-gentle algebras and then study some of their basic properties.
\begin{definition}
A {\em gentle pair} is a pair $(Q,I)$ given by a quiver $Q$ and  a subset $I$ of paths of length two in $Q$ such that 
\begin{itemize}
\item for each $i\in Q_0$, there are at most two arrows with source $i$, and at most two arrows with target $i$;
\item for each arrow $\alpha:i\to j$ in $Q_1$, there exists at most one arrow $\beta$ with target $i$ such that $\beta\alpha\in I$ and at most one arrow $\beta'$ with target $i$ such that $\beta'\alpha\notin I$;

\item for each arrow $\alpha:i\to j$ in $Q_1$, there exists at most one arrow $\beta$ with source $j$ such that $\alpha\beta\in I$ and at most one arrow $\beta'$ with source $j$ such that $\alpha\beta'\notin I$.
\item the algebra $A(Q,I):=kQ/I$ is finite dimensional.
\end{itemize}
An algebra is {\em gentle} if it admits a presentation $A=kQ/I$ where $(Q,I)$ is a gentle pair.
\end{definition}

We follow \cite{BeHo} stating the  definition which appeared first in \cite{GePe}:
\begin{definition}
A \emph{skew-gentle triple} $(Q,I,{\rm Sp})$ is the data of  a quiver $Q$,  a subset $I$ of paths of length two in $Q$, and  a subset ${\rm Sp}$ of loops in $Q$ (called 'special loops') such that $(Q, I\amalg  \{e^2, e\in {\rm Sp})$ is a gentle pair. 
In this case, the algebra $\bar{A}(Q,I,{\rm Sp}):= kQ/\langle I\amalg \{e^2-e,e\in {\rm Sp}\rangle,$ is  called a  {\em skew-gentle algebra.}
Note that as a gentle algebra is finite dimensional, so is a skew-gentle algebra.
\end{definition}

Skew-gentle algebras are known to be tame algebras, and a classification of their indecomposable modules is given in \cite{CB,De} using the notion of a certain matrix problem called clan, hence they use the name clannish algebra.
Skew-gentle algebras can also be related to clannish matrix problems by gluing them together from smaller pieces as in \cite{Br}, we present this method here to obtain another description of the class of skew-gentle algebras: 

First recall from \cite[Prop 5.2]{Br} that gentle algebras can be obtained by gluing together the following puzzle pieces $S_n$ and $\widetilde{S}_n$:

\begin{itemize}
    
    \item[(a)] $S_n$ denotes, for $n \ge 1$, the linearly oriented quiver of type $A_n$ with radical square zero:
     \[ \xymatrix{
x_1 \; \ar[r]^{\alpha_1}  & x_2 \ar[r]^{\alpha_2} & \quad \cdots & x_{n-1} \ar[r]^{\alpha_{n-1}} & x_n}\]
with $\alpha_i \alpha_{i+1}=0$ for $1 \le i \le n-2. $\medskip
    \item[(b)] $\widetilde{S}_n$ denotes, for $n \ge 1$, the cyclically oriented quiver of type $\widetilde{A}_n$ with radical square zero:
     \[ \xymatrix{
x_1 \; \ar[r]^{\alpha_1}  & x_2 \ar[r]^{\alpha_2} & \quad \cdots & x_{n-1} \ar[r]^{\alpha_{n-1}} & x_n=x_1}\]
with $\alpha_i \alpha_{i+1}=0$ for $1 \le i \le n-1 \mod n. $
    
\end{itemize}
A gluing is obtained by choosing a (not necessarily perfect) matching of the vertices of a collection of puzzle pieces of type (a) or (b), and identifying the pairs of vertices related by the matching. The resulting algebra is gentle if it is finite-dimensional, and every gentle algebra is obtained in this way. 

We produce skew-gentle algebras by allowing one additional puzzle piece, a special loop $sp$:

\begin{itemize}
    
    \item[(c)] $sp$ denotes the quiver with one vertex and one loop $e$ with relation $e^2-e=0$ :
     \[ \xymatrix{ x \ar@(ur,dr)^{e}[]
}
    \]
\end{itemize}
\bigskip

\begin{lemma}\label{lem::gluing}
A gluing of puzzle pieces of type (a), (b) or (c)  yields a  skew-gentle algebra if it is finite-dimensional, and every skew-gentle algebra is obtained in this way.
\end{lemma}
\begin{proof}
Replacing all special loops in a skew-gentle algebra by loops $e$ with $e^2=0$ one obtains a gentle algebra, which is glued from puzzle pieces (a) and (b) by \cite[Prop 5.2]{Br}. Note that the condition of the algebra being finite dimensional requires that every loop $e$ in the gentle case satisfies $e^2=0$. The special loops are then obtained from gluing pieces of type (c) instead of loops $e$ with $e^2=0$.
\end{proof}

The proof of the previous lemma used the fact that replacing all special loops in a skew-gentle algebra by loops $e$ with $e^2=0$ one obtains a gentle algebra. More generally, given a skew-gentle algebra 
$\bar{A}=\bar{A}(Q,I,{\rm Sp}),$ let us define for every $t \in k$ the algebra
\[\bar{A}_t := kQ/\langle I\amalg \{e^2-te,e\in {\rm Sp}\rangle.\]
Then $\bar{A}_0$ is the gentle algebra used in the proof of Lemma \ref{lem::gluing}, and $\bar{A}_1$ is the original skew-gentle algebra. 

\begin{lemma}\label{lem::degeneration}
Let $k$ be an algebraically closed field. Any skew-gentle algebra $\bar{A}(Q,I,{\rm Sp})$ is a deformation of the corresponding gentle algebra $\bar{A}_0$.
\end{lemma}
\begin{proof}
We assume the field to be algebraically closed so we can speak about deformation of  $k$-algebras along one-parameter families in the sense of \cite{Ge,CB2}. It is sufficient to show that for all $t \neq 0$ the algebras  $\bar{A}_t$ are isomorphic to $\bar{A}_1$, since $\bar{A}_0$ lies then in the closure of this family of isomorphic algebras.
Define, for all $t \in k$, an algebra morphism $\phi_t : \bar{A}_t \to \bar{A}_1$
by sending $e \mapsto te$ if $e$ is a special loop, and $a \mapsto a$ for the remaining arrows.
This transforms the relation $e^2-te =0$ in $\bar{A}_t$ into $t^2(e^2-e) =0$, thus $\phi_t$ is indeed well-defined. It admits, for all $t \neq 0$ an inverse defined by sending $e \mapsto \frac{e}{t}$.
\end{proof}

Note that the theorem of Geiss \cite{Ge} implies that skew-gentle algebras are tame since they degenerate to a gentle algebra. However, degeneration does not provide precise information about indecomposable modules, so we use Lemma \ref{lem::degeneration} more to compare different geometric models.
A similar deformation argument to the one in Lemma \ref{lem::degeneration} has been used in \cite{BPS,GL-FS} to show that similar classes of algebras are tame.

\subsection{The quiver of a skew-gentle algebra}\label{subsection::quiver}

Note that every gentle algebra is skew-gentle (with empty set of special loops). In this case, the quiver $Q$ is the quiver $Q_{\bar{A}}$ defined by the algebra $\bar{A}$. This is not the case when we have special loops, since the relation $e^2-e$ is not admissible. In fact, the idempotent $e$ attached to vertex $i$ of $Q$ splits the vertex into two so that the quiver $Q_{\bar{A}}$ of the algebra $\bar{A}$ has two vertices for every vertex of $Q$ with a special loop. The arrows are split accordingly, hence the quiver of a skew-gentle algebras is described as follows:

Consider the skew-gentle algebra $\bar{A}=\bar{A}(Q,I,{\rm Sp}).$ We divide the vertex set $Q_0$ of the quiver $Q$ into two disjoint sets:
Denote by $Q_0^{sp}$ the set of 'special' vertices of $Q$ where a special loop is attached, and let $Q_0^{ord}$ be the remaining 'ordinary' vertices.
Then the quiver $\bar{Q}$ of the algebra $\bar{A}$ is given as follows:
\begin{itemize}
\item
The vertices of $\bar{Q}$ are bijection with 
$$Q_0^{ord} \cup (Q_0^{sp} \times \mathbb Z_2).$$
\item
Given two ordinary vertices $i$ and $j$ in $Q_0^{ord}$, then arrows in $\bar{Q}$ between $i$ and $j$ are bijection with 
 the arrows in $Q$ between $i$ and $j$;
\item Given an ordinary vertex $i$ and a special vertex $j \in Q_0^{sp}$, there are  two arrows 
 \[
  \scalebox{0.8}
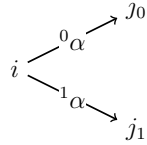
{
  \begin{tikzpicture}
  \node (I) at (0,0) {$i$};
  \node (J) at (2,1) {$j_0$};
  \node (J1) at (2,-1) {$j_1$};
  \draw[thick, ->] (I)--node[fill=white, inner sep=0pt]{$^0\alpha$}(J);
   \draw[thick, ->] (I)--node[fill=white, inner sep=0pt]{$^1\alpha$}(J1);
  \end{tikzpicture}}\] in $\bar{Q}$ for every arrow 
   \[
  \scalebox{0.8}{
  \begin{tikzpicture}
  \node (I) at (0,0) {$i$};
  \node (J) at (2,0) {$j$};
  \draw[thick, ->] (I)--node[fill=white, inner sep=0pt]{$\alpha$}(J);
  \end{tikzpicture}}\]
  in $Q$.
  \medskip
\item dually, every arrow 
   \[
  \scalebox{0.8}{
  \begin{tikzpicture}
  \node (I) at (0,0) {$j$};
  \node (J) at (2,0) {$i$};
  \draw[thick, ->] (I)--node[fill=white, inner sep=0pt]{$\alpha$}(J);
  \end{tikzpicture}}\]
  in $Q$ with $i$ ordinary and $j$ special yields two arrows
 \[
  \scalebox{0.8}
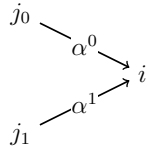
{
  \begin{tikzpicture}
  
  \begin{scope}[xshift=5cm]
   \node (I) at (2,0) {$i$};
  \node (J) at (0,1) {$j_0$};
  \node (J1) at (0,-1) {$j_1$};
  \draw[thick, ->] (J)--node[fill=white, inner sep=0pt]{$\alpha^0$}(I);
   \draw[thick, ->] (J1)--node[fill=white, inner sep=0pt]{$\alpha^1$}(I);
   
  \end{scope}
  \end{tikzpicture}}\] in $\bar{Q}$;
  \medskip
\item for every arrow 
   \[ \xymatrix{ i 
 \ar[r]^{\alpha}  & j}\]
  in $Q$ with both $i$ and $j$ special, there are four arrows 
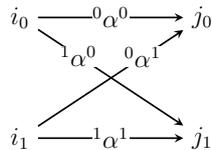
\[
  \scalebox{0.8}{
  \begin{tikzpicture}[scale=1, >=stealth]
  \node (I) at (0,1) {$i_0$};
  \node (I1) at (0,-1) {$i_1$};
  \node (J) at (3,1) {$j_0$};
  \node (J1) at (3,-1) {$j_1$};
  \draw[thick, ->] (I)--node[fill=white, inner sep=0pt]{$^0\alpha ^0$}(J);
   \draw[thick, ->] (I1)--node[fill=white, inner sep=0pt, xshift=15pt, yshift=10pt]{$^0\alpha ^1$}(J);
    \draw[thick, ->] (I)--node[fill=white, inner sep=0pt, xshift=-15pt, yshift=10pt]{$^1\alpha ^0$}(J1);
   \draw[thick, ->] (I1)--node[fill=white, inner sep=0pt]{$^1\alpha ^1$}(J1);
   \end{tikzpicture}}\]
in $\bar{Q}$.
\end{itemize}
\bigskip

The relations generating the ideal $\bar{I}$ of the algebra $\bar{A}=k \bar{Q}/\bar{I}$ can be described as follows: Consider a relation $\beta\alpha$  in $I$ 
 \[ \xymatrix{
 \ar[r]^{\alpha}  & i \ar[r]^{\beta} &  }.\]
If  $i$ is an ordinary vertex, then we have  $^{\epsilon}\beta\alpha^{\epsilon'}\in \bar{I}$, for each $\epsilon=0,1,\emptyset$ and $\epsilon'=0,1,\emptyset$, where the expression makes sense. 
When $i$ is a special vertex, then we have $(^{\epsilon}\beta^0) (^0\alpha^{\epsilon'})+(^{\epsilon}\beta^1)( ^1\alpha^{\epsilon'})\in \bar{I}$, for all possible $\epsilon=0,1,\emptyset$ and $\epsilon'=0,1,\emptyset$.
\bigskip

\begin{example}\label{example::garland}
Consider the skew-gentle algebra $\bar{A}=\bar{A}(Q,I,{\rm Sp})$ obtained by gluing a piece $S_5$ (which is of type (a)) with three special loops in the middle vertices, thus the quiver $Q$ is given by
 \[ \xymatrix{ 1 \ar[r]^{\alpha} & 2 \ar@(ul,ur)^{e}[] \ar[r]^{\beta} & 3
 \ar@(ul,ur)^{f}[] \ar[r]^{\gamma} & 4
 \ar@(ul,ur)^{g}[] \ar[r]^{\delta} & 5
}
    \]
with relations $\alpha \beta = \beta \gamma = \gamma \delta = 0$ and special loops Sp $=\{e,f,g \}$.
Then the quiver $\bar{Q}$ of the algebra $\bar{A}=k \bar{Q}/\bar{I}$ is a garland where all squares are anti-commutative:
\[
  \scalebox{0.8}{
  \begin{tikzpicture}
  \node (I) at (0,0) {$1$};
  \node (J) at (2,1) {$2_0$};
  \node (J1) at (2,-1) {$2_1$};
  \draw[thick, ->] (I)--node[fill=white, inner sep=0pt]{$^0\alpha$}(J);
   \draw[thick, ->] (I)--node[fill=white, inner sep=0pt]{$^1\alpha$}(J1);
   
  \node (K) at (4,1) {$3_0$};
  \node (K1) at (4,-1) {$3_1$};
  \draw[thick, ->] (J)--node[fill=white, inner sep=0pt]{$^0\beta ^0$}(K);
   \draw[thick, ->] (J1)--node[fill=white, inner sep=0pt, xshift=15pt, yshift=10pt]{$^0\beta ^1$}(K);
    \draw[thick, ->] (J)--node[fill=white, inner sep=0pt, xshift=-15pt, yshift=10pt]{$^1\beta ^0$}(K1);
   \draw[thick, ->] (J1)--node[fill=white, inner sep=0pt]{$^1\beta ^1$}(K1);
   
   \node (H) at (6,1) {$4_0$};
  \node (H1) at (6,-1) {$4_1$};
  \draw[thick, ->] (K)--node[fill=white, inner sep=0pt]{$^0\gamma ^0$}(H);
   \draw[thick, ->] (K1)--node[fill=white, inner sep=0pt, xshift=15pt, yshift=10pt]{$^0\gamma ^1$}(H);
    \draw[thick, ->] (K)--node[fill=white, inner sep=0pt, xshift=-15pt, yshift=10pt]{$^1\gamma ^0$}(H1);
   \draw[thick, ->] (K1)--node[fill=white, inner sep=0pt]{$^1\gamma ^1$}(H1);
   
   \node (L) at (8,0) {$5$};
  \draw[thick, ->] (H)--node[fill=white, inner sep=1pt]{$\delta^0$}(L);
   \draw[thick, ->] (H1)--node[fill=white, inner sep=1pt]{$\delta^1$}(L);
  \end{tikzpicture}}
  \]
\end{example}

It is clear from the description of quiver and relations that a skew-gentle algebra admits a ${\mathbb Z}_2$-action. This has been explored in \cite{GePe}, and we will come back to it in section \ref{section::Z2action} using a geometric description of skew-gentle algebras.
\bigskip

The fact that the quiver $Q$ defining  $\bar{A}=\bar{A}(Q,I,{\rm Sp}) $ is in general not the quiver  $\bar{Q}$ of the skew-gentle algebra  $\bar{A}$ creates some ambiguity of the data defining skew-gentle algebras: Let $Q$ be the following quiver with a special loop attached to vertex 2
 \[ \xymatrix{ 1 \ar[r]^a & 2 \ar@(ur,dr)^{e}[]
}
    \]
and consider also the quiver 
$$Q': 2^+ \longleftarrow 1 \longrightarrow 2^-$$ 
If both sets $I$ and $I'$ are empty, then the skew-gentle algebras 
$\bar{A}(Q,I,\{ e \})$ and $\bar{A}(Q',I', \emptyset)$ are isomorphic, but the quivers $Q$ and $Q'$ are not.
This example illustrates the fact that the quiver of Dynkin type $D_3$ (skew-gentle) is actually an equi-oriented quiver of type $A_3$ (which is gentle). 
We address in the following lemma the question when it is possible to express a skew-gentle algebra with non-empty set of special loops as a gentle algebra:

\begin{lemma}
Let $\Lambda$ be a connected gentle algebra, and assume $\Lambda$ can be expressed as a skew-gentle algebra $\Lambda \cong \bar{A}(Q,I,{\rm Sp})$ with ${\rm Sp} \neq \emptyset$.
Then $\bar{A}(Q,I,{\rm Sp})$  is one of the following cases or its dual:

 \[ \xymatrix{ 1 \ar@(ul,dl)_{e}[]\ar[r]^a & 2 
} \qquad \qquad
\xymatrix{ 1 \ar@(ul,dl)_{e}[] \ar[r]^a & 2 \ar@(ur,dr)^{f}[]
}
    \]
\end{lemma}
\medskip

\begin{proof}
We assume that $\Lambda$ can be presented as a skew-gentle algebra $\Lambda \cong \bar{A}(Q,I,{\rm Sp})$ with a special loop $e$ at vertex $y$.
If $y$ lies on a path
$x \to y \to z$ in $Q$, then the quiver $\bar{Q}$ of $\bar{A}(Q,I,{\rm Sp})$ contains an anti-commutative square, and thus $\Lambda$ is not gentle. 
Therefore there is exactly one arrow in $Q$ attached to the vertex $y$, and we can assume up to duality it is $a: y \to z.$ If there is a further arrow between $z$ and some different vertex $w$ in $Q$, then the quiver $\bar{Q}$ of $\Lambda$ contains a subquiver of type $D_n$ with $n \ge 4$, thus it is not gentle. Therefore, only the cases described in the lemma are possible.
\end{proof}

The gentle pair $(Q,I)$ of a gentle algebra is well defined up to isomorphism of gentle pairs.  We generalize the proof of this fact  and show that the same holds true for skew-gentle algebras and their triples, when avoiding the cases described in the previous lemma:

\begin{proposition}\label{prop::isoalgebras-isopairs}
Let $\Lambda$ be a connected skew-gentle algebra which is not gentle. Then $\Lambda \cong \bar{A}(Q,I,{\rm Sp})$ for a unique skew-gentle triple $(Q,I,{\rm Sp})$, up to an isomorphism of quivers $Q\to Q'$ sending $I$ to $I'$ and $\rm Sp$ to ${\rm Sp}'$. 
\end{proposition}

\begin{proof}
Let $\varphi :A=A(Q,I,{\rm Sp})\to A'=A(Q',I',{\rm Sp}')$ be an isomorphism. Using $\varphi$ we construct an isomorphism $\bar{Q}\to \bar{Q}'$ sending $\bar{I}$ to $\bar{I'}$.
First $\varphi$ sends any $e_i$ with $i\in \bar{Q}_0$ to $e_{i'}$ with $i'\in \bar{Q}'_0$, so induces an isomorphism $\bar{Q}_0\to \bar{Q}'_0$.
Denote by $\mathbf{r}$ (resp. $\mathbf{r}'$ ) the radical of $A$ (resp. $A'$). Then $\varphi$ induces vector space isomorphisms $$\varphi_{i,j}^n: e_j ({\mathbf r}^n/{\mathbf r}^{n+1})e_i\to e_{j'} ({\mathbf r}'^n/{\mathbf r}'^{n+1})e_{i'} $$ which are compatible with the multiplication.
We first show that unless $\bar{Q}$ is  the Kronecker quiver (which is gentle),  $\varphi^1$ sends any arrow to a multiple of an arrow of $\bar{Q}'$.
If $Q$ has no double arrows, this is clear. Now since $A$ is finite dimensional, there are no oriented cycles of double arrows in $\bar{Q}$. If $\alpha$ and $\beta$ are parallel arrows then since the quiver is not the Kronecker quiver, there exists $\gamma$ that composes with $\alpha$ or $\beta$. If $\varphi^1(\gamma)$ is a multiple of $\gamma$, then one can check that so are $\varphi^1(\alpha)$ and $\varphi^1(\beta)$. Using this argument, one can show by induction that any arrow is sent to a multiple of an arrow by $\varphi^1$.

Now  $\varphi^1$ induces a isomorphism $\bar{Q}$ to $\bar{Q}'$, we denote by $a'\in \bar{Q}'_1$ the image of $a\in \bar{Q}_1$. Let us check that it sends $\bar{I}$ to $\bar{I}'$. 
If $\alpha\beta\in I$, then $\varphi^1(\alpha)\varphi^1(\beta)$ is in $\bar{I'}$. Since it is a multiple $\alpha'\beta'$, we have $\alpha'\beta'\in \bar{I}'$.
If $\alpha_0\beta_0+\alpha_1\beta_1$ is in $I$, then $\varphi^1(\alpha_0)\varphi^1(\beta_0)+ \varphi^1(\alpha_1)\varphi^1(\beta_1)$ is in $I'$ and is a linear combination of $\alpha'_0\beta'_0$ and $\alpha'_1\beta'_1$, therefore it must be a multiple of  $\alpha'_0\beta'_0+\alpha'_1\beta'_1$ .

We conclude using that the assignment $(Q,I,{\rm Sp})\mapsto (\bar{Q},\bar{I})$ is injective unless $(Q,I,{\rm Sp})$ is as described in the previous lemma.
In fact, the set of special loops is determined by $(\bar{Q},\bar{I})$ as follows: Every anti-commutative square in $(\bar{Q},\bar{I})$ is given by a path of length two in $Q$ passing through a vertex equipped with a special loop. 
Moreover, every subquiver of $\bar{Q}$ of the form $D_n$ with $n >4$ 
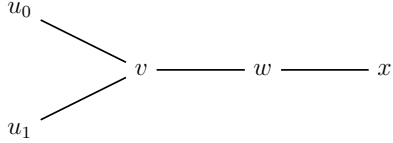
\[
  \scalebox{0.8}{
  \begin{tikzpicture}
   \node (I) at (2,0) {$v$};
   \node (K) at (4,0) {$w$};
  \node (L) at (6,0) {$x$};
  \node (J) at (0,1) {$u_0$};
  \node (J1) at (0,-1) {$u_1$};
  \draw[thick, -] (J)--node{}(I);
   \draw[thick, -] (J1)--node{}(I);
   \draw[thick, -] (I)--node{}(K);
   \draw[thick, -] (L)--node{}(K);
  \end{tikzpicture}}\]
 is necessarily obtained from $Q$ by splitting a vertex $u$ into two vertices $u_0,u_1$ by means of a special loop attached at $u$. 

\end{proof}

\subsection{Gentle algebras and dissected surfaces}

In this subsection, we recall some definitions and results from \cite{OPS} (but we mostly follow the notation in \cite{APS}).

A \emph{marked surface} $(\surf, M,P)$ is the data of  
\begin{itemize}
\item an orientable closed smooth surface $\surf$ with non empty boundary, that is a compact closed smooth surface from which some open discs are removed;
\item a finite set of marked points $M$ on the boundary, such that there is at least one marked point on each boundary component (this set corresponds to the set $M_\rpoint$ in \cite{APS});
\item a finite set $P$ of marked points in the interior of $\surf$ (which corresponds to the set $P_\rpoint$ in \cite{APS}).
\end{itemize}
The points in $M$ and $P$ are called marked points. A curve on the boundary of $\surf$ intersecting marked points only on its endpoints is called a \emph{boundary segment}.

An \emph{arc} on $(\surf, M, P)$ is a curve $\gamma:[0,1]\to \surf$ such that $\gamma_{|(0,1)}$ is injective and $\gamma(0)$ and $\gamma(1)$ are marked points. Each arc is considered up to isotopy (fixing endpoints).

\begin{definition}
A \emph{$\rpoint$-dissection} is a collection $D=\{\gamma_1,\ldots,\gamma_s\}$ of arcs cutting $\surf$ into polygons with exactly one side being a boundary segment.

Two dissected surfaces $(\surf,M,P,D)$ and $(\surf',M',P',D')$ are called diffeomorphic if there exists an orientation preserving diffeomorphism $\Phi:\surf\to \surf'$ such that $\Phi(M)=M'$, $\Phi(P)=P'$, and $\Phi(D)=D'$.
\end{definition} 
 Following \cite{OPS}, one can associate to the dissection $D$ a quiver $Q$, together with a subset of quadratic monomial relations $I$, such that the algebra $A(D) := A(Q,I)$ is a gentle algebra.
 In the next subsection, we explain this construction in detail and give illustrating examples in the more general context of skew-gentle algebras.

\begin{proposition}\cite{OPS}\label{prop::bijection-gentle-dissection}
The assignment $D \mapsto A(D)$ induces a bijection 
$$\left\{\begin{array}{cc}(\surf,M,P,D) \\ \mbox{\rm dissected surface}\end{array} \right\}_{\textstyle / \rm \; diffeo} \longleftrightarrow \quad
\left\{\begin{array}{cc}A(Q,I) \\
\mbox{\rm gentle algebra }\end{array}\right\}_{\textstyle /{\rm iso}}$$
\end{proposition}

\bigskip

\subsection{Skew-gentle algebras and dissected surfaces}

\begin{definition}
A \emph{marked orbifold} $(\surf, M,P, X)$ is the data of  
\begin{itemize}
\item a marked surface $(\surf,M,P)$
\item a finite set  $X$ of points in the interior of $\surf$, called \emph{orbifold points}.
\end{itemize}

An \emph{arc} on $(\surf, M, P, X)$ is an arc with endpoints in $M$, $P$ or $X$.

A \emph{$\cross$-dissection} is a $\rpoint$-dissection $D$ of the marked surface $(\surf, M,P\cup X)$ such that each $\cross$ in $X$ is the endpoint of exactly one arc $j_\cross$. We call these arcs $j_\cross$ the $\cross$-arcs of $D$, and arcs with both endpoints in $M \cup P$ are referred to as $\rpoint$-arcs.

Two $\cross$-dissected orbifolds $(\surf,M,P,X,D)$ and $(\surf',M',P',X',D')$ are called diffeomorphic if there exists an orientation preserving diffeomorphism $\Phi:\surf\to \surf'$ such that $\Phi(M)=M'$, $\Phi(P)=P'$, $\Phi(X)=X'$ and $\Phi(D)=D'$.
\end{definition} 

Considering the $\cross$-dissection $D$ as a $\rpoint$-dissection of $(\surf,M,P\cup X)$, one can associate to $D$ a gentle pair $(Q,I)$. The condition for a $\rpoint$-dissection to be a $\cross$-dissection implies that there is a distinguished set of square zero loops corresponding to  the unique arcs linking a $\cross\in X$ to a $\rpoint\in M\cup P$. Hence, one  can define a skew-gentle triple $(Q,I',{\rm Sp})$ with skew-gentle algebra $\bar{A}(D):= \bar{A}(Q,I',{\rm Sp}))$ from $D$, where ${\rm Sp}$ is the set of loops attached to the $\cross$'s, and where $I':=I\setminus \{e^2, e\in {\rm Sp}\}$.

\begin{proposition}\label{prop::bijection-skewgentle-dissection}
The assignment $D \mapsto \bar{A}(D)$ maps $\cross$-dissections to skew-gentle algebras, and all skew-gentle algebras are obtained in this way. 

\end{proposition}
\begin{proof}
The gentle algebra $\bar{A}_0(Q,I,{\rm Sp})$ which is a degeneration of a given skew-gentle algebra $\bar{A}(Q,I,{\rm Sp})$ is obtained by the bijection in Proposition \ref{prop::bijection-gentle-dissection} uniquely by a $\rpoint$-dissection of a surface $(\surf,M,P)$. Square-zero loops of $\bar{A}_0(Q,I,{\rm Sp})$ correspond under this bijection to self-folded triangles containing one $\rpoint$ in its interior. Changing the $\rpoint$ to a $\cross$, one obtains a $\cross$-dissection $D$, and the choice of $\cross$'s corresponds to a selection of special loops, thus $\bar{A}(D) = \bar{A}(Q,I,{\rm Sp})$.
\end{proof}
\bigskip

We now describe in detail the generalized version of the assignment $D \mapsto A(D)$ from \cite{OPS}. Let $D$ be a $\cross$-dissection of a surface $(\surf,M,P)$. Then the quiver $\bar{Q}$ of the algebra $\bar{A}(D)$  and its set of relations $\bar{I}$ such that $\bar{A}(D)= k\bar{Q}/\bar{I}$ can be constructed as follows: 
\begin{itemize}
\item
The vertices of $\bar{Q}$ are in bijection with 
$$\{ i\ \rpoint {\textrm -arc\}}\cup (\{j\ \cross{\textrm -arc}\}\times \mathbb Z_2).$$
\item
Given $i$ and $j$ $\rpoint$-arcs in $D$, there is one arrow 
 \[
  \scalebox{0.8}{
  \begin{tikzpicture}
  \node (I) at (0,0) {$i$};
  \node (J) at (2,0) {$j$};
  \draw[thick, ->] (I)--node[fill=white, inner sep=0pt]{$\alpha$}(J);
  \end{tikzpicture}}\] in $\bar{Q}$ whenever the arcs $i$ and $j$ have a common endpoint $\rpoint$ and when $i$ is immediately followed by the arc $j$ in the counterclockwise order around $\rpoint$;
\item Given a $\rpoint$-arc $i$ and a $\cross$-arc $j$ in $D$, there are two arrows 
 \[
  \scalebox{0.8}{
  \begin{tikzpicture}
  \node (I) at (0,0) {$i$};
  \node (J) at (2,1) {$j_0$};
  \node (J1) at (2,-1) {$j_1$};
  \draw[thick, ->] (I)--node[fill=white, inner sep=0pt]{$^0\alpha$}(J);
   \draw[thick, ->] (I)--node[fill=white, inner sep=0pt]{$^1\alpha$}(J1);
  
  \begin{scope}[xshift=5cm]
  \node at (-1,0) {( resp.};
   \node (I) at (2,0) {$i$};
  \node (J) at (0,1) {$j_0$};
  \node (J1) at (0,-1) {$j_1$};
  \node at (2.5,0) {)};
  \draw[thick, ->] (J)--node[fill=white, inner sep=0pt]{$\alpha^0$}(I);
   \draw[thick, ->] (J1)--node[fill=white, inner sep=0pt]{$\alpha^1$}(I);
   
  \end{scope}
  \end{tikzpicture}}\] in $\bar{Q}$ whenever the arcs $i$ and $j$ have a common endpoint $\rpoint$ and when $i$ is immediately followed by the arc $j$ in the counterclockwise (resp. clockwise) order around $\rpoint$ ;
\item Given $i$ and $j$ $\cross$-arcs in $D$, there are four arrows 
\[
  \scalebox{0.8}{
  \begin{tikzpicture}[scale=1, >=stealth]
  \node (I) at (0,1) {$i_0$};
  \node (I1) at (0,-1) {$i_1$};
  \node (J) at (3,1) {$j_0$};
  \node (J1) at (3,-1) {$j_1$};
  \draw[thick, ->] (I)--node[fill=white, inner sep=0pt]{$^0\alpha ^0$}(J);
   \draw[thick, ->] (I1)--node[fill=white, inner sep=0pt, xshift=15pt, yshift=10pt]{$^0\alpha ^1$}(J);
    \draw[thick, ->] (I)--node[fill=white, inner sep=0pt, xshift=-15pt, yshift=10pt]{$^1\alpha ^0$}(J1);
   \draw[thick, ->] (I1)--node[fill=white, inner sep=0pt]{$^1\alpha ^1$}(J1);
   
   \end{tikzpicture}}\]

in $\bar{Q}$ whenever the arcs $i$ and $j$ have a common endpoint $\rpoint$ and when $i$ is immediately followed by the arc $j$ in the counterclockwise order around $\rpoint$.

\end{itemize}

The set of relations $\bar{I}$ can be described as follows: If $i$, $j$, and $k$ have a common endpoint $\rpoint$, and are consecutive arcs following the counterclockwise  order around $\rpoint$, then 

\begin{itemize}
\item if $j$ is a $\rpoint$-arc, we have  $^{\epsilon}\beta\alpha^{\epsilon'}\in \bar{I}$, for each $\epsilon=0,1,\emptyset$ and $\epsilon'=0,1,\emptyset$, when the expression makes sense . 
\item if $j$ is a $\cross$-arc, then we have $(^{\epsilon}\beta^0) (^0\alpha^{\epsilon'})+(^{\epsilon}\beta^1)( ^1\alpha^{\epsilon'})\in \bar{I}$, where $\epsilon=0,1,\emptyset$ and $\epsilon'=0,1,\emptyset$, when the expression makes sense.
\end{itemize}

\begin{example}\label{example::Dn}
Consider a disc with one orbifold point and $n-1$ marked points on the boundary with the following $\cross$-dissection depicted for $n=5$:
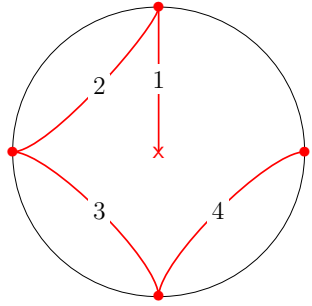
\[
  \scalebox{0.8}{
  \begin{tikzpicture}[scale=0.6, >=stealth]
  
 \draw (0,0) circle (4);
 
 \node at (0,0) {$\cross$};
 \node at (4,0) {$\rpoint$};

 \node at (0,4) {$\rpoint$};

 \node at (-4,0) {$\rpoint$};

 \node at (0,-4) {$\rpoint$};

 \draw[red, thick] (0,0)--node [black,fill=white, inner sep=3pt]{$1$} (0,4);
 \draw[red,thick] (0,4)..node [black,fill=white, inner sep=2pt]{$2$} controls (0,3.5) and (-3,0).. (-4,0)..node [black,fill=white, inner sep=2pt]{$3$}controls (-3,0)  and (0,-3)..(0,-4).. node [black,fill=white, inner sep=2pt]{$4$}controls (0,-3)  and (3,0).. (4,0);
  
  \end{tikzpicture}}\]
  
  The corresponding skew-gentle algebra is a quiver of type $D_n$ as follows:
 \[
  \scalebox{0.8}
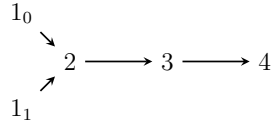
{
  \begin{tikzpicture}[scale=0.8, >=stealth]
  \node (10) at (0,1) {$1_0$}; 
 \node (11) at (0,-1) {$1_1$}; 
 \node (2) at (1,0) {$2$};
  \node (3) at (3,0) {$3$};
   \node (4) at (5,0) {$4$};
  
\draw[thick, ->] (10)--(2);
\draw[thick, ->] (11)--(2);
\draw[thick, ->] (2)--(3);
\draw[thick, ->] (3)--(4);

\end{tikzpicture}}\]

Now consider the disc with $n-2$ marked points on the boundary and $2$ orbifold points with a  dissection of the following form:
 
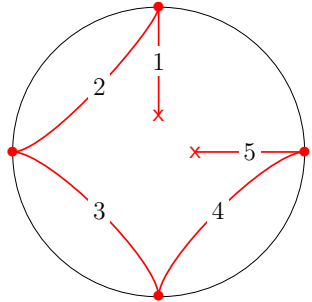
\[
  \scalebox{0.8}{
  \begin{tikzpicture}[scale=0.6, >=stealth]
  
 \draw (0,0) circle (4);
 
 \node at (0,1) {$\cross$};
 \node at (1,0) {$\cross$};
 \node at (4,0) {$\rpoint$};
 
 \node at (0,4) {$\rpoint$};

 \node at (-4,0) {$\rpoint$};

 \node at (0,-4) {$\rpoint$};

 \draw[red, thick] (0,1)--node [black,fill=white, inner sep=3pt]{$1$} (0,4);
 \draw[red, thick] (1,0)--node [black,fill=white, inner sep=2pt]{$5$} (4,0);
 \draw[red,thick] (0,4)..node [black,fill=white, inner sep=2pt]{$2$} controls (0,3.5) and (-3,0).. (-4,0)..node [black,fill=white, inner sep=2pt]{$3$}controls (-3,0) and (0,-3)..(0,-4).. node [black,fill=white, inner sep=2pt]{$4$}controls (0,-3)  and (3,0).. (4,0);

  \end{tikzpicture}}\]    
  
  Then the corresponding skew-gentle algebra is of type $\widetilde{D}_n$:
  
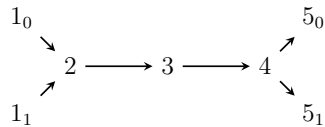
 \[
  \scalebox{0.8}{
  \begin{tikzpicture}[scale=0.8, >=stealth]
  \node (10) at (0,1) {$1_0$}; 
 \node (11) at (0,-1) {$1_1$}; 
 \node (2) at (1,0) {$2$};
  \node (3) at (3,0) {$3$};
   \node (4) at (5,0) {$4$};
   
       \node (80) at (6,1) {$5_0$};
       \node (81) at (6,-1) {$5_1$};
       
\draw[thick, ->] (10)--(2);
\draw[thick, ->] (11)--(2);
\draw[thick, ->] (2)--(3);
\draw[thick, ->] (3)--(4);
\draw[thick, ->] (4)--(80);
\draw[thick,->] (4)--(81);
\end{tikzpicture}}\]

\end{example}
\section{Skew-gentle as skew-group algebras, and $\mathbb Z_2$-action on a surface}\label{section::Z2action}

From now on, and in the rest of the paper, $G$ will be the group $\mathbb Z_2$. 

\subsection{$\mathbb Z_2$-action on dissected surfaces}

Let $(\surf,M,P)$ be a marked surface, and let $\sigma:\surf\to \surf$ be a diffeomorphism of order $2$, preserving setwise $P$ and $M$, and having finitely many fixed points which are all in $\surf \backslash P$.
We call these data a $G$-marked surface.

This induces a free action of the group $G = \{1, \sigma\}$ on the sets $M$ and $P$. We denote by $X$ the set of fixed points of $\sigma$ and we define a $G$-dissection $D$ to be a $\rpoint$-dissection of $(\surf,M,P)$ which is fixed (globally) by $\sigma$. 
We also refer to $(\surf,M,P,\sigma, D)$ as a $G$-dissected surface.

Two $G$-dissected surfaces $(\surf,M,P,\sigma,D)$ and $(\surf',M',P',\sigma',D')$ are called $G$-diffeomorphic if there exists an orientation preserving diffeomorphism $\Phi:\surf\to \surf'$ preserving the marked points, sending $D$ to $D'$, and such that $\Phi\circ\sigma=\sigma'\circ \Phi$.

From a $G$-dissection $D$, we obtain a gentle algebra $A(D)$ given by a gentle pair $(Q,I)$, and the diffeomorphism $\sigma$ induces a $G$-action on $Q$, fixing globally the paths of $I$. Therefore we get the following result:

\begin{proposition}\label{prop::bijection-G-gentle-dissection}
The assignment $D \mapsto A(D)$ induces an injective map
$$\left\{\begin{array}{cc}(\surf,M,P,\sigma,D) \\ \mbox{G\rm -dissected surface}\end{array} \!\right\}_{\textstyle / \; G\rm -diffeo} \longrightarrow \quad
\left\{\begin{array}{cc}A(Q,I) \\
\mbox{\rm $G$-gentle algebra }\end{array}\!\!\!\right\}_{\textstyle /G-{\rm iso}}$$

Moreover for each $G$-gentle algebra obtained above, the action of $G$ comes from an action on the quiver which is free on the arrows.
\end{proposition}

Given a diffeomorphism $\sigma:\surf\to \surf$ of order two, the quotient $\orbifold:=\surf / \sigma$ has a structure of orbifold surface, with orbifold points $X$. Denote by $p:\surf\to \orbifold$ the quotient map. We may consider $(\orbifold,\overline{M},\overline{P},X)$ as a marked orbifold.

\begin{proposition}
Let $(\surf,M,P,\sigma)$ be a $G$-marked surface. Then the projection $p:\surf\to \orbifold$ induces a bijection 
$$\{ D,\ G{\rm -dissection} (\surf,M,P,\sigma)\}\longleftrightarrow \{ \overline{D}, \cross{\rm-dissection} (\orbifold,\overline{M},\overline{P},X)\}.$$

Two $G$-dissections  are $G$-diffeomorphic if and only if the corresponding $\cross$-dissections are diffeomorphic.

\end{proposition}

\begin{proof}
Let $D$ be a $G$-dissection of $(\surf,M,P,\sigma)$. We first show that a fixed point $x$ of $\sigma$ cannot be in the interior of one of the polygons cut out by $D$. Indeed, the diffeomorphism $\sigma$ acts locally around $x$ as a central symmetry, so it would fix globally the polygon containing $x$. But this polygon has exactly one side which is on the boundary of $\surf$, thus $\sigma$ would fix globally this side, and $\sigma$ would have a fixed point on the boundary, a contradiction.

Therefore every fixed point of $\sigma$ lies in the interior of an arc of $D$.  If $\gamma$ is an arc in $D$ containing two distinct fixed points, then $\gamma$ would fix a point in between, this would contradict the fact that $X$ is finite by an easy induction . Finally, if $\gamma$ contains $x\in X$, then $\gamma$ is fixed by $\sigma$ since $\gamma$ does not intersect another arc of $D$. Again arguing by finiteness of $X$, we cannot have $\sigma(\gamma)=\gamma$, hence we conclude $\sigma(\gamma)=\gamma^{-1}$. We have therefore shown that the $G$-dissection $D$ has exactly $m = |X|$ arcs $\gamma$ such that $\sigma(\gamma)=\gamma^{-1}$ and each of them contains exactly one point in $X$. Setting $X= \{x_1, \ldots ,x_m\}$, we can write $$D=\{\gamma_1,\ldots,\gamma_m,\alpha_1,\ldots,\alpha_s, \sigma(\alpha_1),\ldots,\sigma(\alpha_s)\}$$
with $\sigma(\gamma_i)=\gamma^{-1}_i$, and we can assume $x_i= \gamma_i(\frac{1}{2})$. 
Cutting the self-symmetric arcs into two parts at the fixed point, we write $\gamma_i=\gamma_i^0.\gamma_i^1$ where $\gamma_i^0(0)=x_i$. Then the set of arcs
$$\{\gamma_1^0, \gamma_1^1,\ldots,\gamma_m^0,\gamma_m^1,\alpha_1,\ldots,\alpha_s, \sigma(\alpha_1),\ldots,\sigma(\alpha_s)\}$$ is a dissection of $(\surf\setminus X,M,P\cup X)$ for which every $x_i\in X$ belongs exactly to the two arcs $\gamma_i^0$ and $\gamma_i^1$. 
Therefore the collection $$\overline{D}=\{p(\gamma_1^0),\ldots, p(\gamma^0_m),p(\alpha_1),\ldots, p(\alpha_s)\}$$ is a system of non-intersecting arcs.
The action of $\sigma$ on the polygons cut out by $D$ is free, indeed if one polygon were fixed, then $\sigma$ would have a fixed point in its interior.
Since the projection $\surf\setminus X\to \overline{\surf}\setminus X$ is a two folded cover without branched points, the collection $\overline{D}$ cuts the surface $\overline{\surf}$ into polygons with exactly one boundary segment on the boundary. Therefore $\overline{D}$ is a $\cross$-dissection of $\overline{\surf}$.

\medskip

Conversely, let 
$\overline{D}=\{\bar{\gamma}_1, \ldots, \bar{\gamma}_m, \bar{\alpha}_1, \ldots, \bar{\alpha}_s\}$ be a $\cross$-dissection of $(\overline{\surf},\overline{M}, \overline{P},X)$ where the $\bar{\gamma}_i$ are the arcs incident to a point in $X$. Then $p^{-1}(\bar{\alpha}_i)$ is a union of two arcs that do not intersect and that are mapped under $\sigma$ onto each other, thus we can write 
$p^{-1}(\bar{\alpha}_i)=\{\alpha_i, \sigma(\alpha_i)\}.$

The preimage $p^{-1}(\bar{\gamma}_i)$ is a union of two curves that both have $x_i\in X$ as endpoint. So if we write 
$p^{-1}(\bar{\gamma}_i)=\{\gamma_i, \sigma(\gamma_i)\},$
we have that $\widetilde{\gamma}_i:=\gamma_i.\sigma(\gamma_i)^{-1}$ is an arc of $(\surf, M,P)$. It is then easy to see that 
$$D:=\{\widetilde{\gamma}_1,\ldots,\widetilde{\gamma}_m,\alpha_1,\ldots,\alpha_s, \sigma(\alpha_1),\ldots,\sigma(\alpha_s)\}$$ is a dissection of $(\surf, M,P)$ which is invariant under $\sigma$.

\end{proof}

\begin{proposition}\label{prop::iso AG-Abar} Let $(\surf,M,P,\sigma,D)$ be a $G$-dissected marked surface. There is an isomorphism of algebras 
$$ (A(D)G)_{\rm b}\simeq \bar{A}(\overline{D}),$$
where $(A(D)G)_{\rm b}$ is the basic algebra of the skew-group algebra $A(D)G$.
\end{proposition}

\begin{proof}
We denote by $A:=A(D)$ and $\bar{A}:=\bar{A}(\bar{D})$.

The description of the quiver of $AG$ follows from \cite{ReitenRiedtmann}. But in order to understand the relations, we need to exhibit a specific idempotent $\eta\in AG$ which turns $\eta AG\eta $ into a basic algebra, together with an explicit isomorphism $\varphi:\bar{A}\longrightarrow \eta AG\eta $.

The action of $\sigma$ is free on $M$ and $P$, which allows to write $M=M^+\amalg M^-$ and $P=P^+\amalg P^-$ by choosing a representative for each orbit. This choice induces a partition $Q_1=Q_1^+\amalg Q_1^-$ of the arrows of $Q$ where an arrow 
$i\to j$ is in  $Q_1^+$ if and only if the corresponding endpoint common to arc $i$ and $j$ is in $M^+\cup P^+$. We denote the arrows in $Q_1^\epsilon$ by $\alpha^\epsilon$ for $\epsilon \in \{+,-\}$. This partition of arrows implies that a composition $\alpha^{\epsilon}\beta^{\epsilon'}$ is in $I$ if and only if $\epsilon=\epsilon'$. 

Now we choose a representative for each $\sigma$-orbit of the vertices $Q_0$ and denote $Q_0=Q_0^+\amalg Q_0^-\amalg Q_0^{\rm fix}$ (this choice is done independently from the choice of arrows).
Then a complete set of primitive pairwise orthogonal idempotents of $AG$ is given as follows:
$$\{e_{i^+}\smallten 1_G, i^+\in Q_0^+\}\cup \{e_{i^-}\smallten 1_G, i^-\in Q_0^-\}\cup \{e_j\smallten \frac{1+\sigma}{2}, e_j\smallten \frac{1-\sigma}{2}, j\in Q_0^{\rm fix}\}.$$
The automorphism $\sigma\otimes 1_G$ of $AG$ induces an isomorphism between the projectives $(e_{i^+}\smallten 1) AG$ and $(e_{i^-}\smallten 1) AG$. Let us fix
$$\eta:= \sum_{i^+\in Q_0^+}e_{i^+}\smallten 1 + \sum_{j\in Q_0^{\rm fix}}e_j\smallten 1,$$ then using \cite{ReitenRiedtmann}, the algebra $\eta AG \eta$ is basic and we have an isomorphism of algebras
$(AG)_{\rm b}\simeq \eta AG \eta$.

Consider now   the projection map $p:\surf\to \bar{\surf}$. 
Let $\gamma_i$ be a $\rpoint$-arc in $\bar{D}$, corresponding to a vertex $i$ in the quiver $Q(\bar{D})$ of $\bar{A}(\bar{D})$. Then $p^{-1}(\gamma_i)$ is a pair ${\gamma_i^+,\gamma_i^-}\in D$ with $\sigma (\gamma^+_i)=\gamma^-_i$ which corresponds to vertices $i^+$ and $i^-$ in $Q(D)$. If $\gamma_i$ is a $\cross$-arc in $\bar{D}$, it corresponds to two vertices $i_0$ and $i_1$ in $\bar{Q}$. Its preimage in $\surf$ is an arc of $D$ which is $\sigma$-invariant, so there is one corresponding vertex in $Q(D)$ denoted by $i$.

Let $\gamma_i$ and $\gamma_j$ be two arcs in $\bar{D}$ ($\rpoint$, or $\cross$) having a common endpoint $m\in \bar{M}$ and such that $\gamma_i$ is immediately followed by $\gamma_j$ in the counterclockwise direction around  $m$. The point $m$ has exactly two preimages $m^+\in M^+$ and $m^-\in M^-$ in $\surf$. Hence there are exactly two arrows $\alpha^+\in Q_1^+$ and $\alpha^{-}\in Q_1^-$ in the quiver of $D$. Note that if $\gamma_i$ (resp. $\gamma_j$) is a $\rpoint$-arc, the  source (resp. tail) of $\alpha^+$ maybe either $i^+$ or $i^-$ (resp. $j^+$ or $j^-$). The possible local configurations of the two quivers $Q(D)$ and $Q(\bar{D})$ are summarized in Figure \ref{figure::basic algebra}.

\begin{figure}
\renewcommand{\arraystretch}{1.2}
\begin{tabular}{|c|c|c|c|}
\hline
$\bar{D}$
   &  $Q(\bar{D})$ & $D$ & $Q(D)$\\
\hline $\scalebox{0.8}{
  \begin{tikzpicture}[scale=0.6, >=stealth]
  \node at (0,0) {$\rpoint$};
  \node at (1,2) {$\rpoint$};
  \node at (2,0) {$\rpoint$};
  \draw[red] (0,0) --node [xshift=-5pt]{$\gamma_i$} (1,2);
  \draw[red] (1,2) --node [xshift=5pt]{$\gamma_j$} (2,0);
  \node at (1,2.4) {$m$};
  \end{tikzpicture}}$
  
&  $\scalebox{0.6}{\begin{tikzpicture}[scale=1, >=stealth]
\node (I) at (0,0) {$i$};
\node (J) at (2,0) {$j$};
\draw[thick, ->] (I)--node[fill=white, inner sep=0pt]{$\alpha$}(J);
\end{tikzpicture}}$ & $\scalebox{0.8}{
  \begin{tikzpicture}[scale=0.4, >=stealth]
  \node at (0,0) {$\rpoint$};
  \node at (1,2) {$\rpoint$};
  \node at (2,0) {$\rpoint$};
  \draw[red] (0,0) -- (1,2);
  \draw[red] (1,2) --(2,0);
   \node at (1,2.6) {$m^+$};
  \begin{scope}[rotate=180, xshift=-6cm]
  \node at (0,0) {$\rpoint$};
  \node at (1,2) {$\rpoint$};
  \node at (2,0) {$\rpoint$};
  \draw[red] (0,0) -- (1,2);
  \draw[red] (1,2) --(2,0);
   \node at (1,2.3) {$m^-$};
  \end{scope}
  \end{tikzpicture}}$ & $\scalebox{0.6}{\begin{tikzpicture}[scale=1, >=stealth]
\node (I) at (0,0) {$.$};
\node (J) at (2,0) {$.$};
\node (I1) at (0,1) {$.$};
\node (J1) at (2,1) {$.$};
\draw[thick, ->] (I)--node[fill=white, inner sep=0pt]{$\alpha^+$}(J);
\draw[thick, ->] (I1)--node[fill=white, inner sep=0pt]{$\alpha^-$}(J1);
\end{tikzpicture}}$ \\
\hline

$\scalebox{0.8}{
  \begin{tikzpicture}[scale=0.6, >=stealth]
  \node at (0,0) {$\rpoint$};
  \node at (1,2) {$\rpoint$};
  \node at (2,0) {$\cross$};
  \draw[red] (0,0) --node [xshift=-5pt]{$\gamma_i$} (1,2);
  \draw[red] (1,2) --node [xshift=5pt]{$\gamma_j$} (2,0);
  \node at (1,2.4) {$m$};
  \end{tikzpicture}}$
  
&  $\scalebox{0.6}{\begin{tikzpicture}[scale=1, >=stealth]
\node (I) at (0,0) {$i$};
\node (J0) at (2,1) {$j_0$};
\node (J1) at (2,-1) {$j_1$};
\draw[thick, ->] (I)--node[fill=white, inner sep=0pt]{$^0\alpha$}(J0);
\draw[thick, ->] (I)--node[fill=white, inner sep=0pt]{$^1\alpha$}(J1);
\end{tikzpicture}}$ 
& $\scalebox{0.8}{
  \begin{tikzpicture}[scale=0.4, >=stealth]
  \node at (0,0) {$\rpoint$};
  \node at (1,2) {$\rpoint$};
  \node at (2,0) {$\cross$};
  \draw[red] (0,0) -- (1,2);
  \draw[red] (1,2) --(2,0);
   \node at (1,2.6) {$m^+$};
  \begin{scope}[rotate=180, xshift=-4cm]
  \node at (0,0) {$\rpoint$};
  \node at (1,2) {$\rpoint$};
 
  \draw[red] (0,0) -- (1,2);
  \draw[red] (1,2) --(2,0);
   \node at (1,2.3) {$m^-$};
  \end{scope}
  \end{tikzpicture}}$ & $\scalebox{0.6}{\begin{tikzpicture}[scale=1, >=stealth]
\node (I) at (0,1) {$.$};
\node (J) at (2,0) {$j$};
\node (I1) at (0,-1) {$.$};

\draw[thick, ->] (I)--node[fill=white, inner sep=0pt]{$\alpha^+$}(J);
\draw[thick, ->] (I1)--node[fill=white, inner sep=0pt]{$\alpha^-$}(J);
\end{tikzpicture}}$ \\
\hline

$\scalebox{0.8}{
  \begin{tikzpicture}[scale=0.6, >=stealth]
  \node at (0,0) {$\cross$};
  \node at (1,2) {$\rpoint$};
  \node at (2,0) {$\rpoint$};
  \draw[red] (0,0) --node [xshift=-5pt]{$\gamma_i$} (1,2);
  \draw[red] (1,2) --node [xshift=5pt]{$\gamma_j$} (2,0);
  \node at (1,2.4) {$m$};
  \end{tikzpicture}}$
  
&  $\scalebox{0.6}{\begin{tikzpicture}[scale=1, >=stealth]
\node (I) at (0,1) {$i_0$};
\node (I1) at (0,-1) {$i_1$};
\node (J) at (2,0) {$j$};
\draw[thick, ->] (I)--node[fill=white, inner sep=0pt]{$\alpha^0$}(J);
\draw[thick, ->] (I1)--node[fill=white, inner sep=0pt]{$\alpha^1$}(J);
\end{tikzpicture}}$ & $\scalebox{0.8}{
  \begin{tikzpicture}[scale=0.4, >=stealth]
  \node at (0,0) {$\cross$};
  \node at (1,2) {$\rpoint$};
  \node at (2,0) {$\rpoint$};
  \draw[red] (0,0) -- (1,2);
  \draw[red] (1,2) --(2,0);
   \node at (1,2.6) {$m^+$};
  \begin{scope}[rotate=180]
 
  \node at (1,2) {$\rpoint$};
  \node at (2,0) {$\rpoint$};
  \draw[red] (0,0) -- (1,2);
  \draw[red] (1,2) --(2,0);
   \node at (1,2.3) {$m^-$};
  \end{scope}
  \end{tikzpicture}}$ & $\scalebox{0.6}{\begin{tikzpicture}[scale=1, >=stealth]
\node (I) at (0,0) {$i$};
\node (J) at (2,1) {$.$};
\node (J1) at (2,-1) {$.$};
\draw[thick, ->] (I)--node[fill=white, inner sep=0pt]{$\alpha^+$}(J);
\draw[thick, ->] (I)--node[fill=white, inner sep=0pt]{$\alpha^-$}(J1);
\end{tikzpicture}}$ \\
\hline

$\scalebox{0.8}{
  \begin{tikzpicture}[scale=0.6, >=stealth]
  \node at (0,0) {$\cross$};
  \node at (1,2) {$\rpoint$};
  \node at (2,0) {$\cross$};
  \draw[red] (0,0) --node [xshift=-5pt]{$\gamma_i$} (1,2);
  \draw[red] (1,2) --node [xshift=5pt]{$\gamma_j$} (2,0);
  \node at (1,2.4) {$m$};
  \end{tikzpicture}}$
  
&  $\scalebox{0.6}{\begin{tikzpicture}[scale=1, >=stealth]
\node (I) at (0,1) {$i_0$};
\node (J) at (2,1) {$j_0$};
\draw[thick, ->] (I)--node[fill=white, inner sep=0pt]{$^0\alpha^0$}(J);
\node (I1) at (0,-1) {$i_1$};
\node (J1) at (2,-1) {$j_1$};
\draw[thick, ->] (I1)--node[fill=white, inner sep=0pt, xshift=8pt, yshift=8pt]{$^0\alpha^1$}(J);
\draw[thick, ->] (I)--node[fill=white, inner sep=0pt, xshift=8pt, yshift=-8pt]{$^1\alpha^0$}(J1);
\draw[thick, ->] (I1)--node[fill=white, inner sep=0pt]{$^1\alpha^1$}(J1);
\end{tikzpicture}}$ 

& $\scalebox{0.8}{
  \begin{tikzpicture}[scale=0.4, >=stealth]
  \node at (0,0) {$\cross$};
  \node at (1,2) {$\rpoint$};
  \node at (2,0) {$\cross$};
  \node at (-1,-2) {$\rpoint$};
  \draw[red] (-1,-2) -- node [xshift=-8pt]{$\gamma_i$}(1,2);
  \draw[red] (1,2) --(2,0);
   \node at (1,2.6) {$m^+$};
  \begin{scope}[rotate=180, xshift=-4cm]
  \node at (0,0) {$\cross$};
  \node at (1,2) {$\rpoint$};
  \node at (-1,-2) {$\rpoint$};
  \draw[red] (-1,-2) --node [xshift=8pt]{$\gamma_i$} (1,2);
  \draw[red] (1,2) --(2,0);
   \node at (1,2.3) {$m^-$};
  \end{scope}
  \end{tikzpicture}}$ & $\scalebox{0.6}{\begin{tikzpicture}[scale=1, >=stealth]
\node (I) at (0,0) {$i$};
\node (J) at (2,0) {$j$};

\draw[thick, ->] (0.2,0.1)--node[yshift=8pt]{$\alpha^+$}(1.8,0.1);
\draw[thick, ->] (0.2,-0.1)--node[yshift=-5pt]{$\alpha^-$}(1.8,-0.1);
\end{tikzpicture}}$ \\
\hline
\end{tabular}
\caption{\label{figure::basic algebra}}
\end{figure}

Now we define a map $\Phi:k Q(\bar{D})\to \eta A(D)G \eta$ by 
\[ \begin{array}{lcll}
\Phi (\bar{e}_i) &= &e_{i^+}\smallten 1 & \\
\Phi(\bar{e}_{i_{\epsilon}}) &=&e_i\smallten \frac{1+(-1)^\epsilon \sigma}{2}   & \epsilon=0,1\\
\Phi(\alpha) &=&(e_{j^+}\smallten 1)((\alpha^++\alpha^-)\smallten (1+\sigma)) (e_{i^+}\smallten 1) & {\rm for}\ \alpha:i\to j ;\\

\Phi(\alpha^\epsilon) &=&(e_{j^+}\smallten 1)(\alpha^+\smallten 1+(-1)^{\epsilon}\alpha^-\smallten 1)(e_i\smallten \frac{1+(-1)^{\epsilon}\sigma}{2}) & \textrm{ for }\alpha^\epsilon :i_{\epsilon}\to j\\

\Phi(^{\epsilon}\alpha ) &=&(e_{j}\smallten \frac{1+(-1)^\epsilon \sigma}{2})(\alpha^+\smallten 1+(-1)^{\epsilon}\alpha^-\smallten 1)(e_{i^+}\smallten 1) & \textrm{ for }^\epsilon\alpha :i\to j_{\epsilon}\\

\Phi(^{\epsilon'}\alpha^{\epsilon} ) &=&(e_{j}\smallten \frac{1+(-1)^{\epsilon'} \sigma}{2})(\alpha^+\smallten 1)(e_{i}\smallten\frac{1+(-1)^{\epsilon}\sigma}{2} ) & \textrm{ for }^{\epsilon'}\alpha^\epsilon :i_{\epsilon}\to j_{\epsilon'}
\end{array}\]

It remains to check that the map $\Phi$ factors through the skew-gentle relations. 
Let $i$, $j$ and $k$ be consecutive arcs around a $\rpoint$-point in $\bar{D}$. 
Assume first that $i$, $j$ and $k$ are $\rpoint$-arcs. Then we have $$((\beta^++\beta^-)\smallten (1+\sigma)).((\alpha^++\alpha^-)\smallten(1+\sigma))=2(\beta^+\alpha^++\beta^-\alpha^-)\smallten (1+\sigma),$$
since the arrows $\beta^+$ and $\alpha^-$  (resp. $\beta^-$ and $\alpha^+$) do not compose. One can check that $\Phi(\beta\alpha)$ is  one of the 8 terms of the right hand side, depending on the sign index of the source and tail of $\alpha^+$ and $\beta^+$. 
Therefore we clearly have $\Phi(\beta\alpha)=0$, since $\beta^+\alpha^+$ and $\beta^-\alpha^-$ are in $I$.

For example assume that $\alpha^+:i^-\to j^-$ and $\beta^+:j^-\to k^+$. Then one has
\[ \Phi(\beta)  =  \beta^+\smallten \sigma \textrm{ and } \Phi(\alpha)=\alpha^-\smallten 1\]
 thus $\Phi(\beta\alpha)= \beta^+\alpha^+\smallten \sigma.$
The computations are similar if one of $i$, or $k$, or both are $\cross$-arcs. 

\medskip

Assume now that $j$ is a $\cross$-arc, and $i$ and $k$ are $\rpoint$-arcs. Then we have 
\begin{align*}
(\beta^+\smallten 1+\beta^-\smallten 1)(e_j\smallten\frac{1+\sigma}{2})(\alpha^+\smallten 1+\alpha^-\smallten 1)\\+(\beta^+\smallten 1-\beta^-\smallten 1)(e_j\smallten\frac{1-\sigma}{2})(\alpha^+\smallten 1-\alpha^-\smallten 1)\\
=(\beta^+\alpha^++\beta^-\alpha^-)\smallten( 1+\sigma)
\end{align*}
Therefore we obtain $\Phi((\beta^0) (^0\alpha)+(\beta^1) (^1\alpha))\in \eta (I\otimes 1+I\otimes \sigma)\eta$.
The computations are similar for one of $i$, $k$ or both  being $\cross$-arcs.

Finally $\Phi$ is an isomorphism of algebras.
\end{proof}

\subsection{Examples}

\begin{example}\label{example::coveringDn}

Consider the disc with $2n$ marked points on the boundary with $\sigma$ being the central symmetry and with the following $G$-dissection.

 \[
  \scalebox{0.8}
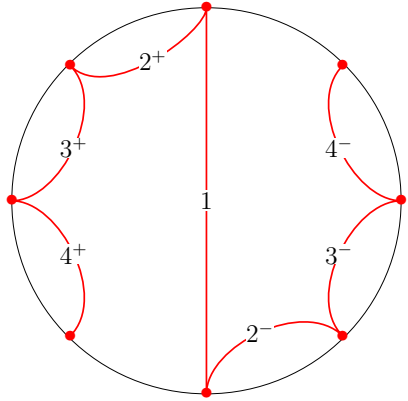
{
  \begin{tikzpicture}[scale=0.8, >=stealth]
  
 \draw (0,0) circle (4);

 \node at (4,0) {$\rpoint$};
 \node at (2.8,2.8) {$\rpoint$};
 \node at (0,4) {$\rpoint$};
 \node at (-2.8,2.8) {$\rpoint$};
 \node at (-4,0) {$\rpoint$};
 \node at (-2.8,-2.8) {$\rpoint$};
 \node at (0,-4) {$\rpoint$};
 \node at (2.8,-2.8) {$\rpoint$};

 \draw[red,thick] (0,-4)--node [black,fill=white, inner sep=1pt]{$1$}(0,4)..node [black,fill=white, inner sep=0pt]{$2^+$} controls (0,3.5) and (-2,2)..(-2.8,2.8)..node [black,fill=white, inner sep=0pt]{$3^+$}controls (-2,2) and (-3,0).. (-4,0)..node [black,fill=white, inner sep=0pt]{$4^+$}controls (-3,0) and (-2,-2)..(-2.8,-2.8);
  \draw[red, thick](0,-4).. node [black,fill=white, inner sep=0pt]{$2^-$}controls (0,-3) and (2,-2).. (2.8,-2.8).. node [black,fill=white, inner sep=0pt]{$3^-$}controls (2,-2) and (3,0).. (4,0).. node [black, fill=white, inner sep=0pt]{$4^-$} controls (3,0) and (2,2) ..(2.8,2.8);

  \end{tikzpicture}}\]    
  
  The corresponding gentle algebra is the path algebra $\Lambda$ of the following quiver:
   \[
  \scalebox{0.8}
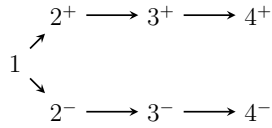
{
  \begin{tikzpicture}[scale=0.8, >=stealth]
  \node (1) at (0,0) {$1$}; 

 \node (2) at (1,1) {$2^+$};
  \node (3) at (3,1) {$3^+$};
   \node (4) at (5,1) {$4^+$};
    \node (2-) at (1,-1) {$2^-$};
     \node (3-) at (3,-1) {$3^-$};
      \node (4-) at (5,-1) {$4^-$};

\draw[thick, ->] (1)--(2);
\draw[thick, ->] (1)--(2-);
\draw[thick, ->] (2)--(3);
\draw[thick, ->] (3)--(4);
\draw[thick, ->] (2-)--(3-);
\draw[thick, ->] (3-)--(4-);

\end{tikzpicture}}\]

The automorphism $\sigma$ has a unique fixed point, so it is immediate to see that the corresponding orbifold surface $\surf/\sigma$ is the disc with one orbifold point and $n$ marked points on the boundary. The skew-group algebra $\Lambda G$ (where the group action is given by seding the vertices $+$ to $-$)  is Morita equivalent to the path algebra of $\mathbb D_n$ (cf Example \ref{example::Dn}).

\medskip

Similarly, taking a cylinder with $n$ marked points on each boundary component, with $\sigma$ sending one boundary component to the other, we  can consider the following $G$-invariant dissection:

  \[
  \scalebox{0.8}
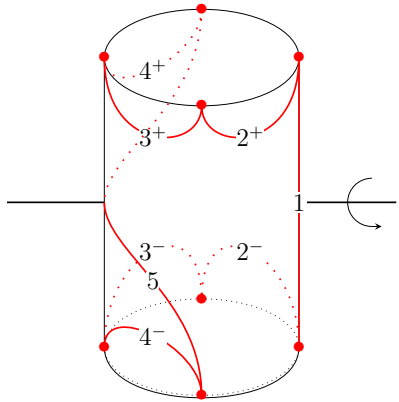
{
  \begin{tikzpicture}[scale=0.8, >=stealth]

\draw (0,0) circle (2 and 1);  
\draw[dotted] (0,-6) circle (2 and 1);
\draw (-2,0)--(-2,-6).. controls (-2,-7.4) and (2,-7.4).. (2,-6)--(2,0);

\node at (-2,0) {$\rpoint$};
\node at (0,1) {$\rpoint$};
\node at (2,0) {$\rpoint$};
\node at (0,-1) {$\rpoint$};
\node at (2,-6) {$\rpoint$};
\node at (-2,-6) {$\rpoint$};
\node at (0,-5) {$\rpoint$};
\node at (0,-7) {$\rpoint$};
  
\draw[thick] (-4,-3)--(-2,-3);
\draw[thick] (4,-3)--(2,-3);

\draw[red, thick] (-2,0)..node [black,fill=white, inner sep=0pt]{$3^+$}controls (-2,-2) and (0,-2).. (0,-1)..node [black,fill=white, inner sep=0pt]{$2^+$}controls (0,-2) and (2,-2).. (2,0)--node [black,fill=white, inner sep=1pt]{$1$}(2,-6);

\draw[red, thick] (-2,-6)..node [black,fill=white, inner sep=0pt]{$4^-$}controls (-2,-5) and (0,-6).. (0,-7).. node [black,fill=white, inner sep=0pt]{$5$}controls (0,-5) and (-2,-4).. (-2,-3);

\draw[red, thick, loosely dotted] (-2,-6).. node [black,fill=white, inner sep=0pt]{$3^-$}controls (-2,-4) and (0,-3).. (0,-5).. node [black,fill=white, inner sep=0pt]{$2^-$}controls (0,-3) and (2,-4).. (2,-6);

\draw[red, thick, loosely dotted] (-2,0).. node [black,fill=white, inner sep=0pt]{$4^+$}controls (-2,-1) and (0,0)..(0,1)
 ..controls (0,-1) and (-2,-2)..(-2,-3);
 
 \draw (3.5,-2.5) arc (90:270:0.5);
 \draw[->] (3.5,-3.5)--(3.7,-3.5);
  
  \end{tikzpicture}}\]
  
 The corresponding gentle algebra is the path algebra $\Lambda$ of the following quiver  
 
   \[
  \scalebox{0.8}
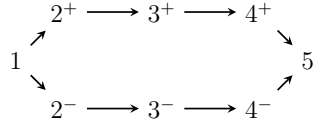
{
  \begin{tikzpicture}[scale=0.8, >=stealth]
  \node (1) at (0,0) {$1$}; 

 \node (2) at (1,1) {$2^+$};
  \node (3) at (3,1) {$3^+$};
   \node (4) at (5,1) {$4^+$};
    \node (2-) at (1,-1) {$2^-$};
     \node (3-) at (3,-1) {$3^-$};
      \node (4-) at (5,-1) {$4^-$};
      \node (5) at (6,0) {$5$};

\draw[thick, ->] (1)--(2);
\draw[thick, ->] (1)--(2-);
\draw[thick, ->] (2)--(3);
\draw[thick, ->] (3)--(4);
\draw[thick, ->] (2-)--(3-);
\draw[thick, ->] (3-)--(4-);
\draw[thick, ->] (4-)--(5);
\draw[thick, ->] (4)--(5);

\end{tikzpicture}}\]

The automorphism $\sigma$ has two fixed points, and it is immediate to see that the corresponding orbifold surface $\surf/\sigma$ is the disc with two orbifold points and $n$ marked point on the boundary. The skew-group algebra $\Lambda G$ (where the group action is given by sending the vertices $+$ to $-$)  is Morita equivalent to the path algebra of $\widetilde{\mathbb D}_n$ (cf Example \ref{example::Dn}).
 \medskip

More generally, one can consider a surface of genus $g$ with one or two boundary components and with $\sigma$ being the hyperelliptic involution:

   \[
  \scalebox{0.8}
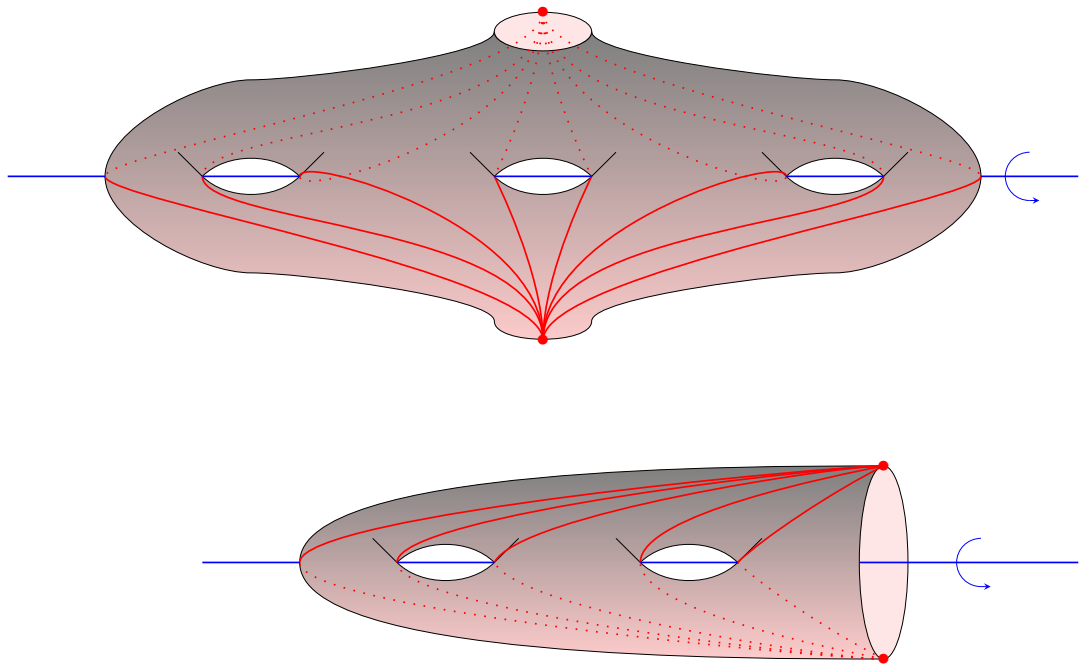
{
  \begin{tikzpicture}[scale=0.8, >=stealth]

\shadedraw[bottom color=red!20] (-9,0)..controls (-9,1) and (-7,2)..(-6,2)..controls (-5,2) and (-1,2.5)..(-1,3).. controls (-1,2.5) and (1,2.5)..(1,3).. controls (1,2.5) and (5,2).. (6,2).. controls (7,2) and (9,1).. (9,0).. controls (9,-1) and (7,-2).. (6,-2).. controls (5,-2) and (1,-2.5).. (1,-3).. controls (1,-3.5) and (-1,-3.5).. (-1,-3).. controls (-1,-2.5) and (-5,-2).. (-6,-2).. controls (-7,-2) and (-9,-1).. (-9,0);  

\draw[fill=red!10] (0,3) circle (1 and 0.4);

\draw[fill=white] (-7,0)..controls (-6.5,-0.5) and (-5.5,-0.5).. (-5,0).. controls (-5.5,0.5) and (-6.5,0.5).. (-7,0);

\draw (-7,0)--(-7.5,0.5);
\draw (-5,0)--(-4.5,0.5);

 \draw[fill=white] (7,0)..controls (6.5,-0.5) and (5.5,-0.5).. (5,0).. controls (5.5,0.5) and (6.5,0.5).. (7,0); 
 \draw (7,0)--(7.5,0.5);
\draw (5,0)--(4.5,0.5);

 \draw[fill=white] (-1,0)..controls (-0.5,-0.5) and (0.5,-0.5).. (1,0).. controls (0.5,0.5) and (-0.5,0.5).. (-1,0); 
 \draw (1,0)--(1.5,0.5);
\draw (-1,0)--(-1.5,0.5);

\node at (0,3.4) {$\rpoint$};

\node at (0,-3.4) {$\rpoint$};

\draw[red, thick] (0,-3.4).. controls (0,-1) and (-5,0.5)..(-5,0);
\draw[red, thick] (0,-3.4).. controls (0,-1) and (5,0.5)..(5,0);
 \draw[red, thick] (0,-3.4).. controls (0,-1) and (-7,-1)..(-7,0); 
  \draw[red, thick] (0,-3.4).. controls (0,-1) and (7,-1)..(7,0); 
  
   \draw[red, thick] (0,-3.4).. controls (0,-2) and (-9,-0.5)..(-9,0); 
      \draw[red, thick] (0,-3.4).. controls (0,-2) and (9,-0.5)..(9,0); 
\draw[red, thick] (0,-3.4).. controls (0,-2) and (-1,0).. (-1,0);      
      \draw[red, thick] (0,-3.4).. controls (0,-2) and (1,0).. (1,0);  
      
      \draw[red, thick, loosely dotted] (0,3.4).. controls (0,1) and (-5,-0.5)..(-5,0);
\draw[red, thick, loosely dotted] (0,3.4).. controls (0,1) and (5,-0.5)..(5,0);
 \draw[red, thick, loosely dotted] (0,3.4).. controls (0,1) and (-7,1)..(-7,0); 
  \draw[red, thick, loosely dotted] (0,3.4).. controls (0,1) and (7,1)..(7,0); 
  
   \draw[red, thick, loosely dotted] (0,3.4).. controls (0,2) and (-9,0.5)..(-9,0); 
      \draw[red, thick, loosely dotted] (0,3.4).. controls (0,2) and (9,0.5)..(9,0); 
\draw[red, thick, loosely dotted] (0,3.4).. controls (0,2) and (-1,0).. (-1,0);      
      \draw[red, thick, loosely dotted] (0,3.4).. controls (0,2) and (1,0).. (1,0);    
      
\draw[thick, blue] (-11,0)--(-9,0);
\draw[thick, blue] (-7,0)--(-5,0);
\draw[thick, blue] (5,0)--(7,0);
\draw[thick, blue] (9,0)--(11,0);
\draw[thick, blue] (-1,0)--(1,0);

\draw[blue] (10,0.5) arc (90:270:0.5);
\draw[->, blue] (10,-0.5)--(10.2,-0.5);

\begin{scope}[yshift=-8cm, xshift=-5cm]

\shadedraw[bottom color=red!20] (0,0)..controls (0,2) and ( 9,2).. (12,2).. controls (11.5,2) and (11.5,-2).. (12,-2).. controls (9,-2) and (0,-2).. (0,0);
\draw[fill=red!10] (12,0) circle (0.5 and 2);

\draw[fill=white] (2,0)..controls (2.5,-0.5) and (3.5,-0.5).. (4,0).. controls (3.5,0.5) and (2.5,0.5).. (2,0);

\draw (2,0)--(1.5,0.5);
\draw (4,0)--(4.5,0.5);

\draw[fill=white] (7,0)..controls (7.5,-0.5) and (8.5,-0.5).. (9,0).. controls (8.5,0.5) and (7.5,0.5).. (7,0);

\draw (7,0)--(6.5,0.5);
\draw (9,0)--(9.5,0.5);

\node at (12,2) {$\rpoint$};
\node at (12,-2) {$\rpoint$};

\draw[red, thick] (12,2).. controls (11,1.5) and (9.5,0.5)..(9,0);
\draw[red, thick] (12,2).. controls (11,1.8) and (7,1)..(7,0);
\draw[red, thick] (12,2).. controls (11,2) and (4.5,1)..(4,0);
\draw[red, thick] (12,2).. controls (11,2) and (2,1)..(2,0);
\draw[red, thick] (12,2).. controls (11,2) and (0,1)..(0,0);

\draw[red, thick, loosely dotted] (12,-2).. controls (11,-1.5) and (9.5,-0.5)..(9,0);
\draw[red, thick, loosely dotted] (12,-2).. controls (11,-1.8) and (7,-1)..(7,0);
\draw[red, thick, loosely dotted] (12,-2).. controls (11,-2) and (4.5,-1)..(4,0);
\draw[red, thick, loosely dotted] (12,-2).. controls (11,-2) and (2,-1)..(2,0);
\draw[red, thick, loosely dotted] (12,-2).. controls (11,-2) and (0,-1)..(0,0);

\draw[thick, blue] (-2,0)--(0,0);
\draw[thick, blue] (2,0)--(4,0);
\draw[thick, blue] (7,0)--(9,0);
\draw[thick, blue] (11.5,0)--(16,0);

\draw[blue] (14,0.5) arc (90:270:0.5);
\draw[->, blue] (14,-0.5)--(14.2,-0.5);  
\end{scope}

  \end{tikzpicture}}\]
 
The corresponding orbifold surface is a disc with an even number of orbifold points in the interior in the first case, and with odd number in the second case. 

The corresponding gentle and skew-gentle algebras are as follows:

   \[
  \scalebox{0.8}
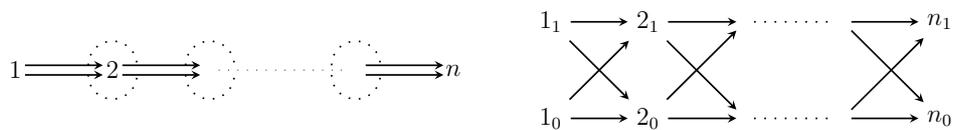
{
  \begin{tikzpicture}[scale=0.8, >=stealth]
\node (A1) at (0,0) {$1$};
\node (A2) at (2,0) {$2$};
\node (A3) at (4,0) {$ $};
\node (A4) at (7,0) {$ $};
\node (A5) at (9,0) {$n$};

\draw[thick, ->] (0.2,0.1)--(1.8,0.1);
\draw[thick, ->] (0.2,-0.1)--(1.8,-0.1);
\draw[thick, ->] (2.2,0.1)--(3.8,0.1);
\draw[thick, ->] (2.2,-0.1)--(3.8,-0.1);

\draw[thick,loosely dotted] (1.5,0.1) arc (180:0:0.5);
\draw[thick,loosely dotted] (1.5,-0.1) arc (180:360:0.5);
\draw[thick,loosely dotted] (3.5,0.1) arc (180:0:0.5);
\draw[thick,loosely dotted] (3.5,-0.1) arc (180:360:0.5);
\draw[thick,loosely dotted] (6.5,0.1) arc (180:0:0.5);
\draw[thick,loosely dotted] (6.5,-0.1) arc (180:360:0.5);

\draw[loosely dotted] (A3)--(A4);
\draw[thick, ->] (7.2,0.1)--(8.8,0.1);
\draw[thick, ->] (7.2,-0.1)--(8.8,-0.1);

\begin{scope}[xshift=11cm]
\node (10) at (0,-1) {$1_0$};
\node (11) at (0,1) {$1_1$};
\node (20) at (2,-1) {$2_0$};
\node (21) at (2,1) {$2_1$};
\node (30) at (4,-1) {$ $};
\node (31) at (4,1) {$ $};

\node (0) at (6,-1) {$ $};
\node (1) at (6,1) {$ $};
\node (n0) at (8,-1) {$n_0$};
\node (n1) at (8,1) {$n_1$};

\draw[thick, ->] (10)--(20);
\draw[thick, ->] (10)--(21);
\draw[thick, ->] (11)--(20);
\draw[thick, ->] (11)--(21);
\draw[thick, ->] (20)--(30);
\draw[thick, ->] (20)--(31);
\draw[thick, ->] (21)--(30);
\draw[thick, ->] (21)--(31);

\draw[thick, loosely dotted] (31)--(1);
\draw[thick, loosely dotted] (30)--(0);

\draw[thick,->] (0)--(n0);
\draw[thick, ->] (0)--(n1);
\draw[thick, ->] (1)--(n0);
\draw[thick, ->] (1)--(n1);

\end{scope}

  \end{tikzpicture}}\]
 
\end{example}

\begin{example}\label{example1}

Let us consider the following $G$-dissected surface, where $\sigma$ is given by the central symmetry around the center of the square. It is a torus with two boundary components.

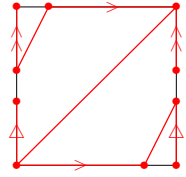
\[\scalebox{0.6}{
  \begin{tikzpicture}[scale=0.7]
  
  \draw (0,0)--(5,0)--(5,5)--(0,5)--(0,0); 
 \node at (0,0) {$\rpoint$};
 \node at (4,0) {$\rpoint$};
 \node at (5,0) {$\rpoint$};
 \node at (5,2) {$\rpoint$};
 \node at (5,3) {$\rpoint$};
 \node at (5,5) {$\rpoint$};
 \node at (1,5) {$\rpoint$};
 \node at (0,5) {$\rpoint$};
 \node at (0,3) {$\rpoint$};
 \node at (0,2) {$\rpoint$};

 \draw[thick, red] (0,5)--(0,3)--(1,5)--(5,5)--(0,0)--(4,0)--(5,2)--(5,0);
 \draw[thick,red] (0,0)--(0,2);
 \draw[thick, red] (5,3)--(5,5);
 
 \node[red] at (2,0) {$>$};
 \node[red] at (3,5) {$>$};
 \node[red] at (0,1) {$\triangle$};
 \node[red] at (5,1) {$\triangle$};
 \node[red, rotate=90] at (0,4) {$>>$};
 \node[red, rotate=90] at (5,4) {$>>$};

 \end{tikzpicture}}\] 
 
 One immediately sees that $\sigma$ has two fixed points, marked here by a $\cross$. Let us choose some representative in each orbit of the arcs of $D$, and of each marked point.
 
 \[\scalebox{0.6}
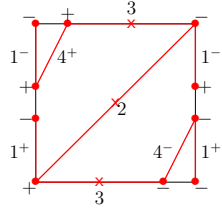
{
  \begin{tikzpicture}[scale=0.7]
  
  \draw (0,0)--(5,0)--(5,5)--(0,5)--(0,0); 
 \node at (0,0) {$\rpoint$};
 \node at (4,0) {$\rpoint$};
 \node at (5,0) {$\rpoint$};
 \node at (5,2) {$\rpoint$};
 \node at (5,3) {$\rpoint$};
 \node at (5,5) {$\rpoint$};
 \node at (1,5) {$\rpoint$};
 \node at (0,5) {$\rpoint$};
 \node at (0,3) {$\rpoint$};
 \node at (0,2) {$\rpoint$};
 \node at (-0.2,-0.2) {$+$};
 \node at (4,-0.2) {$-$};
 \node at (5.2,-0.2) {$-$};
 \node at (5.3,2) {$-$};
 \node at (5.3,3) {$+$};
 \node at (5.2,5.2) {$-$};
 \node at (1,5.3) {$+$};
 \node at (-0.2,5.2) {$-$};
 \node at (-0.3,3) {$+$};
 \node at (-0.3,2) {$-$};

 \draw[thick, red] (0,5)--(0,3)--(1,5)--(5,5)--(0,0)--(4,0)--(5,2)--(5,0);
 \draw[thick,red] (0,0)--(0,2);
 \draw[thick, red] (5,3)--(5,5);
 
 \node at (2,0) {$\cross$};
 \node at (2.5,2.5) {$\cross$};
 \node at (3,5) {$\cross$};
 \node at (2,-0.5) {$3$};
 \node at (3,5.5) {$3$};
 \node at (-0.5,1) {$1^+$};
 \node at (5.5,1) {$1^+$};
 \node at (-0.5,4) {$1^-$};
 \node at (5.5,4) {$1^-$};
 \node at (2.7,2.3) {$2$};
 \node at (1,4) {$4^+$};
 \node at (4,1) {$4^-$};

 \end{tikzpicture}}\] 
 
 The associated gentle pair $(Q,I)$ is as follows:
 
 \[\scalebox{0.6}{
  \begin{tikzpicture}[scale=0.7,>=stealth]
  \node at (-1,0) {$Q=$};
 \node (1+) at (0,2) {$1^+$};
 \node (1-) at (0,-2) {$1^-$};
 \node (2) at (2,0) {$2$};
 \node (3) at (5,0) {$3$};
 \node (4+) at (7,-2) {$4^+$};
 \node (4-) at (7,2) {$4^-$};
 \draw[thick,->] (1+)--node [fill=white, inner sep=0pt]{$a^+$} (2);
 \draw[thick, ->] (1-)--node [fill=white, inner sep=0pt]{$a^-$}(2);
 \draw[thick,->] (2.3,0.2)--node [fill=white, inner sep=0pt]{$b^+$} (4.7,0.2);
 \draw[thick,->] (2.3,-0.2)--node [fill=white, inner sep=0pt, xshift=0.2cm]{$b^-$} (4.7,-0.2);
 \draw[thick,->] (3)--node [fill=white, inner sep=0pt, xshift=-0.2cm]{$c^+$} (4+);
 \draw[thick, ->] (3)--node [fill=white, inner sep=0pt, xshift=0.1cm]{$c^-$} (4-);
 \draw[thick, ->] (1+)--node [fill=white, inner sep=0pt]{$d^+$} (4-);
 \draw[thick, ->] (1-)--node [fill=white, inner sep=0pt]{$d^-$} (4+);
 
 \node at (12,0) {$I=\{b^+a^+, b^-a^-,c^+b^+,c^-b^-\}$};

  \end{tikzpicture}}\] 
  
  Note that in this example, it is not possible to choose representatives of the orbits so that no $+$ labeled arrow  has a $-$ labeled start or end vertex.

 The $\cross$-dissected orbifold corresponding to $(\surf,\sigma,D)$ is a cylinder with two orbifold points.
 
 \[\scalebox{0.8}
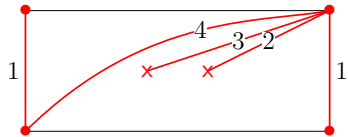
{
  \begin{tikzpicture}[scale=1,>=stealth]
 
 \draw (0,0)--(5,0)--(5,2)--(0,2)--(0,0);
\node at (0,0) {$\rpoint$};
\node at (5,0)  {$\rpoint$};
\node at (5,2) {$\rpoint$};
\node at (0,2) {$\rpoint$};
\node at (2,1) {$\cross$};
\node at (3,1) {$\cross$};

\draw[thick, red] (0,0)--node [black, xshift=-0.2cm]{$1$} (0,2);
\draw[thick, red] (5,0)--node [black, xshift=0.2cm]{$1$} (5,2);
\draw[thick, red] (5,2)--node [black,fill=white, inner sep=0pt]{$3$} (2,1);
\draw[thick, red] (5,2)--node [black,fill=white, inner sep=0pt]{$2$} (3,1);
\draw[thick, red] (0,0)..controls (2,2) and (4,1.8)..node [black,fill=white, inner sep=0pt]{$4$} (5,2);

\end{tikzpicture}}\]

The corresponding skew-gentle algebra is given by the following skew-gentle triple 

 \[\scalebox{0.8}{
  \begin{tikzpicture}[scale=1.5,>=stealth]
\node at (0,0) {$Q=$};
\node (1) at (1,0) {$1$};
\node  (2) at (2,1) {$2$};
\node (3) at (4,1) {$3$};
\node (4) at (5,0) {$4$};

\draw[->] (1)--node [fill=white, inner sep=0pt]{$a$}(2);
\draw[->] (2)--node [fill=white, inner sep=0pt]{$b$}(3);
\draw[->] (3)--node [fill=white, inner sep=0pt]{$c$}(4);
\draw[->] (1)--node [fill=white, inner sep=0pt]{$d$}(4);

\draw[->] (1.9,1)..controls (1.5,1.5) and (2.5,1.5)..node [fill=white, inner sep=0pt]{$\epsilon_2$}(2.1,1);
\draw[->] (3.9,1)..controls (3.5,1.5) and (4.5,1.5)..node [fill=white, inner sep=0pt]{$\epsilon_3$}(4.1,1);

\node at (7,0) {$\bar{I}=\{ba, cb\}$};
\node at (7,1) {${\rm Sp}=\{\epsilon_2,\epsilon_3\}$};
  
  \end{tikzpicture}}\]
  
Therefore the algebra $\bar{A}$ is given by the following quiver with relations 

 \[\scalebox{0.8}{
  \begin{tikzpicture}[scale=1.5,>=stealth]
\node at (0,0) {$\bar{Q}=$};
\node (1) at (1,0) {$1$};
\node  (20) at (2,1) {$2_0$};
\node (30) at (4,1) {$3_0$};
\node  (21) at (2,-1) {$2_1$};
\node (31) at (4,-1) {$3_1$};
\node (4) at (5,0) {$4$};

\draw[->] (1)--node [fill=white, inner sep=0pt]{$^0a$}(20);
\draw[->] (1)--node [fill=white, inner sep=0pt]{$^1a$}(21);
\draw[->] (20)--node [fill=white, inner sep=0pt]{$^0b^0$}(30);
\draw[->] (20)--node [fill=white, inner sep=0pt, xshift=0.5cm,yshift=-0.5cm]{$^1b^0$}(31);
\draw[->] (21)--node [fill=white, inner sep=0pt,xshift=0.5cm, yshift=0.5cm]{$^0b^1$}(30);
\draw[->] (21)--node [fill=white, inner sep=0pt]{$^1b^1$}(31);
\draw[->] (30)--node [fill=white, inner sep=0pt]{$c^0$}(4);
\draw[->] (31)--node [fill=white, inner sep=0pt]{$c^1$}(4);
\draw[->] (1)--node [fill=white, inner sep=0pt, xshift=-1cm]{$d$}(4);

\node at (3,-2) {$\bar{I}=\{(^0b^0)(^0a)+(^0b^1)(^1a), (^1b^0)(^0a)+(^1b^1)(^1a),(c^0)(^0b^0)+(c^1)(^1b^0), (c^0)(^0b^1)+(c^1)(^1b^1)\}$};

  \end{tikzpicture}}\]

\end{example}

\subsection{Construction of a cover from a $\cross$-dissection}

Given a skew-gentle algebra $\bar{A}$ associated to a $\cross$-dissection $D$, there is a natural action of $\dual=\{1,\chi\}$ on its quiver  $\bar{Q}(D)$ defined as follows:

\begin{enumerate}
\item it fixes all vertices corresponding to $\rpoint$-arcs;
\item it fixes all arrows between two vertices corresponding to $\rpoint$-arcs;
\item for each $\cross$-arc, it switches the two vertices corresponding to it, 
\item it switches accordingly the arrows with at least one vertex attached to a $\cross$-arc.
\end{enumerate}
This action clearly induces an action on the skew-gentle algebra $\bar{A}$.
It is known from \cite{GePe} that the skew-group algebra $\bar{A}\dual$ with such an action is (Morita equivalent to) a gentle algebra. The next result relates geometrically the two corresponding dissected surfaces.

\begin{theorem}\label{thm::construction covering}
Let $(\surf, M,P,X,D)$ be a $\cross$-dissected surface and let $\bar{A}=\bar{A}(\surf,M,P,X,D)$ be the corresponding skew-gentle algebra. Then there exists a $G$-marked surface $(\widetilde{\surf},\widetilde{M},\widetilde{P},\sigma)$ such that:
\begin{enumerate}
\item there exists a $2$-folded cover $p:\widetilde{\surf}\to \surf$ branched in the points in $X$ that induces a diffeomorphism $(\widetilde{\surf}\setminus \widetilde{X})/\sigma\to \surf\setminus X$ where $\widetilde{X}=p^{-1}(X)$ are the points fixed by $\sigma$;
\item $\widetilde{D}:=p^{-1}(D)$ is a $G$-dissection of $(\widetilde{\surf}, \widetilde{M}, \widetilde{P},\sigma)$;
\item there is a $\dual$-isomorphism $\eta(A(\widetilde{D})G)\eta\simeq \bar{A}$, where  $\eta$ is a $\dual$-invariant idempotent of $A(D) G$;
\item there is a $G$-isomorphism $\bar{\eta}(\bar{A}\dual)\bar{\eta}\simeq A(\widetilde{D})$, where $\bar{\eta}$ is a $G$-invariant idempotent of $\bar{A} \dual$.
\end{enumerate}
\end{theorem}

\begin{proof}
The construction of the double cover $\widetilde{\surf}$ of $\surf$ is  similar to the construction in \cite[Sections 3.2 and 3.3]{AP}: The $\cross$-dissection cuts the surface $\surf$ into polygons with exactly one side being a boundary segment. Fix a point on each boundary segment that we denote by a green $\gpoint$.  Enumerate the orbifold points by $X= \{X_1, \ldots, X_f \}$. In each polygon containing at least one $\cross$ on its boundary, draw curves $\gamma_i$ from the green point $\gpoint$ to each $X_i$ on its boundary so that the $\gamma_i$'s do not intersect and stay in the interior of the polygon. 

In a first step, we cut the surface $\surf$ along all the curves $\gamma_i$ (see picture below). We obtain a surface $\surf^+$ which is connected since each $\cross$  is adjacent to exactly one arc, hence it is in the boundary of exactly one polygon. In $\surf^+$ the curves $\gamma_i$ are now boundary segments $[P_i^+,Q_i^+]$ containing $X_i$. We take another copy of this new surface, that we call $\surf^-$. The surface $\widetilde{S}$ is defined as the quotient $\surf^+\cup \surf^-/( \Psi_i)$ where $\Psi_i$ is a diffeomorphism sending $[P_{i}^+,Q_{i}^+]$ to $[Q_{i}^-,P_{i}^-]$ (given in the picture below by identifying parallel green sides). Then by an argument similar as Theorem 3.5 in \cite{AP}, the surface $\widetilde{\surf}$ is an oriented smooth surface with boundary. Moreover the diffeomorphism $\sigma:\surf^+\to \surf^-$ induces a diffeomorphism of order $2$ on $\widetilde{S}$ whose fixed points are exactly the $X_i^+=X_i^-$'s.

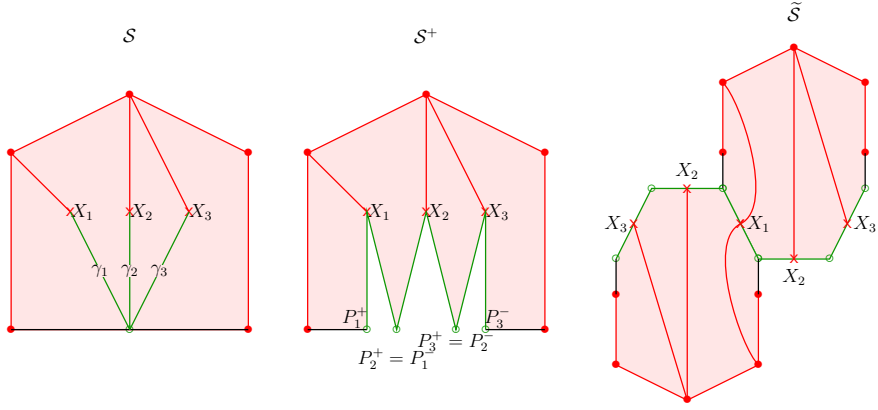
 \[\scalebox{0.6}{
  \begin{tikzpicture}[scale=1.3,>=stealth]
  
  \node at (2,5) {$\surf$};
  \draw[thick, red, fill=red!10] (0,0)--(0,3)--(2,4)--(4,3)--(4,0)--(0,0);
\node at (0,0) {$\rpoint$};
\node at (0,3) {$\rpoint$};
\node at (2,4) {$\rpoint$};
\node at (4,3) {$\rpoint$};
\node at (4,0) {$\rpoint$};
\node at (2,0) {$\gpoint$};
\node at (1,2) {$\cross$};
\node at (2,2) {$\cross$};
\node at (3,2) {$\cross$};
\node at (2.2,2) {$X_2$};
\node at (1.2,2) {$X_1$};
\node at (3.2,2) {$X_3$};

\draw[thick, red] (0,3)--(1,2);
\draw[thick, red] (2,2)--(2,4)--(3,2);

\draw[thick] (0,0)--(4,0);
\draw[thick, dark-green] (2,0)--node [black, fill=red!10, inner sep=0pt]{$\gamma_1$}(1,2);
\draw[thick, dark-green] (2,0)--node [black, fill=red!10, inner sep=0pt]{$\gamma_2$}(2,2);
\draw[thick, dark-green] (2,0)--node [black, fill=red!10, inner sep=0pt]{$\gamma_3$}(3,2);

\begin{scope}[xshift=5cm]

\node at (2,5) {$\surf^+$};

 \draw[white, fill=red!10] (0,0)--(0,3)--(2,4)--(4,3)--(4,0)--(3,0)--(3,2)--(2.5,0)--(2,2)--(1.5,0)--(1,2)--(1,0)--(0,0);
 
\node at (0,0) {$\rpoint$};
\node at (0,3) {$\rpoint$};
\node at (2,4) {$\rpoint$};
\node at (4,3) {$\rpoint$};
\node at (4,0) {$\rpoint$};
\node at (1,0) {$\gpoint$};
\node at (1.5,0) {$\gpoint$};
\node at (2.5,0) {$\gpoint$};
\node at (3,0) {$\gpoint$};

\node at (2.2,2) {$X_2$};
\node at (1.2,2) {$X_1$};
\node at (3.2,2) {$X_3$};

\node at (1,2) {$\cross$};
\node at (2,2) {$\cross$};
\node at (3,2) {$\cross$};

\draw[thick, red] (0,0)--(0,3)--(2,4)--(4,3)--(4,0);

\draw[thick, red] (0,3)--(1,2);
\draw[thick, red] (2,2)--(2,4)--(3,2);

\draw[thick] (0,0)--(1,0);
\draw[thick] (3,0)--(4,0);

\draw[thick, dark-green] (1,0)--(1,2)--(1.5,0)--(2,2)--(2.5,0)--(3,2)--(3,0);

\node at (0.8,0.2) {$P_1^+$};
\node at (1.5,-0.5) {$P_2^+=P_1^-$};
\node at (2.5,-0.2) {$P_3^+=P_2^-$};
\node at (3.2,0.2) {$P_3^-$};

\end{scope}

\begin{scope}[xshift=12cm, yshift=3cm,scale=0.6]

\node at (2,4) {$\widetilde{\surf}$};
\draw[white, fill=red!10] (0,-1)--(0,2)--(2,3)--(4,2)--(4,-1)--(3,-3)--(1,-3)--(0,-1);

\node at (0,0) {$\rpoint$};
\node at (0,2) {$\rpoint$};
\node at (2,3) {$\rpoint$};
\node at (4,2) {$\rpoint$};
\node at (4,0) {$\rpoint$};

\node at (0,-1) {$\gpoint$};
\node at (1,-3) {$\gpoint$};
\node at (3,-3) {$\gpoint$};
\node at (4,-1) {$\gpoint$};

\draw[thick, red] (0,0)--(0,2)--(2,3)--(4,2)--(4,0);
\draw[thick,dark-green](0,-1)--(1,-3)--(3,-3)--(4,-1);
\draw[thick] (0,0)--(0,-1);
\draw[thick] (4,0)--(4,-1);

\node at (0.5,-2) {$\cross$};
\node at (2,-3) {$\cross$};
\node at (3.5,-2) {$\cross$};
\node at (1,-2) {$X_1$};
\node at (2,-3.5) {$X_2$};
\node at (4,-2) {$X_3$};

\draw[red, thick] (0.5,-2)..controls (1.5,-1.5) and (0.5,1.5)..(0,2);
\draw[red,thick] (2,-3)--(2,3);
\draw[red, thick] (3.5,-2)--(2,3);

\begin{scope}[ xshift=1cm, yshift=-4cm, rotate=180]
\draw[white, fill=red!10] (0,-1)--(0,2)--(2,3)--(4,2)--(4,-1)--(3,-3)--(1,-3)--(0,-1);

\node at (0,0) {$\rpoint$};
\node at (0,2) {$\rpoint$};
\node at (2,3) {$\rpoint$};
\node at (4,2) {$\rpoint$};
\node at (4,0) {$\rpoint$};

\node at (0,-1) {$\gpoint$};
\node at (1,-3) {$\gpoint$};
\node at (3,-3) {$\gpoint$};
\node at (4,-1) {$\gpoint$};

\draw[thick, red] (0,0)--(0,2)--(2,3)--(4,2)--(4,0);
\draw[thick,dark-green](0,-1)--(1,-3)--(3,-3)--(4,-1);
\draw[thick] (0,0)--(0,-1);
\draw[thick] (4,0)--(4,-1);

\node at (0.5,-2) {$\cross$};
\node at (2,-3) {$\cross$};
\node at (3.5,-2) {$\cross$};
\node at (2,-3.5) {$X_2$};
\node at (4,-2) {$X_3$};

\draw[red, thick] (0.5,-2)..controls (1.5,-1.5) and (0.5,1.5)..(0,2);
\draw[red,thick] (2,-3)--(2,3);
\draw[red, thick] (3.5,-2)--(2,3);

\end{scope}
 \end{scope}   
  
  \end{tikzpicture}}\]

\medskip

We now prove that $p^-{1}(D)$ is a $G$-dissection.  First note that the $\rpoint$-arcs cut the surface $\surf^+$ into polygons, each of which has exactly one boundary side which is the concatenation of one half of a boundary segment, several green segments, and one half boundary segment (see picture above). If $n$ is the number of $\cross$'s in its boundary, then this polygon is cut into $n+1$-polygons $\{\mathcal{P}_0^+,\ldots,\mathcal{P}_n^+\}$ by the $\cross$-arcs.  The polygons $\mathcal{P}_i^+$ contain exactly one boundary segment. After gluing $\surf^+$ with $\surf^-$ along the green boundaries, we obtain that $\mathcal{P}_0^+$ is glued to $\mathcal{P}_1^-$ along one green boundary, $\mathcal{P}_i^+$ is glued to $\mathcal{P}_{i-1}^-$ along exactly one green segment, and to $\mathcal{P}_{i+1}^-$ along the other. Finally we obtain that the red arcs cut the surface $\widetilde{S}$ into polygons of the form $$\mathcal{P}_0^+\cup\mathcal{P}_1^-\cup \mathcal{P}_2^+\cup \ldots \mathcal{P}_n^{\pm}$$
which are polygons that contain exactly one boundary segment which is the gluing of the boundary of $\mathcal{P}_0^+$ with the boundary of $\mathcal{P}_n^{\pm}$ (see picture below). 

 \[\scalebox{0.6}
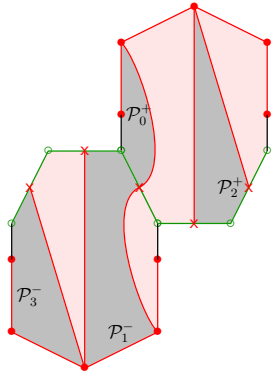
{
  \begin{tikzpicture}[scale=0.8,>=stealth]

\draw[white, fill=red!10] (0,-1)--(0,2)--(2,3)--(4,2)--(4,-1)--(3,-3)--(1,-3)--(0,-1);

\draw[white, fill=gray!50] (0,-1)--(0,2)..controls (0.5,1.5) and (1.5,-1.5)..(0.5,-2)--(0,-1);
\node at (0.5,0) {$\mathcal{P}^+_0$};
\draw[white, fill=gray!50] (2,-3)--(2,3)--(3.5,-2)--(3,-3)--(2,-3);
\node at (3,-2) {$\mathcal{P}^+_2$};

\node at (0,0) {$\rpoint$};
\node at (0,2) {$\rpoint$};
\node at (2,3) {$\rpoint$};
\node at (4,2) {$\rpoint$};
\node at (4,0) {$\rpoint$};

\node at (0,-1) {$\gpoint$};
\node at (1,-3) {$\gpoint$};
\node at (3,-3) {$\gpoint$};
\node at (4,-1) {$\gpoint$};

\draw[thick, red] (0,0)--(0,2)--(2,3)--(4,2)--(4,0);
\draw[thick,dark-green](0,-1)--(1,-3)--(3,-3)--(4,-1);
\draw[thick] (0,0)--(0,-1);
\draw[thick] (4,0)--(4,-1);

\node at (0.5,-2) {$\cross$};
\node at (2,-3) {$\cross$};
\node at (3.5,-2) {$\cross$};

\draw[red, thick] (0.5,-2)..controls (1.5,-1.5) and (0.5,1.5)..(0,2);
\draw[red,thick] (2,-3)--(2,3);
\draw[red, thick] (3.5,-2)--(2,3);

\begin{scope}[ xshift=1cm, yshift=-4cm, rotate=180]
\draw[white, fill=red!10] (0,-1)--(0,2)--(2,3)--(4,2)--(4,-1)--(3,-3)--(1,-3)--(0,-1);

\draw[white, fill=gray!50] (0.5,-2)--(1,-3)--(2,-3)--(2,3)--(0,2)..controls (0.5,1.5) and (1.5,-1.5)..(0.5,-2);

\node at (1,2) {$\mathcal{P}^-_1$};

\draw[white, fill=gray!50] (2,3)--(4,2)--(4,-1)--(3.5,-2)--(2,3);

\node at (3.5,1) {$\mathcal{P}^-_3$};

\node at (0,0) {$\rpoint$};
\node at (0,2) {$\rpoint$};
\node at (2,3) {$\rpoint$};
\node at (4,2) {$\rpoint$};
\node at (4,0) {$\rpoint$};

\node at (0,-1) {$\gpoint$};
\node at (1,-3) {$\gpoint$};
\node at (3,-3) {$\gpoint$};
\node at (4,-1) {$\gpoint$};

\draw[thick, red] (0,0)--(0,2)--(2,3)--(4,2)--(4,0);
\draw[thick,dark-green](0,-1)--(1,-3)--(3,-3)--(4,-1);
\draw[thick] (0,0)--(0,-1);
\draw[thick] (4,0)--(4,-1);

\node at (0.5,-2) {$\cross$};
\node at (2,-3) {$\cross$};
\node at (3.5,-2) {$\cross$};

\draw[red, thick] (0.5,-2)..controls (1.5,-1.5) and (0.5,1.5)..(0,2);
\draw[red,thick] (2,-3)--(2,3);
\draw[red, thick] (3.5,-2)--(2,3);

\end{scope}

  \end{tikzpicture}}\]

\medskip

To prove assertion (3), we apply Proposition~\ref{prop::iso AG-Abar} for a particular choice of idempotent (that is a particular choice of orbits) that comes from the construction of $\widetilde{S}$. The only thing to check is that the isomorphism constructed in the proof of Proposition~\ref{prop::iso AG-Abar} is a $\dual$-isomorphism for this particular choice of idempotent $\eta$. 

Given a point in $\widetilde{M}$ or in $\widetilde{P}$, it is either in $\surf^+$ or in $\surf^-$ but not on both. Therefore, we choose the orbits $\widetilde{M}^+\cup \widetilde{M}^-$ accordingly. Now if an arc in $\widetilde{D}$  is not fixed by $\sigma$, then it is either entirely in $\surf^+$, or entirely in $\surf^-$. Hence, there is a natural partition $Q_0(\widetilde{D})=Q_0^+\cup Q_0^-\cup Q_0^{\rm fix}$. With this choice of orbits, we have the following property (that may fail for any other choice of orbits, see example \ref{example1}):

\emph{For each arrow $\alpha^+\in Q_1(\widetilde{D}^+)=Q_1^+$, the source and the target of $\alpha^+$ are in $Q_0^+\cup Q_0^{\rm fix}$.}

 Hence in this special setup, the map $\Phi:\bar{A}(D)\to \eta A(\widetilde{D}) G \eta$ defined in the proof of Proposition \ref{prop::iso AG-Abar} becomes: 
\[ \begin{array}{rcll}
\Phi (\bar{e}_i) &= &e_{i^+}\smallten 1 & \\
\Phi(\bar{e}_{i_{\epsilon}}) &=&e_i\smallten \frac{1+(-1)^\epsilon \sigma}{2}   & \\
\Phi(\alpha) &=&\alpha^+\smallten 1 & {\rm for}\ \alpha:i\to j ;\\
\Phi(\alpha^{\epsilon} ) &= &\alpha^+\smallten \frac{1+(-1)^\epsilon \sigma}{2} & \textrm{ for }\alpha^{\epsilon}  :i_\epsilon\to j\\

\Phi(^{\epsilon} \alpha ) &=&\frac{1}{2}(\alpha^+\smallten 1+(-1)^{\epsilon}\alpha^-\smallten \sigma)\small & \textrm{ for }^{\epsilon} \alpha :i\to j_\epsilon\\

\Phi(^{\epsilon'}\alpha^{\epsilon} ) &=&(e_{j}\smallten \frac{1+(-1)^{\epsilon'} \sigma}{2})(\alpha^+\smallten 1)(e_{i}\smallten\frac{1+(-1)^{\epsilon}\sigma}{2} ) & \textrm{ for }^{\epsilon'}\alpha^\epsilon :i_{\epsilon}\to j_{\epsilon'}
\end{array}\]

Recall that the action of $\chi\in \dual$ on the skew group algebra $\Lambda G$ is given by $\chi(\lambda\smallten g):=\chi(g) \lambda\smallten g$. Hence the idempotent 
\[\eta:= \sum_{i^+\in Q_0^+}e_{i^+}\smallten 1 + \sum_{j\in Q_0^{\rm fix}}e_j\smallten 1\]
is $\dual$-invariant. And one immediately checks that $\chi$ acts as follows on the quiver of $\eta A(\widetilde{D})G\eta$:

\[ \begin{array}{rcll}
\chi (e_i) &= &e_i &\\
\chi(e_{i_{\epsilon}}) &=&e_{i_{\epsilon+1}}   & \textrm{ for }\epsilon\in\mathbb Z_2 \\
\chi(\alpha) &=&\alpha & \textrm{  for}\ \alpha:i\to j ;\\
\chi(^{\epsilon} \alpha) &=& \ ^{\epsilon+1}\alpha & \textrm{ for }^\epsilon\alpha :i\to j_{\epsilon}\\
\chi(\alpha^{\epsilon} ) &=&\alpha^{\epsilon+1}& \textrm{ for }\alpha^\epsilon :i_{\epsilon}\to j\\

\chi(^{\epsilon'}\alpha^{\epsilon} ) &=&\ ^{\epsilon'+1}\alpha^{\epsilon+1} & \textrm{ for }^{\epsilon'}\alpha^\epsilon :i_{\epsilon}\to j_{\epsilon'}
\end{array}\]

Therefore, the isomorphism $\Phi$ constructed in Proposition \ref{prop::iso AG-Abar} is a $\dual$-isomorphism for this special choice of orbits.

\medskip

Combining (3) with Proposition \ref{prop::isomorphismRR}, we know that the  algebras $(\bar{A}\dual)_{\rm b}$ and $A(\widetilde{D})$ are isomorphic. But to prove (4) we need here a $G$-isomorphism. We construct it explicitly, defining $\bar{\eta}\in \bar{A}\dual$ as the following idempotent:
$$\bar{\eta}:=\sum_{i, \rpoint-{\rm arc}}\bar{e}_i\smallten 1+ \sum_{j,  \cross-{\rm arc}}\bar{e}_{j_0}\smallten 1. $$ We now construct a  morphism $\Psi: kQ(\widetilde{D})\to \bar{\eta} \bar{A}\dual\bar{\eta}$ as follows:

\[\begin{array}{rcll}
\Psi(e_{i^\pm}) &= & \bar{e}_i\smallten \frac{1\pm \chi}{2} &\\
\Psi (e_i) & = & \bar{e}_{i_0}\smallten 1 & \\
\Psi(\alpha^\pm) &=& \bar{\alpha} \smallten \frac{1\pm\chi}{2} & \textrm{for }\alpha^\pm:i^\pm\to j^\pm\\
\Psi(\alpha^\pm) &=& ^0\bar{\alpha}\smallten \frac{1\pm\chi}{2} & \textrm{for }\alpha^\pm:i^\pm\to j\\
\Psi( \alpha^\pm) &=& \frac{1}{2}(\bar{\alpha}^0\smallten 1\pm \bar{\alpha}^1\smallten \chi) & \textrm{for }\alpha^\pm:i\to j^\pm\\
\Psi(\alpha^{\pm}) & = & \frac{1}{2} (^0\bar{\alpha}^0\smallten 1\pm ^0\bar{\alpha}^1\smallten \chi) & \textrm{for }\alpha^{\pm}:i\to j.
\end{array}\]

Checking the relations is then an immediate computation. For example, assume that $\alpha^\pm:i\to j$ and $\beta^\pm:j\to k$ are double arrows in $A(\widetilde{D})$, then we compute 
\begin{align*}
\Psi(\beta^-\alpha^-)= \frac{1}{2}(^0\bar{\beta}^0\smallten 1-\ ^0\bar{\beta}^1\smallten \chi)\frac{1}{2}(^0\bar{\alpha}^0\smallten 1-\ ^0\bar{\alpha}^1\smallten \chi)\\
 =\frac{1}{4}( (^0\bar{\beta}^0 \ ^0\bar{\alpha}^0+^0\bar{\beta}^1 \ ^1\bar{\alpha}^0)\smallten 1 - ( ^0\bar{\beta}^0 \ ^0\bar{\alpha}^1+^0\bar{\beta}^1 \ ^1\bar{\alpha}^1)\smallten \chi)) =0
\end{align*}
while 
 \begin{align*}
\beta^-\alpha^+= \frac{1}{2}(^0\bar{\beta}^0\smallten 1-\ ^0\bar{\beta}^1\smallten \chi)\frac{1}{2}(^0\bar{\alpha}^0\smallten 1+\ ^0\bar{\alpha}^1\smallten \chi)\\
 =\frac{1}{4}( (^0\bar{\beta}^0 \ ^0\bar{\alpha}^0-^0\bar{\beta}^1 \ ^1\bar{\alpha}^0)\smallten 1) + (^0\bar{\beta}^0 \ ^0\bar{\alpha}^1-^1\bar{\beta}^1 \ ^1\bar{\alpha}^1)\smallten \chi)) \neq 0
\end{align*}

Therefore we obtain an isomorphism $\bar{\eta}\bar{A}\dual\bar{\eta}\simeq A(\widetilde{D})$. The action of $G$ on $\bar{A}\dual$ is given by $g(a\smallten \chi):=\chi(g) a\smallten \chi$, hence the idempotent $\bar{\eta}$ is clearly $G$-invariant. It is straightforward to check that the isomorphism constructed above commutes with the action of $G$. 
\end{proof}

A consequence of this double isomorphism (3) and (4) is the fact that we can apply Corollary \ref{cor::Geq=G-hat-eq}. Therefore we obtain the following.

\begin{corollary}\label{cor::Geq=G-hat-eq-gentle} Let $\bar{A}$ and $\bar{A}'$ be two skew-gentle algebras. Denote by $A$ and $A'$ the corresponding gentle $G$-algebras described in Theorem \ref{thm::construction covering}. Then the following are equivalent:

$$\mathcal{D}^b(\bar{A})\underset{\dual}{\sim}\mathcal{D}^b(\bar{A}')\Leftrightarrow \mathcal{D}^b(A)\underset{G}{\sim}\mathcal{D}^b(A').$$
\end{corollary}

\begin{example}
One easily checks that starting with the $\mathbb D_n$ or $\widetilde{\mathbb{D}}_n$ given in Examples \ref{example::Dn}, one obtains the cover given in Examples \ref{example::coveringDn}. More generally, if $(\surf, M,X)$ is a disc with $|M|=1$ and $|X|=n$, then for any $\cross$-dissection $D$, the corresponding $G$-cover is a surface of genus $g=[\frac{n-2}{2}]$ and with one or two boundary components depending on the parity of $n$.
\end{example}

\begin{example}\label{example::cylinder}
Let $(\surf,M,P,X,D)$ be given as in Example \ref{example1}, and $\bar{A}$ the corresponding skew-gentle algebra. 
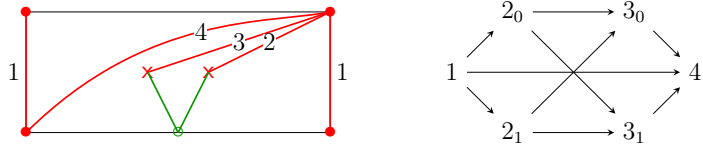
\[\scalebox{0.8}{
  \begin{tikzpicture}[scale=1,>=stealth]
 
 \draw (0,0)--(5,0)--(5,2)--(0,2)--(0,0);
\node at (0,0) {$\rpoint$};
\node at (5,0)  {$\rpoint$};
\node at (5,2) {$\rpoint$};
\node at (0,2) {$\rpoint$};
\node at (2,1) {$\cross$};
\node at (3,1) {$\cross$};

\node at (2.5,0) {$\gpoint$};

\draw[thick, red] (0,0)--node [black, xshift=-0.2cm]{$1$} (0,2);
\draw[thick, red] (5,0)--node [black, xshift=0.2cm]{$1$} (5,2);
\draw[thick, red] (5,2)--node [black,fill=white, inner sep=0pt]{$3$} (2,1);
\draw[thick, red] (5,2)--node [black,fill=white, inner sep=0pt]{$2$} (3,1);
\draw[thick, red] (0,0)..controls (2,2) and (4,1.8)..node [black,fill=white, inner sep=0pt]{$4$} (5,2);

\draw[thick, dark-green] (2.5,0)--(2,1);
\draw[thick, dark-green] (2.5,0)--(3,1);

\begin{scope}[xshift=6cm, yshift=1cm]

\node (1) at (1,0) {$1$};
\node  (20) at (2,1) {$2_0$};
\node (30) at (4,1) {$3_0$};
\node  (21) at (2,-1) {$2_1$};
\node (31) at (4,-1) {$3_1$};
\node (4) at (5,0) {$4$};

\draw[->] (1)--(20);
\draw[->] (1)--(21);
\draw[->] (20)--(30);
\draw[->] (20)--(31);
\draw[->] (21)--(30);
\draw[->] (21)--(31);
\draw[->] (30)--(4);
\draw[->] (31)--(4);
\draw[->] (1)--(4);

\end{scope}

\end{tikzpicture}}\]

Then the surface $\surf^+$ is as follows.

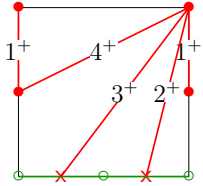
\[\scalebox{0.8}{
  \begin{tikzpicture}[scale=0.7,>=stealth]
  
  \draw (0,0)--(4,0)--(4,4)--(0,4)--(0,0);
  \node at (0,0) {$\gpoint$};
  \node at (4,0) {$\gpoint$};
  \node at (2,0) {$\gpoint$};
  \node at (0,2) {$\rpoint$};
  \node at (0,4) {$\rpoint$};
  \node at (4,4) {$\rpoint$};
  \node at (4,2) {$\rpoint$};
  \node at (1,0) {$\cross$};
  \node at (3,0) {$\cross$}; 
  
  \draw[thick, dark-green] (0,0)--(4,0);
  \draw[thick, red] (0,4)--node [black,fill=white, inner sep=0pt]{$1^+$}(0,2)--node [black,fill=white, inner sep=0pt]{$4^+$}(4,4)--node [black,fill=white, inner sep=0pt]{$1^+$}(4,2);
  \draw[thick, red] (1,0)--node [black,fill=white, inner sep=0pt]{$3^+$}(4,4);
  \draw[thick, red] (3,0)--node [black,fill=white, inner sep=0pt]{$2^+$}(4,4);
  \end{tikzpicture}}\]

Hence the dissected surface $(\widetilde{\surf},\widetilde{M},\widetilde{D})$ is as follows:

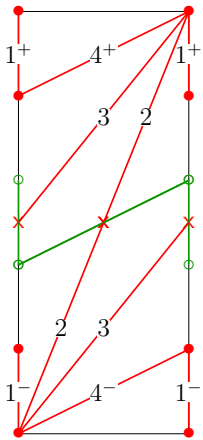
\[\scalebox{0.8}{
  \begin{tikzpicture}[scale=0.7,>=stealth]
  
  \draw (4,0)--(4,4)--(0,4)--(0,0);
  \node at (0,0) {$\gpoint$};
  \node at (4,0) {$\gpoint$};
  \node at (0,-2) {$\gpoint$};
  \node at (0,2) {$\rpoint$};
  \node at (0,4) {$\rpoint$};
  \node at (4,4) {$\rpoint$};
  \node at (4,2) {$\rpoint$};
  \node at (0,-1) {$\cross$};
  \node at (2,-1) {$\cross$}; 
  
  \draw[thick, dark-green] (0,0)--(0,-2)--(4,0);
  \draw[thick, red] (0,4)--node [black,fill=white, inner sep=0pt]{$1^+$}(0,2)--node [black,fill=white, inner sep=0pt]{$4^+$}(4,4)--node [black,fill=white, inner sep=0pt]{$1^+$}(4,2);
  \draw[thick, red] (0,-1)--node [black,fill=white, inner sep=0pt]{$3$}(4,4);
  \draw[thick, red] (2,-1)--node [black,fill=white, inner sep=0pt]{$2$}(4,4);
  
  \begin{scope}[rotate=180, xshift=-4cm, yshift=2cm]
   \draw (4,0)--(4,4)--(0,4)--(0,0);
  \node at (0,0) {$\gpoint$};
  \node at (4,0) {$\gpoint$};
  \node at (0,-2) {$\gpoint$};
  \node at (0,2) {$\rpoint$};
  \node at (0,4) {$\rpoint$};
  \node at (4,4) {$\rpoint$};
  \node at (4,2) {$\rpoint$};
  \node at (0,-1) {$\cross$};
  \node at (2,-1) {$\cross$}; 
  
  \draw[thick, dark-green] (0,0)--(0,-2)--(4,0);
  \draw[thick, red] (0,4)--node [black,fill=white, inner sep=0pt]{$1^-$}(0,2)--node [black,fill=white, inner sep=0pt]{$4^-$}(4,4)--node [black,fill=white, inner sep=0pt]{$1^-$}(4,2);
  \draw[thick, red] (0,-1)--node [black,fill=white, inner sep=0pt]{$3$}(4,4);
  \draw[thick, red] (2,-1)--node [black,fill=white, inner sep=0pt]{$2$}(4,4);
  \end{scope}
  \end{tikzpicture}}\]
  where the two external green segments are identified. One easily checks that it is a sphere with four holes. 
 
The corresponding gentle pair is given by  
 \[\scalebox{0.6}{
  \begin{tikzpicture}[scale=0.7,>=stealth]
  \node at (-1,0) {$Q=$};
 \node (1+) at (0,2) {$1^+$};
 \node (1-) at (0,-2) {$1^-$};
 \node (2) at (2,0) {$2$};
 \node (3) at (5,0) {$3$};
 \node (4+) at (7,2) {$4^+$};
 \node (4-) at (7,-2) {$4^-$};
 \draw[thick,->] (1+)--node [fill=white, inner sep=0pt]{$a^+$} (2);
 \draw[thick, ->] (1-)--node [fill=white, inner sep=0pt]{$a^-$}(2);
 \draw[thick,->] (2.3,0.2)--node [fill=white, inner sep=0pt]{$b^+$} (4.7,0.2);
 \draw[thick,->] (2.3,-0.2)--node [fill=white, inner sep=0pt, xshift=0.2cm]{$b^-$} (4.7,-0.2);
 \draw[thick,->] (3)--node [fill=white, inner sep=0pt, xshift=-0.2cm]{$c^+$} (4+);
 \draw[thick, ->] (3)--node [fill=white, inner sep=0pt]{$c^-$} (4-);
 \draw[thick, ->] (1+)--node [fill=white, inner sep=0pt]{$d^+$} (4+);
 \draw[thick, ->] (1-)--node [fill=white, inner sep=0pt]{$d^-$} (4-);
 
 \node at (12,0) {$I=\{b^+a^+, b^-a^-,c^+b^+,c^-b^-\}$};

  \end{tikzpicture}}\] 
  
  Note that here, the cover and the gentle pair are different from the one in Example \ref{example1}. 
  
 The map $\Phi:\bar{A}\to \eta AG\eta$ constructed in  Proposition \ref{prop::iso AG-Abar} sends the arrow $d$ to $(e_{4^+}\smallten 1) ((d^++d^-)\smallten (1+\sigma))(e_{1^+}\smallten 1)$. 
 In the cover of Example \ref{example1}, the arrow $d^+$ is $1^+\to 4^-$, thus we have $\Phi(d)=d^-\smallten\sigma$, while in the above example we have $\Phi(d)=d^+\smallten 1$, since $d^+:1^+\to 4^+$.
Therefore, in the cover given in Example \ref{example1}, the $\dual$-action on $\bar{A}$ induced by $\Phi$ and the action of $\dual$ on $AG$ sends $d$ to $-d$, since we have $\chi (d^-\smallten \sigma)=\chi(\sigma) (d^-\smallten \sigma)$.  
  
\end{example}

\section{Derived equivalence for skew-gentle algebras}

\subsection{Tilting objects in $\mathcal{D}^b(A)$}

In this subsection, we recall results from  \cite{APS} that are essential in this paper. We associate to any gentle algebra a marked surface with a line field on it.  The idea of associating a line field to a gentle algebra goes back to \cite{HKK} (see also \cite{LP}). Note that the line field described here is the one defined in \cite{APS}, and is slightly different from the one used in \cite{LP}.

\subsubsection{Line fields and graded arcs}
Let $(\surf,M_\rpoint,P_\rpoint, D)$ be a $\rpoint$-dissected surface, and $A$ the corresponding gentle algebra. We define a line field $\eta_D$ on $\surf\setminus (\partial S\cup P)$, that is, a section of the projectivized tangent bundle $\mathbb P(TS)\to \surf $. The line field is tangent along each arc of $D$ and is defined up to homotopy in each polygon cut out by $D$ by the following foliation:

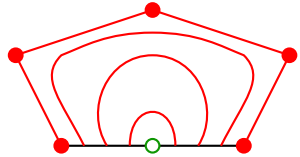
\[\scalebox{1}{
\begin{tikzpicture}[>=stealth,scale=0.6]
\draw[thick, red] (0,0)--(-1,2)--(2,3)--(5,2)--(4,0);
\draw[thick](4,0)--(0,0);
\draw[red,thick,fill=red] (0,0) circle (0.15);
\draw[red,thick,fill=red] (-1,2) circle (0.15);
\draw[red,thick,fill=red] (2,3) circle (0.15);
\draw[red,thick,fill=red] (5,2) circle (0.15);
\draw[red,thick,fill=red] (4,0) circle (0.15);
\draw[dark-green,thick,fill=white] (2,0) circle (0.15);

\draw[thick,red](0.5,0)..controls (0,1) and (-0.5,1.5)..(0,2).. controls (0.5,2.2) and (1,2.5)..(2,2.5)..controls (3,2.5) and   (3.5,2.2)..(4,2)..controls (4.5,1.5) and  (4,1)..(3.5,0);

\draw[thick,red] (1,0)..controls (0.5,0.8) and (1,2)..(2,2).. controls (3,2) and (3.5,0.8)..(3,0);

\draw[thick, red] (1.5,0)..controls (1.5,1) and (2.5,1)..(2.5,0);

\end{tikzpicture}}\]

For a smooth curve $\gamma$ intersecting transversally the line field $\eta$ at its endpoints we denote by $w_{\eta}(\gamma)$ or $w_D(\gamma)$ its winding number with respect to the line field $\eta$. It is a well defined map on the regular homotopy class of $\gamma$, see \cite{APS} for details.

We fix a finite set of green points $M_{\gpoint}$ on the boundary of $\surf$ such that each boundary segment contains exactly one point in $M_{\gpoint}$. An \emph{$\gpoint$-arc} is a curve $\gamma:[0,1]\to \surf$ such that $\gamma_{|_{(0,1)}}$ is injective and in $\surf\setminus(\partial \surf\cup P)$, and such that $\gamma(0)$ and $\gamma(1)$ belong to $M_{\gpoint}$. Arcs are considered up to isotopy fixing the endpoints. Hence each $\gpoint$-arc can be assumed to intersect minimally and transversally the $\rpoint$-dissection $D$. 

A \emph{graded $\gpoint$-arc} is a pair $(\gamma,\mathbf{n})$ where $\gamma$ is a $\gpoint$-arc, and $\mathbf{n}$ is map $\mathbf{n}:\gamma(0,1)\cap D\to \mathbb Z$ satisfying:
$$\mathbf{n}(\gamma(t_{i+1}))=\mathbf{n}(\gamma(t_i))+w_{\eta}(\gamma_{|_{[t_i,t_{i+1}]}}),$$
if $\gamma(t_i)$ and $\gamma(t_{i+1})$ are two consecutive intersections of $\gamma$ with $D$. More concretely, on $[t_i,t_{i+1}]$, the curve $\gamma$ intersects one polygon cut out by $D$, and we have 
$$\mathbf{n}(\gamma(t_{i+1}))=\mathbf{n}(\gamma(t_i))+ 1$$ if the boundary segment the polygon is on the left of the curve $\gamma_{|_{[t_i,t_{i+1}]}}$, and 
$$\mathbf{n}(\gamma(t_{i+1}))=\mathbf{n}(\gamma(t_i))+ 1$$
if the boundary segment lies on the right.

To a graded $\gpoint$-arc $(\gamma,\grading)$, one can associate an object denoted $P_{(\gamma,\grading)}$ in the category $\mathcal{D}^b(A)$. Denote by $t_1<t_2<\cdots<t_r\in (0,1)$ the parameters such that the $\gamma(t_j)$ are the intersection points of $\gamma$ with the dissection $D$. Denote by $i_1,\ldots, i_r$ the corresponding arcs of $D$. For $j=1,\ldots, r-1$ one can associate a path $p_j(\gamma)$ of the quiver $Q(D)$ as in the following picture.

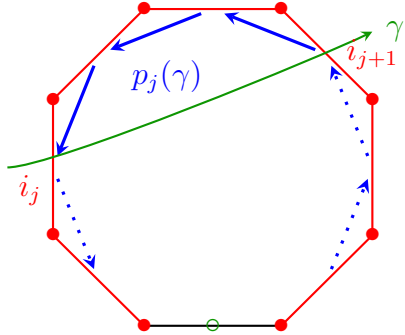
\[\scalebox{1}{
\begin{tikzpicture}[>=stealth,scale=0.6]

\draw[red, thick] (2,-2)--node (A1){} (0,0)--node (A2){}(0,3)--node (A3){}(2,5)--node (A4){}(5,5)--node (A5){}(7,3)--node (A6){}(7,0)--node (A7){}(5,-2);
\draw[thick] (2,-2)--(5,-2);
\node at (3.5,-2) {$\gpoint$};
\node at (0,0) {$\rpoint$};
\node at (0,3) {$\rpoint$};
\node at (2,5) {$\rpoint$};
\node at (5,5) {$\rpoint$};
\node at (7,3) {$\rpoint$};
\node at (7,0) {$\rpoint$};
\node at (5,-2) {$\rpoint$};
\node at (2,-2) {$\rpoint$};

\draw[->,very thick, blue,  loosely dotted] (A7)--(A6);
\draw[->,very thick, blue,  loosely dotted] (A6)--(A5);
\draw[->,very thick, blue] (A5)--(A4);
\draw[->,very thick, blue] (A4)--(A3);
\draw[->,very thick, blue] (A3)--(A2);
\draw[->,very thick, blue, loosely dotted] (A2)--(A1);

\draw[thick, dark-green,->] (-1,1.5)..controls (0,1.5) and (6,4)..(7,4.5);

\node[red] at (-0.5,1) {$i_j$};
\node[red] at (7,4) {$i_{j+1}$};

\node[blue] at (2.5,3.5) {$p_j(\gamma)$};
\node[dark-green] at (7.5,4.5) {$\gamma$};

\end{tikzpicture}}\]

As a graded $A$-module, $P_{(\gamma,\grading)}$ is defined to be $$P_{(\gamma,\grading)}:=\bigoplus_{j=1}^r e_{i_j} A[\grading(\gamma(t_j)].$$
The differential is given by the following $r\times r$ matrix $(d_{(k,\ell)})_{k,\ell}$
\begin{itemize} 
\item if $ w_{\eta}(\gamma_{|_{(t_j,t_{j+1})}})=+1$, then  $d_{(j+1,j)}=p_j(\gamma)[\grading(\gamma(t_j))]$ 

\item if $w_{\eta}(\gamma_{|_{(t_j,t_{j+1})}})=-1$, then  $d_{(j,j+1)}=p_j(\gamma)[\grading(\gamma(t_{j+1}))]$
\item all other values of  $d_{(k,\ell)}$ are $0$. 
\end{itemize}

Moreover we have $P_{(\gamma,\grading)}\simeq P_{(\gamma',\grading')}$ if and only if $\gamma=\gamma'$ (up to isotopy) and $\grading=\grading'$, or $\gamma^{-1}=\gamma'$ and $\grading=\grading'$.

\subsubsection{Tilting objects as $\gpoint$-dissections}

\begin{definition} A $\gpoint$-dissection is a collection  $\{\gamma_i, i\in I\}$ of $\gpoint$-arcs  cutting  the surface $\surf$ into polygons that have 
\begin{itemize}
\item either exactly one $\rpoint$ on its boundary and no $\rpoint$ in its interior, 
\item or no $\rpoint$ on its boundary and exactly one $\rpoint$ in its interior. 
\end{itemize}

\end{definition}

There is a duality between $\rpoint$-dissections and $\gpoint$-dissections. More precisely, for each $\rpoint$-dissection there exists a unique $\gpoint$-dissection such that each $\gpoint$-arc intersects exactly one $\rpoint$-arc and vice versa. 

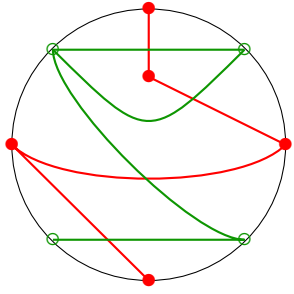
\[\scalebox{1}{
\begin{tikzpicture}[>=stealth,scale=0.6]
\draw (0,0) circle (3);

\node at (-3,0) {$\rpoint$};
\node at (0,3) {$\rpoint$};
\node at (3,0) {$\rpoint$};
\node at (0,-3) {$\rpoint$};
\node at (0,1.5) {$\rpoint$};

\draw[thick, red] (0,3)--(0,1.5)--(3,0)..controls (2,-1) and (-2,-1)..(-3,0)--(0,-3);

\node at (2.1,2.1) {$\gpoint$};
\node at (-2.1,2.1) {$\gpoint$};
\node at (2.1,-2.1) {$\gpoint$};
\node at (-2.1,-2.1) {$\gpoint$};

\draw[thick, dark-green] (-2.1,-2.1)--(2.1,-2.1)..controls (1,-2.1) and (-2.1,1)..(-2.1,2.1)--(2.1,2.1);
\draw[thick, dark-green] (-2.1,2.1)..controls (0,0) and (0,0)..(2.1,2.1);

\end{tikzpicture}}\]

The following is the main result we use in this section.

\begin{theorem}\cite{APS}\label{theorem::APS}
Let $(\surf, M_{\rpoint}, P_{\rpoint}, D)$ be a dissected surface and $A=A(D)$ be the corresponding gentle algebra. 

\begin{enumerate}
\item If $T$ is a basic tilting object in $\mathcal{D}^b(A)$, then there exists a collection of graded arcs $\{(\gamma_i,\mathbf{n}_i), i\in I\}$ such that 
\begin{enumerate} 
\item $T\simeq \bigoplus_{i\in I}P_{(\gamma_i,\grading_i)}$;
\item $\{\gamma_i, i\in I\}$ is a $\gpoint$-dissection whose dual $\rpoint$-dissection is denoted by $D_T$;
\item we have an isomorphism of algebras $\End_{\mathcal{D}^b(A)}(T)\simeq A(D_T)$;
\item for any $\delta\in \pi_1(\surf)$, we have $w_D(\delta)=w_{D_T}(\delta)$.
\end{enumerate}
\item Let $\{\gamma_i, i\in I\}$ be a $\gpoint$-dissection, and denote by $D'$ its dual $\rpoint$-dissection. If for any $\delta\in \pi_1(\surf)$ we have $w_D(\delta)=w_{D'}(\delta)$, then there exists a grading $\grading_i$ for any $i\in I$ such that $\bigoplus_{i\in I} P_{(\gamma_i,\grading_i)}$ is a tilting object in $\mathcal{D}^b(A)$. 
\end{enumerate}
\end{theorem}

In this result, the object $A$ seen as a tilting object in $\mathcal{D}^b(A)$ corresponds to the dual $\gpoint$-dissection of $D$ with the zero grading.

\begin{remark}\label{remark::n_i}A key point in the proof of Theorem \ref{theorem::APS} is the following fact:  if $T=\bigoplus_{i\in I}P_{(\gamma_i,\mathbf{n}_i)}$ is a tilting object, and if $\gamma_i$ and $\gamma_j$ intersects on the boundary (say $\gamma_i(0)=\gamma_j(0)$), then $\mathbf{n}_i(\gamma_i(t_1))=\mathbf{n}_j(\gamma_j(t'_1))$ where $\gamma_i(t_1)$ (resp. $\gamma_j(t'_1)$) is the first intersection point of $\gamma_i$ (resp. $\gamma_j$) with $D$.
\end{remark}

\subsection{$G$-invariant tilting objects}

Our aim is now to adapt Theorem~\ref{theorem::APS} to the case of a $G$-marked surface.

Let $(\surf, M,P,\sigma, D)$ be a $G$-$\rpoint$-dissected surface, and $A$ the corresponding gentle $G$-algebra. 

\begin{lemma}\label{lemma::sigma P gamma}
Let $(\gamma,\grading)$ be a graded curve. Then we have $(P_{(\gamma,\grading)})^\sigma\simeq P_{(\sigma\circ\gamma,\grading\circ \sigma)}$ in $\mathcal{D}^b(A)$.
\end{lemma}

\begin{proof}
First note that if $i$ is a vertex of $Q(D)$, then the automorphism $\sigma$ of $A$ induces an isomorphism of projective $A$-modules $$(e_i A)^\sigma=e_i A_{\sigma}\simeq e_{\sigma(i)} A.$$

If $\gamma$ intersects the arcs $i_1,\ldots, i_r$ of $D$ in $t_1<\cdots<t_r$, then the arc $\sigma\circ\gamma$ intersects the arcs $\sigma(i_1),\ldots, \sigma(i_r)$ in $t_1<\cdots<t_r$. It is immediate to see that $p_j(\sigma\circ\gamma)=\sigma(p_j(\gamma))$. Hence as a graded $A$-module we have $$P_{(\sigma\circ\gamma,\grading\circ\sigma^{-1})}=\bigoplus_{j=1}^r e_{\sigma(i_j)}A[\grading(\gamma(t_j))].$$
Now, since $D$ is $G$-invariant, the line field $\eta$ attached to it is also $G$-invariant, that is  we have $\sigma^*(\eta)=\eta$.  Therefore we have 
$$w_{\eta}(\sigma\circ \gamma_{|_{[t_j,t_{j+1}]}})=w_{\sigma^*(\eta)}(\gamma_{|_{[t_j,t_{j+1}]}})=w_{\eta}(\gamma_{|_{[t_j,t_{j+1}]}}).$$

Hence we get the result.

\end{proof}

\begin{theorem}\label{thm::G-tilting=G-dissection}
Let $(\surf, M_{\rpoint}, P_{\rpoint}, \sigma,D)$ be a $G$-dissected surface and $A=A(D)$ be the corresponding gentle $G$-algebra. 

\begin{enumerate}
\item If $T$ is a basic $G$-invariant tilting object  in $\mathcal{D}^b(A)$, then there exists a collection of graded arcs $\{(\gamma_i,\mathbf{n}_i), i\in I\}$ such that 
\begin{enumerate} 
\item $T\simeq \bigoplus_{i\in I}P_{(\gamma_i,\grading_i)}$;
\item $\{\gamma_i, i\in I\}$ is a $\gpoint$-dissection which is $G$-invariant, and whose dual $\rpoint$-dissection is denoted by $D_T$;
\item we have an isomorphism of $G$-algebras $\End_{\mathcal{D}^b(A)}(T)\simeq A(D_T)$;
\item for any $\delta\in \pi_1(\surf)$, we have $w_D(\delta)=w_{D_T}(\delta)$.
\end{enumerate}
\item Let $\{\gamma_i, i\in I\}$ be a $G$-invariant $\gpoint$-dissection, and denote by $D'$ its dual $\rpoint$-dissection. If for any $\delta\in \pi_1(\surf)$ we have $w_D(\delta)=w_{D'}(\delta)$, then there exist  a grading $\grading_i$ for any $i\in I$ such that $\bigoplus_{i\in I} P_{(\gamma_i,\grading_i)}$ is a $G$-invariant tilting object in $\mathcal{D}^b(A)$. 
\end{enumerate}
\end{theorem}

\begin{proof}
Assume that $T$ is a $G$-invariant tilting object. Then by Theorem~\ref{theorem::APS}, $T$ is of the form $\bigoplus_{i\in I}P_{(\gamma_i,\grading_i)}$ for some $\gpoint$-dissection $\{\gamma_i, i\in I\}$. Since $T$ is $G$-invariant, we have by Lemma \ref{lemma::sigma P gamma}$$\bigoplus_{i\in I}P_{(\gamma_i,\grading_i)}^\sigma\simeq\bigoplus_{i\in I} P_{(\sigma\circ \gamma_i,\grading\circ\sigma)}\simeq \bigoplus_{i\in I} P_{(\gamma_i,\grading_i)}.$$
Moreover, $P_{(\gamma,\grading)}\simeq P_{(\gamma',\grading')}$ implies that $\gamma'$ is homotopic to $\gamma$ or $\gamma^{-1}$, hence we obtain that $\{\gamma_i, i\in I\}$ and its dual $D_T$ are $\sigma$-invariant. Thus we get (1) (b).

\medskip

Now we want to check that the isomorphism $\End_{\mathcal{D}^b(A)}(T)\simeq A(D_T)$ commutes with the action of $\sigma$. It is enough to verify that the action commutes on the generators, that is on the quiver. First, the vertices of $Q(D_T)$ are in bijection with the arcs of $D_T$ which are in bijection with the arcs $\gamma_i$. The action of $\sigma$ on the vertex corresponding to $\gamma_i$ is then $\sigma(\gamma_i)$. Since $P_{(\gamma_i,\grading_i)}^\sigma$ is isomorphic to $P_{(\sigma(\gamma_i), \grading\circ\sigma)}$, the action is compatible on the vertices.

Secondly, let $\alpha:i\to j$ be an arrow in the quiver $Q(D_T)$. We will explicitly construct its image $p_\alpha$ through the isomorphism $A(D_T)\to \End_{\mathcal{D}^b(A)}(T)$.  The arrow $\alpha$ goes from $i$ to $j$ in $Q(D_T)$ precisely when the arcs $\gamma_i$ and $\gamma_j$ share an endpoint (assume $\gamma_i(0)=\gamma_j(0)$) and $\gamma_j$ follows directly $\gamma_i$  counterclockwise around $\gamma_i(0)$. Moreover, by Remark \ref{remark::n_i}, we have 
$\mathbf{n}_i(\gamma_i(t_1))=\mathbf{n}_j(\gamma_j(t'_1))$ where $\gamma_i(t_1)$ (resp. $\gamma_j(t'_1)$) is the first intersection point of $\gamma_i$ (resp. $\gamma_j$) with $D$.

 Denote by $\ell$ (resp. $k$) the arc of $D$ such that $\gamma_i(t_1)\in \ell$ (resp. $\gamma_j(t'_1)\in k$). The arcs $\ell$ and $k$ are a side of a common polygon cut out by $D$ (the one containing $\gamma_i(0)=\gamma_j(0)$ on its boundary). So there is a path (that maybe trivial ) from $\ell$ to $k$ in $Q(D)$, which corresponds to a non zero map $$p_{\alpha}:e_{\ell}A(D)[\grading_i(\gamma_i(t_1))]\to e_{k}A(D)[\grading_j(\gamma_j(t'_1)].$$  The image of $\alpha:i\to j$ in $\End_{\mathcal{D}^b(A)}(T)$ is the morphism $P_{(\gamma_i,\grading_i)}\to P_{(\gamma_j,\grading_j)}$ induced by the map $p_{\alpha}$. 
 
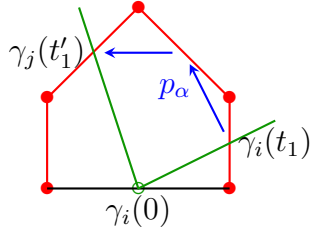
\[\scalebox{1}{
\begin{tikzpicture}[>=stealth,scale=0.6]
\node at (0,0) {$\rpoint$};
\node at (0,2) {$\rpoint$};
\node at (2,4) {$\rpoint$};
\node at (4,2) {$\rpoint$};
\node at (4,0) {$\rpoint$};
\node at (2,0) {$\gpoint$};

\draw[thick, red] (0,0)--node (A1) {}(0,2)--node (A2) {}(2,4)--node (A3) {}(4,2)--node (A4) {}(4,0);
\draw[thick] (0,0)--(4,0);
\draw[thick, blue, ->] (A4)--(A3);\draw[thick, blue, ->] (A3)--(A2);

\draw[thick, dark-green] (5,1.5)--(2,0)--(0.7,4);

\node at (2,-0.5) {$\gamma_i(0)$};

\node at (5,1) {$\gamma_i(t_1)$};
\node at (0,3) {$\gamma_j(t'_1)$};
\node[blue] at (2.8,2.2) {$p_{\alpha}$};

\end{tikzpicture}}\] 
 
From the construction, it is now clear that $\sigma (p_{\alpha})=p_{\sigma(\alpha)}$ and we get (1) (c).

\medskip

Let $\{\gamma_i, i\in I\}$ be a $G$-invariant $\gpoint$-dissection as in (2). Then by Theorem \ref{theorem::APS} there exists a grading $\grading_i$ for each $i\in I$ such that $\bigoplus_{i\in I} P_{(\gamma_i,\grading_i)}$ is a tilting object. Since the collection $\{\gamma_i, i\in I\}$ is $G$-invariant, there exists a permutation $\omega$ of the indices $i\in I$ such that $\sigma(\gamma_i)=\gamma_{\omega(i)}$ or $\sigma(\gamma_i)=\gamma_{\omega(i)}^{-1}$. In order to prove that $T$ is $G$-invariant we need to show that  for any $i\in I$, if $t$ is such that $\gamma_i(t)$ is in $D$, then \begin{equation}\label{eq-grading}
\grading_i\circ\sigma (\sigma(\gamma_i(t))=\grading_i(\gamma_i(t))=\grading_{\omega (i)}(\sigma(\gamma_i(t)) 
\end{equation}

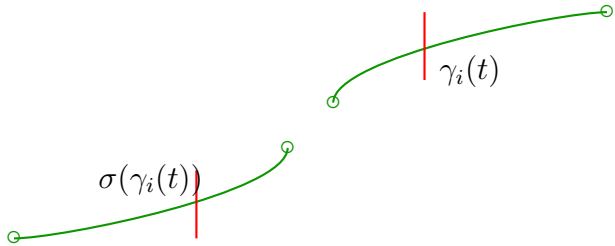
\[\scalebox{1}{
\begin{tikzpicture}[>=stealth,scale=0.6]

\node at (0,0) {$\gpoint$};
\node at (6,2) {$\gpoint$};

\draw[dark-green, thick] (0,0)..controls (0,1) and (5,2)..(6,2);

\draw[red, thick] (2,0.5)--(2,2);
\node at (3,0.7) {$\gamma_i(t)$};

\begin{scope}[rotate=180, yshift=1cm, xshift=1cm]
\node at (0,0) {$\gpoint$};
\node at (6,2) {$\gpoint$};

\draw[dark-green, thick] (0,0)..controls (0,1) and (5,2)..(6,2);

\draw[red, thick] (2,0.5)--(2,2);
\node at (3,0.7) {$\sigma(\gamma_i(t))$};

\end{scope}

\end{tikzpicture}}\]

First assume that $i$ is such that $\omega(i)=i$. This means that $\sigma(\gamma_i)=\gamma_i^{-1}$, and there exists a unique point of $\gamma_i$ fixed by $\sigma$. This point is then a $\cross$, and without loss of generality we may assume that it is $\gamma_i(\frac{1}{2})$. Let $t<\frac{1}{2}$ be such that $\gamma_i(t)\in D$. By definition of a grading we have $$\grading_i(\gamma_i(t))=\grading_i(\gamma_i(\frac{1}{2}))-w_{\eta}(\gamma_{|[t,\frac{1}{2}]}).$$

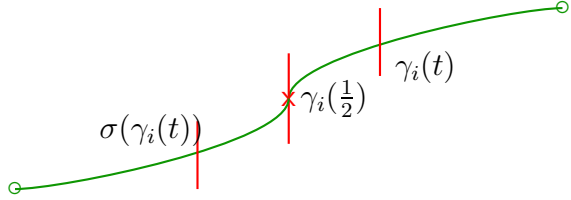
\[\scalebox{1}{
\begin{tikzpicture}[>=stealth,scale=0.6]

\node at (0,0) {$\cross$};
\node at (6,2) {$\gpoint$};
\node at (1,0) {$\gamma_i(\frac{1}{2})$};

\draw[dark-green, thick] (0,0)..controls (0,1) and (5,2)..(6,2);

\draw[red, thick] (2,0.5)--(2,2);
\node at (3,0.7) {$\gamma_i(t)$};
\draw[red, thick] (0,-1)--(0,1);

\begin{scope}[rotate=180]

\node at (6,2) {$\gpoint$};

\draw[dark-green, thick] (0,0)..controls (0,1) and (5,2)..(6,2);

\draw[red, thick] (2,0.5)--(2,2);
\node at (3,0.7) {$\sigma(\gamma_i(t))$};

\end{scope}

\end{tikzpicture}}\]
Therefore we have the following equalities:
\begin{align*} \grading_i(\sigma(\gamma_i(t)) = 
  \grading_i(\sigma(\gamma_i(\frac{1}{2})))- w_\eta (\sigma\circ\gamma_{i|_{[t,\frac{1}{2}]}}) \\ = \grading_i(\gamma_i(\frac{1}{2}))- w_{\sigma^*\eta} (\gamma_{i|_{[t,\frac{1}{2}]}})  = \grading_i(\gamma_i(t)) 
 \end{align*}
 since $\sigma^*\eta$ is homotopic to $\eta$.
 That is, we have \eqref{eq-grading} for $i$ such that $\omega(i)=i$.
 
 Now assume that $\gamma_j$ is an arc with $\omega(j)\neq j$. Suppose that $\gamma_j$ shares an endpoint with an arc $\gamma_i$ satisfying \eqref{eq-grading}. Without loss of generality we may assume that $\gamma_i(0)=\gamma_j(0)$.  Define $t_1$ (resp. $t'_1$) such that $\gamma_i(t_1)$ (resp. $\gamma_j(t'_1)$) is the first intersection of $\gamma_i$ (resp. $\gamma_j$) with $D$. Let $t\geq t'_1$ such that $\gamma_j(t)$ is in $D$. Then $\sigma(\gamma_i)$ and $\sigma(\gamma_j)$ also have the same starting point, and their first intersection with $D$ are also at $t_1$ (resp. at $t'_1$).
 
 \[\scalebox{1}
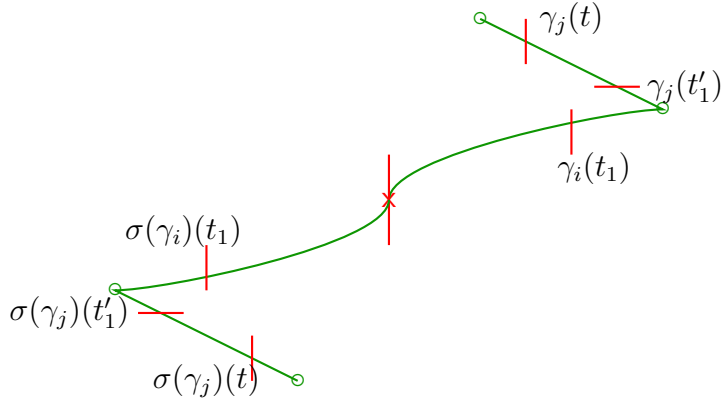
{
\begin{tikzpicture}[>=stealth,scale=0.6]

\node at (0,0) {$\cross$};
\node at (6,2) {$\gpoint$};
\node at (2,4) {$\gpoint$};

\draw[dark-green, thick] (0,0)..controls (0,1) and (5,2)..(6,2)--(2,4);

\draw[red, thick] (4,1)--(4,2);
\node at (4.5,0.7) {$\gamma_i(t_1)$};
\draw[red, thick] (0,-1)--(0,1);
\draw[red, thick] (4.5,2.5)--(5.5,2.5);
\draw[red, thick] (3,3)--(3,4);
\node at (6.5,2.5) {$\gamma_j(t'_1)$};
\node at (4,4) {$\gamma_j(t)$};

\begin{scope}[rotate=180]

\node at (6,2) {$\gpoint$};
\node at (2,4) {$\gpoint$};

\draw[dark-green, thick] (0,0)..controls (0,1) and (5,2)..(6,2)--(2,4);

\draw[red, thick] (4,1)--(4,2);
\draw[red, thick] (4.5,2.5)--(5.5,2.5);
\draw[red, thick] (3,3)--(3,4);

\node at (4.5,0.7) {$\sigma(\gamma_i)(t_1)$};
\node at (7,2.5) {$\sigma(\gamma_j)(t'_1)$};
\node at (4,4) {$\sigma(\gamma_j)(t)$};

\end{scope}

\end{tikzpicture}}\]
 We have the equalities

\begin{align*}
\grading_{\omega(j)}(\sigma(\gamma_j)(t))  =  \grading_{\omega(j)}(\sigma(\gamma_j(t'_1)) + w_{\eta}(\sigma \circ\gamma_{j|_{[t'_1,t]}})\\
 =  \grading_{\omega(i)}(\sigma(\gamma_i(t_1)) + w_{\sigma^*(\eta)}(\gamma_{j|_{[t'_1,t]}}) 
 \\ = \grading_{i}(\gamma_i(t_1)) + w_{\eta}(\gamma_{j|_{[t'_1,t]}}) 
 \\ = \grading_{j}(\gamma_j(t'_1)) + w_{\eta}(\gamma_{j|_{[t'_1,t]}})  
 \\ = \grading_j(\gamma_j(t)) 
\end{align*}

Now we can conclude by induction since the surface $\surf$ is connected and since there exists at least one fixed point for $\sigma$.
 
\end{proof}

\subsection{$\dual$-derived equivalences}

Combining this result with Corollary~\ref{cor::Geq=G-hat-eq-gentle} we obtain the following.

\begin{theorem}\label{theorem::dual-derived equivalence}
Let $\bar{\Lambda}$ and $\bar{\Lambda}'$ be skew-gentle algebras, together with their natural $\dual$-action. Let $(\surf,M,P,\sigma, D)$ (resp. $(\surf',M',P',\sigma', D')$) be the $G$-dissected surface associated to $\bar{\Lambda}$ (resp. to $\bar{\Lambda}'$) as constructed in Theorem \ref{thm::construction covering}.
The following are equivalent
\begin{enumerate}
\item the algebras $\bar{\Lambda}$ and $\bar{\Lambda}'$ are $\dual$-derived equivalent;
\item there exists an orientation preserving $G$-diffeomorphism $\Phi:\surf\to \surf'$ sending $M$ (resp. $P$) to $M'$ (resp. $P'$) such that the line fields $\Phi^*(\eta')$ and $\eta$ are homotopic.
\end{enumerate}
\end{theorem}

\begin{proof}
Denote by $\Lambda$ (resp. $\Lambda'$) the $G$-gentle algebras associated to $\bar{\Lambda}$ (resp. $\bar{\Lambda}'$) as in Theorem \ref{thm::construction covering}. These are the algebras associated with the $G$-dissected surfaces $(\surf,M,P,\sigma,D)$ (resp. $(\surf',M',P',\sigma',D')$). 
From Corollary \ref{cor::Geq=G-hat-eq-gentle}, (1) is equivalent to the fact that $\Lambda$ and $\Lambda'$ are $G$-derived equivalent.

Assume (1), then there exists a $G$-tilting object $T\in \mathcal{D}^b(\Lambda)$ together with a $G$-isomorphism $\End_{\mathcal{D}^b(\Lambda)}(T)\simeq \Lambda'$. Hence by Theorem \ref{thm::G-tilting=G-dissection}, there exists a $G$-invariant dissection $D_T$ of $\surf$, together with a $G$-isomorphism $A(D_T)\simeq \Lambda'\simeq A(D')$. By Proposition \ref{prop::bijection-G-gentle-dissection}, there exists a $G$-invariant diffeomorphism $\Phi:\surf\to \surf$ sending $D_T$ on $D'$. Denote by $\eta$ (resp. $\eta'$) the line field associated with $D$ (resp. $D'$),  then we have for $\delta\in \pi_1(\surf)$  
$$w_{\eta}(\delta)=w_{D}(\delta)=w_{D_T}(\delta)=w_{D'}(\Phi(\delta))=w_{\Phi^*(\eta')}(\delta).$$
Therefore the line fields $\eta$ and $\Phi^*(\eta')$ are homotopic.

\medskip

Asume (2), and denote by $D'':=\Phi^{-1}(D)$. Since $\Phi$ is $G$-invariant, this is a $G$-invariant dissection of $\surf$. Moreover $w_{D''}(\delta)=w_{D'}(\Phi(\delta))=w_{D}(\delta)$ by assumption 
 so we can conclude by Theorem \ref{thm::G-tilting=G-dissection}.
\end{proof}

\begin{remark}
We can apply Theorem 1.2 in \cite{APS} to get a more concrete criterion to check whether two skew-gentle algebras are $\dual$-derived equivalent or not. However, as far as we know, we only get a necessary condition for (1) to be true. Indeed, if (2) is satisfied in Theorem 5.6, then we  get some equalities for the winding numbers of a basis of the fundamental group of the surfaces $\surf$ and $\surf '$ with respect to the line fields $\eta$ and $\eta'$ (see Section \ref{final-examples} for examples). 

However, when trying to apply the converse implication in Theorem 1.2 in \cite{APS}, we only obtain the following: if all the numbers in Theorem 1.2 in \cite{APS} coincide for $\Lambda$ and $\Lambda '$, we deduce that 
\begin{itemize}
\item the line fields $\eta$ and $\eta'$ are $G$-invariant (this is by construction)
\item the surfaces $\surf$ and $\surf'$ are $G$-diffeomorphic;
\item there exists a diffeomorphism $\Phi:\surf\to \surf'$ such that $\Phi^*(\eta')$ is homotopic to $\eta$.
\end{itemize}
But it is not clear that this $\Phi$ is a $G$-diffeomorphism. 
\end{remark}
\subsection{Derived equivalence via $\dual$-tilting objects}

We are now interested in the case where the derived equivalence between two skew-gentle algebras does not necessarily respect the $\dual$-action.

Let $\surf$ be a smooth surface, and $\sigma$ be a diffeomorphism of $\surf$ of order 2 with finitely many fixed points $X$. Denote by $\bar{\surf}=\surf/\sigma$ the corresponding orbifold and $p:\surf\to \bar{\surf}$ the projection. If $\eta$ is a $G$-invariant line field on $\surf$, then there exists a line field $\bar{\eta}=p_*(\eta)$ on $\bar{\surf}\setminus X$, since $p$ is locally a diffeomorphism on $\surf\setminus X$. 
Moreover, if $\eta$ and $\eta'$ are two $G$-invariant line fields on $\surf$, then we have 
\begin{equation}\label{w-wbar} w_\eta=w_{\eta'} \Leftrightarrow w_{\bar{\eta}}=w_{\bar{\eta}'}.
\end{equation}
Indeed, if $\delta$ is a closed curve in $\pi_1(\surf)$, then $p(\delta)$ is a closed curve in $\bar{\surf}$.  Conversely, if $\delta$ is in $\pi_1(\bar{\surf})$, denote by $\widetilde{\delta}$ a lift of $\delta$. 
If $\widetilde{\delta}$ is a closed curve, we clearly have 
\begin{equation}\label{w-wbar2} w_{\bar{\eta}}(\delta)=w_{\eta}(\widetilde{\delta}).
\end{equation} 

If $\widetilde{\delta}$ is not a closed curve, then $\widetilde{\delta}.\sigma(\widetilde{\delta})$ is closed, and 
\begin{equation}\label{w-wbar3} w_{\bar{\eta}}(\delta)=w_{\eta}(\widetilde{\delta})=\frac{1}{2}(w_{\eta}(\widetilde{\delta}.\sigma(\widetilde{\delta})),
\end{equation} since $\eta$ is $\sigma$-invariant.

\begin{remark}\label{remark::degeneration}
Note that when $(\surf,\sigma,\eta)$ is constructed from a $G$-gentle algebra $A$. The line field $\bar{\eta}=p_*(\eta)$ on $\bar{\surf}\setminus X$ is exactly the line field coming from the gentle degeneration $\bar{A}_0$ of the skew-gentle algebra $\bar{A}$. 
\end{remark}

\begin{theorem}\label{thm::derived-skew-gentle}
Let $\bar{\Lambda}$ and $\bar{\Lambda}'$ be skew-gentle algebras associated with $\cross$-dissected surfaces $(\surf, M, P, X, D)$ and $(\surf', M', P', X', D')$. Then the following are equivalent:

\begin{enumerate}
\item there exists an equivalence $\mathcal{D}^b(\bar{\Lambda})\simeq \mathcal{D}^b(\bar{\Lambda}')$ given by a $\dual$-tilting object;
\item there exists an orientation diffeomorphism $\bar{\Phi}:\surf\to \surf'$ sending $M$ to $M'$, $P$ to $P'$, $X$ to $X'$ and such that the line fields $\eta_{D}$ and $\bar{\Phi}^*(\eta_{D'})$ are homotopic.
\end{enumerate}

\end{theorem}

\begin{proof}

Denote by $\Lambda$ the $G$-gentle algebra corresponding to $\bar{\Lambda}$ as constructed in Theorem \ref{thm::construction covering}. We denote by $(\widetilde{\surf},\widetilde{M},\widetilde{P},\sigma,\widetilde{D})$ the corresponding $G$-dissected surface.

\medskip

Assume (1), and denote by $\bar{T}\in \mathcal{D}^b(\bar{\Lambda})$ a $\dual$-invariant tilting object such that $\End_{\mathcal{D}^b(\bar{\Lambda})}(\bar{T})\simeq \bar{\Lambda}'$ (note that we do not ask this isomorphism to be compatible with the action of $\dual$). 
By Theorem \ref{thm::bijection-tilting}, there exists a $G$-tilting object $T$ in $\mathcal{D}^b(\Lambda)$ such that ${\rm add}(\bar{T})={\rm add}(T\lten{\Lambda} \Lambda G e)$ where $e$ is the idempotent defined in Theorem \ref{thm::construction covering}. 

Denote by $D_T$ the $G$-dissection of $\widetilde{\surf}$ corresponding to $T$, and $\bar{D}_T:=p(D_T)$ the corresponding $\cross$-dissection of $\surf$. By Theorem \ref{thm::G-tilting=G-dissection}(1) (c), we have a $G$-isomorphism 
\begin{equation}
\label{G-iso}
\End_{\mathcal{D}^b(\Lambda)}(T)\underset{G}{\simeq} A(D_T) 
\end{equation} 
Therefore we have the following isomorphisms

\[\begin{array}{rcll}
\bar{A}(D') & \simeq & \bar{\Lambda}'\simeq \End_{\mathcal{D}^b(\bar{\Lambda})}(\bar{T}) & \\
 &\simeq & (\End_{\mathcal{D}^b(\Lambda)}(T)G)_{\rm b} & \textrm{by Theorem \ref{thm::bijection-tilting}}\\
  & \simeq &(A(D_T)G)_{\rm b} &  \textrm{by }\eqref{G-iso}\\
   & \simeq & \bar{A}(\bar{D}_T) & \textrm{by Proposition \ref{prop::iso AG-Abar}}
   \end{array}\]
 Hence by Proposition \ref{prop::bijection-G-gentle-dissection}, there exists a diffeomorphism $\bar{\Phi}:\surf \setminus X\to \surf'\setminus X'$ sending marked points to marked points and such that $\bar{\Phi}(\bar{D}_T)=D'$. Now since $T$ is a tilting object, we have $w_{\widetilde{D}}=w_{D_T}$. Hence by \eqref{w-wbar}, we have $w_{D}=w_{\bar{D}_T}$ and so $w_{\eta_D}=w_{\bar{\Phi}^*(\eta_{D'})}.$
 
 \medskip
 
 Assume (2) and denote by $D'':=\bar{\Phi}^{-1}(D')$, which is a $\cross$-dissection of $\surf$. Then $\widetilde{D}'':=p^{-1}(D'')$ is a $G$-invariant dissection of $\widetilde{\surf}$. By a similar argument as above we have $w_{\widetilde{D}}=w_{\widetilde{D}''}$, hence there exists a $G$-invariant tilting object $T$ in $\mathcal{D}^b(\Lambda)$ together with a $G$-isomorphism 
$$
\End_{\mathcal{D}^b(\Lambda)}(T)\underset{G}{\simeq} A(\widetilde{D}'') 
$$
Then the object $\bar{T}:=T\lten{\Lambda}\Lambda G e$ is a $\dual$-tilting object in $\mathcal{D}^b(\bar{\Lambda})$ such that 
\begin{align*}
\End_{\mathcal{D}^b(\bar{\Lambda})}(\bar{T}) & \simeq (\End_{\mathcal{D}^b(\Lambda)}(T) G)_{\rm b} \\ & \simeq (A(\widetilde{D}'')G)_{\rm b}\\ & \simeq \bar{A}(D'')\\ & \simeq  \bar{A}(D')=\bar{\Lambda}'.
\end{align*}
\end{proof}

 \begin{remark}
 Note that in this proof, we work only in the covering $\widetilde{\surf}$ of $\surf$ given by $D$, and never in the covering of $\surf '$ given by $D'$. Indeed, in general, these two coverings may be non homeomorphic surfaces (cf Examples in section \ref{final-examples}).
 \end{remark}
 
 \begin{remark}
 Theorem \ref{thm::derived-skew-gentle} can be used much more easily than Theorem~\ref{theorem::dual-derived equivalence}. Indeed, given two skew-gentle algebras $\bar{\Lambda}$ and $\bar{\Lambda}'$, it is enough to compute the winding numbers with respect to $\eta_{D}$ and $\eta_{D'}$ of some generators of the fundamental group of each surface $\pi_1(\surf\setminus X)$ and $\pi_1(\surf'\setminus X')$ and compare them using Theorem 1.2 in \cite{APS} to decide wether the algebras $\bar{\Lambda}$ and $\bar{\Lambda}'$ are derived equivalent or not. This is illustrated in the section below.
 \end{remark}
 
 Combining Theorem \ref{thm::derived-skew-gentle} with Remark \ref{remark::degeneration}, we obtain the following.
 \begin{corollary}
Let $\bar{A}$ and $\bar{B}$ be two skew-gentle algebras, and denote by $\bar{A}_0$ and $\bar{B}_0$ their corresponding gentle degenerations. If $\bar{A}$ and $\bar{B}$ are derived equivalent via a $\dual$-tilting object, then $\bar{A}_0$ and $\bar{B}_0$ are derived equivalent. 
 
 \end{corollary}
 
 Note that the converse is not true in general. Indeed, if the gentle algebras $\bar{A}_0$ and $\bar{B}_0$ are derived equivalent, then there exists a diffeomorphism between the corresponding surfaces, but this diffeomorphism could a priori send a $\cross$ to a puncture or vice versa.

 \subsection{Examples}\label{final-examples}
 
 Consider the following four $\cross$-dissections $\bar{D}_1,\ldots,\bar{D}_4$ of the cylinder with two orbifold points and two marked points $(\surf,M,X)$ (the set $P$ is empty here), together with their corresponding skew-gentle algebras $\bar{\Lambda}_i$ as in Figure \ref{figure::example} (the special loops are indicated in red).
 
 Note that for $\bar{\Lambda}_2$ and $\bar{\Lambda}_4$ (resp. $\Lambda_2$ and $\Lambda_4$) the quivers are isomorphic, but the relations are different. Also note that the quiver of $\Lambda_3$ is a garland, but the relations are not anticommutative squares, they are quadratic monomial and the algebra is gentle, not skew-gentle.

\medskip
One checks that the covering surface $\widetilde{\surf}_1,\ldots,\widetilde{\surf}_4$ constructed in Theorem \ref{thm::construction covering} is a sphere with four holes for $\bar{\Lambda}_1$ and $\bar{\Lambda}_2$ while  it is  a torus with two holes   for $\bar{\Lambda}_3$ and $\bar{\Lambda}_4$ (see Example \ref{example::cylinder}). Therefore neither of $\bar{\Lambda}_1$ and $\bar{\Lambda}_2$ is  $\dual$-derived equivalent to $\bar{\Lambda}_3$ or $\bar{\Lambda}_4$, by Theorem \ref{theorem::dual-derived equivalence}.

Denote by $c_1$ and $c_2$ curves in $\pi_1(\surf\setminus X)$ surrounding the two boundary components. Computing the winding numbers of these curves for the dissection $\bar{D}_1$, we obtain $w_{\bar{D}_1}(c_1)=-2$ and $w_{\bar{D}_1}(c_2)=0$. The lift $\widetilde{c}_1$ of $c_1$ (resp. $\widetilde{c}_2$ of $c_2$) in $\widetilde{\surf}_1$ is a  closed curve, hence  by \eqref{w-wbar2} we have $$w_{D_1}(\widetilde{c}_1)=-2\quad \textrm{ and }\quad w_{D_1}(\widetilde{c}_2)=0.$$ 
So by symmetry, we obtain that the winding numbers of the four curves surrounding the boundary components of $\widetilde{\surf}_1$ with respect to $D_1$ are $(-2,0,-2,0)$. 
For $\bar{D}_2$, a similar argument shows that the four winding numbers are $(-1,-1,-1,-1)$ since $w_{\bar{D}_2}(c_1)=w_{\bar{D}_2}(c_2)=-1$. Therefore there are no diffeomorphism from $\widetilde{\surf}_1$ to $\widetilde{\surf}_2$ sending $\eta_{D_1}$ to a line field homotopic to $\eta_{D_2}$. By Theorem \ref{theorem::dual-derived equivalence}  the algebras $\bar{\Lambda}_1$ and $\bar{\Lambda}_2$ are then not $\dual$-equivalent (in fact the gentle algebras $\Lambda_1$ and $\Lambda_2$ are not even derived equivalent).

\medskip
For $\bar{\Lambda}_3$ and $\bar{\Lambda}_4$ we can use a similar argument. We have $w_{\bar{D}_3}(c_1)=w_{\bar{D}_3}(c_2)=-1$, but here a lift $\widetilde{c}_1$ of $c_1$ on $\widetilde{\surf}_{3}$ is not a closed curve. However, $\widetilde{c}_1.\sigma\widetilde{c}_1$ is a closed curve surrounding one the boundary component of $\widetilde{\surf}_3$. Therefore by \eqref{w-wbar3} we have that the winding numbers of the curves surrounding the boundary components of $\widetilde{\surf}_3$ are $(-2,-2)$. For $\bar{\Lambda}_4$ they are $(0,-4)$. Therefore there are no diffeomorphisms from $\widetilde{\surf}_3$ to $\widetilde{\surf}_4$ sending $\eta_{D_3}$ to a line field homotopic to $\eta_{D_4}$, and the algebras $\bar{\Lambda}_3$ and $\bar{\Lambda}_4$ are not $\dual$-equivalent.

\medskip

Now consider the surface $(\surf\setminus X,M)$ which is a cylinder with 2 punctures (the points in $X$) and two marked points on the boundary. In order to understand which of the algebras $\bar{\Lambda}_i$ are derived equivalent via a $\dual$-tilting object, we have to understand which of the surfaces with line field $(\surf\setminus X,\eta_{\bar{D}_i})$ are diffeomorphic. Using Theorem 6.4 in \cite{APS}, since the genus of $\surf\setminus X$ is zero, it is enough to compare the collections $(w_{\eta}(c),n(c))$, where $c$ describes the curves surrounding the boundary components, or the punctures, and where $n(c)$ is the number of marked points for the corresponding boundary, or $0$ if $c$ is surrounding a puncture. In our case, we have two curves $c_1$, $c_2$ surrounding the boundary components and two curves $c_3$ and $c_4$ surrounding the punctures. For $\cross$ in $X$, and any dissection $\bar{D}_i$, since there is exactly one arc with endpoint in $\cross$, we have $w_{\bar{\eta}_i}(c_3)=w_{\bar{\eta}_i}(c_4)=-1$.
Therefore the collection of $(w_{\eta}(c),n(c))$ for $\bar{\Lambda}_1$ is 
$$(-2,1),(0,1),(-1,0),(-1,0).$$
Doing the same computations for all the algebras $\bar{\Lambda}_2$, $\bar{\Lambda}_3$ and $\bar{\Lambda}_4$, we conclude that $\bar{\Lambda}_1$ and $\bar{\Lambda}_4$ are derived equivalent, and so are $\bar{\Lambda}_2$ and $\bar{\Lambda}_3$. Moreover $\bar{\Lambda}_1$ is not derived equivalent via a $\dual$-tilting object to $\bar{\Lambda}_2$.

\begin{landscape}
\begin{figure}[!p]
\caption{Example of the cylinder with two orbifold points}
\label{figure::example}
\begin{tabular}{|c|c|c|c|}

\hline

$(\surf, D)$ & $\bar{A}=\bar{\Lambda}_i$ & $(\widetilde{\surf},\sigma)$ & $\Lambda_i=(A,\sigma)$\\

\hline
 
 \scalebox{0.8}{
  \begin{tikzpicture}[scale=1,>=stealth]
 
 \draw (0,0)--(5,0)--(5,2)--(0,2)--(0,0);
\node at (0,0) {$\rpoint$};
\node at (5,0)  {$\rpoint$};
\node at (5,2) {$\rpoint$};
\node at (0,2) {$\rpoint$};
\node at (2,1) {$\cross$};
\node at (3,1) {$\cross$};

\draw[thick, red] (0,0)--node [black, xshift=-0.2cm]{$1$} (0,2);
\draw[thick, red] (5,0)--node [black, xshift=0.2cm]{$1$} (5,2);
\draw[thick, red] (5,2)--node [black,fill=white, inner sep=0pt]{$3$} (2,1);
\draw[thick, red] (5,2)--node [black,fill=white, inner sep=0pt]{$2$} (3,1);
\draw[thick, red] (0,0)..controls (2,2) and (4,1.8)..node [black,fill=white, inner sep=0pt]{$4$} (5,2);

\end{tikzpicture}}

 &
 
 \scalebox{0.8}{
  \begin{tikzpicture}[scale=1,>=stealth]

\node (1) at (0,0) {$1$};
\node[red]  (2) at (2,0) {$2$};
\node[red] (3) at (4,0) {$3$};
\node (4) at (6,0) {$4$};

\draw[->,thick] (1)--(2);
\draw[->,thick] (2)--(3);
\draw[->,thick] (3)--(4);
\draw[->,thick] (0,-0.2)..controls (0,-1) and (6,-1).. (6,-0.2);

\draw[loosely dotted,thick] (1.5,0) arc (180:360:0.5);
\draw[loosely dotted,thick] (3.5,0) arc (180:360:0.5);

\draw[->,thick] (1.9,0.1)..controls (1,1) and (3,1)..(2.1,0.1);
\draw[->,thick] (3.9,0.1)..controls (3,1) and (5,1)..(4.1,0.1);

\end{tikzpicture}}

&

 \scalebox{0.5}{
  \begin{tikzpicture}[scale=1,>=stealth,rotate=90]
\draw[thick] (3,5)--(3,-1);
\node at (3.3,4.7) {$\sigma$};

\draw[->] (2.8,4.5) arc (-180:80:0.2);
\shadedraw[bottom color=blue!30] (0,0)..controls (0.5,1) and (0.5,3)..(0,4)--(2,4)..controls (2,3.5) and (2.5,3)..(3,3)..controls (3.5,3) and (4,3.5).. (4,4)--(6,4)..controls (5.5,3) and (5.5,1)..(6,0)..controls (6,-0.5) and (4,-0.5)..(4,0)..controls (4,0.5) and (3.5,1)..(3,1).. controls (2.5,1) and ( 2,0.5)..(2,0)..controls (2,-0.5) and (0,-0.5)..(0,0);

\shadedraw[bottom color=red!30] (0,0)..controls (0.5,1) and (0.5,3)..(0,4)--(2,4).. controls (2,2) and (2.6,1)..(3,1).. controls (2.5,1) and ( 2,0.5)..(2,0)..controls (2,-0.5) and (0,-0.5)..(0,0);

\draw[fill=red!20] (1,4) ellipse (1 and 0.3);
\draw[fill=blue!20] (5,4) ellipse (1 and 0.3);

\draw[very thick, dark-green] (2,4)..controls (2,3.5) and (2.5,3)..(3,3)..controls (3.5,3) and (4,3.5).. (4,4);
\draw[very thick, dark-green] (2,4).. controls (2,2) and (2.6,1)..(3,1);
\draw[thick, loosely dotted, dark-green] (4,4).. controls (4,2) and (3.4,1)..(3,1);

\node at (3,1) {$\cross$};
\node at (3,3) {$\cross$};

\node at (2,4) {$\gpoint$};
\node at (4,4) {$\gpoint$};

\end{tikzpicture}}

 & \scalebox{0.6}{
  \begin{tikzpicture}[scale=0.7,>=stealth]
  
 \node (1+) at (0,2) {$1^+$};
 \node (1-) at (0,-2) {$1^-$};
 \node (2) at (2,0) {$2$};
 \node (3) at (5,0) {$3$};
 \node (4+) at (7,2) {$4^+$};
 \node (4-) at (7,-2) {$4^-$};
 \draw[thick,->] (1+)-- (2);
 \draw[thick, ->] (1-)--(2);
 \draw[thick,->] (2.3,0.2)-- (4.7,0.2);
 \draw[thick,->] (2.3,-0.2)-- (4.7,-0.2);
 \draw[thick,->] (3)--(4+);
 \draw[thick, ->] (3)-- (4-);
 \draw[thick, ->] (1+)-- (4+);
 \draw[thick, ->] (1-)-- (4-);
 
 \draw[loosely dotted,thick] (3,0.2) arc (0:135:1);
  \draw[loosely dotted,thick] (3,-0.2) arc (0:-135:1);
   \draw[loosely dotted,thick] (4,0.2) arc (180:45:1);
    \draw[loosely dotted,thick] (4,-0.2) arc (180:315:1);

  \end{tikzpicture}} \\

\hline
\scalebox{0.8}{
  \begin{tikzpicture}[scale=1,>=stealth]
 
 \draw (0,0)--(5,0)--(5,2)--(0,2)--(0,0);
\node at (0,0) {$\rpoint$};
\node at (5,0)  {$\rpoint$};
\node at (5,2) {$\rpoint$};
\node at (0,2) {$\rpoint$};
\node at (2,1) {$\cross$};
\node at (3,1) {$\cross$};

\draw[thick, red] (0,0)--node [black, xshift=-0.2cm]{$1$} (0,2);
\draw[thick, red] (5,0)--node [black, xshift=0.2cm]{$1$} (5,2);
\draw[thick, red] (0,0)--node [black,fill=white, inner sep=0pt]{$3$} (2,1);
\draw[thick, red] (5,2)--node [black,fill=white, inner sep=0pt]{$2$} (3,1);
\draw[thick, red] (0,0)..controls (2,2) and (4,1.8)..node [black,fill=white, inner sep=0pt]{$4$} (5,2);

\end{tikzpicture}}
 & 
 \scalebox{0.8}{
  \begin{tikzpicture}[scale=1,>=stealth]

\node (1) at (0,0) {$1$};
\node[red]  (2) at (2,0) {$2$};
\node[red] (3) at (6,0) {$3$};
\node (4) at (4,0) {$4$};

\draw[->,thick] (1)--(2);
\draw[->,thick] (2)--(4);
\draw[->,thick] (4)--(3);
\draw[->,thick] (0,-0.2)..controls (0,-1) and (4,-1).. (4,-0.2);

\draw[loosely dotted,thick] (1.5,0) arc (180:360:0.5);
\draw[loosely dotted,thick] (4.5,0) arc (0:-110:0.5);

\draw[->,thick] (1.9,0.1)..controls (1,1) and (3,1)..(2.1,0.1);
\draw[->,thick] (5.9,0.1)..controls (5,1) and (7,1)..(6.1,0.1);

\end{tikzpicture}}
 & \scalebox{0.5}{
  \begin{tikzpicture}[scale=1,>=stealth,rotate=90]
\draw[thick] (3,5)--(3,-1);
\node at (3.3,4.7) {$\sigma$};

\draw[->] (2.8,4.5) arc (-180:80:0.2);
\shadedraw[bottom color=blue!30] (0,0)..controls (0.5,1) and (0.5,3)..(0,4)--(2,4)..controls (2,3.5) and (2.5,3)..(3,3)..controls (3.5,3) and (4,3.5).. (4,4)--(6,4)..controls (5.5,3) and (5.5,1)..(6,0)..controls (6,-0.5) and (4,-0.5)..(4,0)..controls (4,0.5) and (3.5,1)..(3,1).. controls (2.5,1) and ( 2,0.5)..(2,0)..controls (2,-0.5) and (0,-0.5)..(0,0);

\shadedraw[bottom color=red!30] (0,0)..controls (0.5,1) and (0.5,3)..(0,4)--(2,4).. controls (2,2) and (2.6,1)..(3,1).. controls (2.5,1) and ( 2,0.5)..(2,0)..controls (2,-0.5) and (0,-0.5)..(0,0);

\draw[fill=red!20] (1,4) ellipse (1 and 0.3);
\draw[fill=blue!20] (5,4) ellipse (1 and 0.3);

\draw[very thick, dark-green] (2,4)..controls (2,3.5) and (2.5,3)..(3,3)..controls (3.5,3) and (4,3.5).. (4,4);
\draw[very thick, dark-green] (2,4).. controls (2,2) and (2.6,1)..(3,1);
\draw[thick, loosely dotted, dark-green] (4,4).. controls (4,2) and (3.4,1)..(3,1);

\node at (3,1) {$\cross$};
\node at (3,3) {$\cross$};

\node at (2,4) {$\gpoint$};
\node at (4,4) {$\gpoint$};

\end{tikzpicture}}
 &
 
\scalebox{0.6}{
  \begin{tikzpicture}[scale=0.7,>=stealth]
  
 \node (1+) at (0,2) {$1^+$};
 \node (1-) at (0,-2) {$1^-$};
 \node (2) at (2,0) {$2$};
 \node (3) at (6,0) {$3$};
 \node (4+) at (4,2) {$4^+$};
 \node (4-) at (4,-2) {$4^-$};
 \draw[thick,->] (1+)-- (2);
 \draw[thick, ->] (1-)--(2);
 \draw[thick,->] (4+)-- (3);
 \draw[thick,->] (4-)-- (3);
 \draw[thick,->] (2)--(4+);
 \draw[thick, ->] (2)-- (4-);
 \draw[thick, ->] (1+)-- (4+);
 \draw[thick, ->] (1-)-- (4-);
 
 \draw[loosely dotted, thick] (2.5,0.5) arc (45:135:0.7);
  \draw[loosely dotted, thick] (2.5,-0.5) arc (-45:-135:0.7);
   \draw[loosely dotted,thick] (3,2) arc (180:-45:1);
    \draw[loosely dotted,thick] (3,-2) arc (-180:45:1);

  \end{tikzpicture}}
 
 \\
 \hline
\scalebox{0.8}{
  \begin{tikzpicture}[scale=1,>=stealth]
 
 \draw (0,0)--(5,0)--(5,2)--(0,2)--(0,0);
\node at (0,0) {$\rpoint$};
\node at (5,0)  {$\rpoint$};
\node at (5,2) {$\rpoint$};
\node at (0,2) {$\rpoint$};
\node at (2,1) {$\cross$};
\node at (3,1) {$\cross$};

\draw[thick, red] (0,0)--node [black, xshift=-0.2cm]{$1$} (0,2);
\draw[thick, red] (5,0)--node [black, xshift=0.2cm]{$1$} (5,2);
\draw[thick, red] (0,0)--node [black,fill=white, inner sep=0pt]{$3$} (2,1);
\draw[thick, red] (5,2)--node [black,fill=white, inner sep=0pt]{$2$} (3,1);
\draw[thick, red] (0,0)..controls (2,0.5) and (3,1.5)..node [black,fill=white, inner sep=0pt]{$4$} (5,2);
\end{tikzpicture}}
 & 

 \scalebox{0.8}{
  \begin{tikzpicture}[scale=1,>=stealth]

\node (1) at (0,0) {$1$};
\node[red]  (2) at (2,1) {$2$};
\node[red] (3) at (2,-1) {$3$};
\node (4) at (4,0) {$4$};

\draw[->,thick] (1)--(2);
\draw[->,thick] (2)--(4);
\draw[->,thick] (1)--(3);
\draw[->,thick] (3)--(4);

\draw[loosely dotted,thick] (1,0.5)--(3,0.5); 
\draw[loosely dotted,thick] (1,-0.5)--(3,-0.5); 

\draw[->,thick] (1.9,1.1)..controls (1,2) and (3,2)..(2.1,1.1);
\draw[->,thick] (1.9,-1.1)..controls (1,-2) and (3,-2)..(2.1,-1.1);

\end{tikzpicture}}
 &
 \scalebox{0.5}{
  \begin{tikzpicture}[scale=1,>=stealth,rotate=90]

\draw (1,-2)--(1,-0.5);

\shadedraw[bottom color=red!30] (0,0).. controls (0,-0.5) and (2,-0.5)..(2,0).. controls (2,1) and (3,2)..(3,3)..controls (3,4) and (2,5)..(2,6)--(0,6)..controls (0,5) and (-1,4)..(-1,3)..controls (-1,2) and (0,1)..(0,0);

\shadedraw[bottom color=blue!30] (0,0).. controls (0,1) and (0.5,2.75)..(1,2.75)--(1,3.25)..controls (1.5,3.25) and (2,5)..(2,6)--(0,6)..controls (0,5) and (-1,4)..(-1,3)..controls (-1,2) and (0,1)..(0,0);

\draw[fill=red!20] (1,6) ellipse (1 and 0.3);
\draw[fill=white] (0.5,3)..controls (0.6,2.9) and (0.8,2.75)..(1,2.75)..controls (1.2,2.75) and (1.4,2.9).. (1.5,3)..controls (1.4,3.1) and (1.2,3.25)..(1,3.25)..controls (0.8,3.25) and (0.4,3.1)..(0.5,3);

\draw[dark-green, very thick] (0,0).. controls (0,1) and (0.5,2.75)..(1,2.75);
\draw[dark-green, very thick, loosely dotted] (1,2.75).. controls (1.5,2.75) and (2,1)..(2,0);
\draw[dark-green, very thick,loosely dotted] (0,6).. controls (0,5) and (0.5,3.25)..(1,3.25);
\draw[dark-green, very thick] (2,6)..controls (2,5) and (1.5,3.25)..(1,3.25);

\draw(0.5,3)--(0.3,3.2);\draw(1.5,3)--(1.7,3.2);

\node at (0,0){$\gpoint$};
\node at (2,0) {$\gpoint$};
\node at (0,6){$\gpoint$};
\node at (2,6) {$\gpoint$};
\draw[thick] (1,-2)--(1,-0.4);
\draw[thick] (1,2.75)--(1,3.25);
\draw[thick] (1,5.7)--(1,7);

\draw[->] (0.8,-1) arc (-180:80:0.2);
\node at (1,2.75) {$\cross$};
\node at (1,3.25) {$\cross$};

\end{tikzpicture}}
 & 
 
\scalebox{0.6}{
  \begin{tikzpicture}[scale=0.7,>=stealth]
  
 \node (1+) at (0,1.5) {$1^+$};
 \node (1-) at (0,-1.5) {$1^-$};
 \node (2) at (3,1.5) {$2$};
 \node (3) at (3,-1.5) {$3$};
 \node (4+) at (6,1.5) {$4^+$};
 \node (4-) at (6,-1.5) {$4^-$};
 \draw[thick,->] (1+)-- (2);
 \draw[thick, ->] (1-)--(2);
 \draw[thick,->] (1+)-- (3);
 \draw[thick,->] (1-)-- (3);
 \draw[thick,->] (2)--(4+);
 \draw[thick, ->] (2)-- (4-);
 \draw[thick, ->] (3)-- (4+);
 \draw[thick, ->] (3)-- (4-);
 
 \draw[loosely dotted, thick] (3.7,1.5) arc (0:180:0.7);
  \draw[loosely dotted, thick] (3.7,-1.5) arc (0:-180:0.7);
   \draw[loosely dotted,thick] (3.7,0.7) arc (-45:-135:1);
    \draw[loosely dotted,thick] (3.7,-0.7) arc (45:135:1);

  \end{tikzpicture}} 
 \\
 \hline

\scalebox{0.8}{
  \begin{tikzpicture}[scale=1,>=stealth]
 
 \draw (0,0)--(5,0)--(5,2)--(0,2)--(0,0);
\node at (0,0) {$\rpoint$};
\node at (5,0)  {$\rpoint$};
\node at (5,2) {$\rpoint$};
\node at (0,2) {$\rpoint$};
\node at (2,1) {$\cross$};
\node at (3,1) {$\cross$};

\draw[thick, red] (0,0)--node [black, xshift=-0.2cm]{$1$} (0,2);
\draw[thick, red] (5,0)--node [black, xshift=0.2cm]{$1$} (5,2);
\draw[thick, red] (0,0)--node [black,fill=white, inner sep=0pt]{$3$} (2,1);
\draw[thick, red] (0,0)..controls (2,0.2) and (2.5,0.8)..node [black,fill=white, inner sep=0pt]{$2$} (3,1);
\draw[thick, red] (0,0)..controls (2,0.5) and (3,1.5)..node [black,fill=white, inner sep=0pt]{$4$} (5,2);

\end{tikzpicture}}
 & 
 \scalebox{0.8}{
  \begin{tikzpicture}[scale=1,>=stealth]

\node (1) at (0,0) {$1$};
\node[red]  (2) at (2,0) {$3$};
\node[red] (3) at (6,0) {$2$};
\node (4) at (4,0) {$4$};

\draw[->,thick] (1)--(2);
\draw[->,thick] (2)--(4);
\draw[->,thick] (4)--(3);
\draw[->,thick] (0,-0.2)..controls (0,-1) and (4,-1).. (4,-0.2);

\draw[loosely dotted,thick] (1.5,0) arc (180:360:0.5);
\draw[loosely dotted,thick] (4.5,0) arc (0:180:0.5);

\draw[->,thick] (1.9,0.1)..controls (1,1) and (3,1)..(2.1,0.1);
\draw[->,thick] (5.9,0.1)..controls (5,1) and (7,1)..(6.1,0.1);

\end{tikzpicture}}
 &  \scalebox{0.5}{
  \begin{tikzpicture}[scale=1,>=stealth,rotate=90]

\draw (1,-2)--(1,-0.5);

\shadedraw[bottom color=red!30] (0,0).. controls (0,-0.5) and (2,-0.5)..(2,0).. controls (2,1) and (3,2)..(3,3)..controls (3,4) and (2,5)..(2,6)--(0,6)..controls (0,5) and (-1,4)..(-1,3)..controls (-1,2) and (0,1)..(0,0);

\shadedraw[bottom color=blue!30] (0,0).. controls (0,1) and (0.5,2.75)..(1,2.75)--(1,3.25)..controls (1.5,3.25) and (2,5)..(2,6)--(0,6)..controls (0,5) and (-1,4)..(-1,3)..controls (-1,2) and (0,1)..(0,0);

\draw[fill=red!20] (1,6) ellipse (1 and 0.3);
\draw[fill=white] (0.5,3)..controls (0.6,2.9) and (0.8,2.75)..(1,2.75)..controls (1.2,2.75) and (1.4,2.9).. (1.5,3)..controls (1.4,3.1) and (1.2,3.25)..(1,3.25)..controls (0.8,3.25) and (0.4,3.1)..(0.5,3);

\draw[dark-green, very thick] (0,0).. controls (0,1) and (0.5,2.75)..(1,2.75);
\draw[dark-green, very thick, loosely dotted] (1,2.75).. controls (1.5,2.75) and (2,1)..(2,0);
\draw[dark-green, very thick,loosely dotted] (0,6).. controls (0,5) and (0.5,3.25)..(1,3.25);
\draw[dark-green, very thick] (2,6)..controls (2,5) and (1.5,3.25)..(1,3.25);

\draw(0.5,3)--(0.3,3.2);\draw(1.5,3)--(1.7,3.2);

\node at (0,0){$\gpoint$};
\node at (2,0) {$\gpoint$};
\node at (0,6){$\gpoint$};
\node at (2,6) {$\gpoint$};
\draw[thick] (1,-2)--(1,-0.4);
\draw[thick] (1,2.75)--(1,3.25);
\draw[thick] (1,5.7)--(1,7);

\draw[->] (0.8,-1) arc (-180:80:0.2);
\node at (1,2.75) {$\cross$};
\node at (1,3.25) {$\cross$};

\end{tikzpicture}}&

\scalebox{0.6}{
  \begin{tikzpicture}[scale=0.7,>=stealth]
  
 \node (1+) at (0,2) {$1^+$};
 \node (1-) at (0,-2) {$1^-$};
 \node (3) at (2,0) {$3$};
 \node (2) at (6,0) {$2$};
 \node (4+) at (4,2) {$4^+$};
 \node (4-) at (4,-2) {$4^-$};
 \draw[thick,->] (1+)-- (3);
 \draw[thick, ->] (1-)--(3);
 \draw[thick,->] (4+)-- (2);
 \draw[thick,->] (4-)-- (2);
 \draw[thick,->] (3)--(4+);
 \draw[thick, ->] (3)-- (4-);
 \draw[thick, ->] (1+)-- (4+);
 \draw[thick, ->] (1-)-- (4-);
 
 \draw[loosely dotted, thick] (2.5,0.5) arc (45:135:0.7);
  \draw[loosely dotted, thick] (2.5,-0.5) arc (-45:-135:0.7);
   \draw[loosely dotted,thick] (4.5,1.5) arc (-45:-135:0.7);
    \draw[loosely dotted,thick] (4.5,-1.5) arc (45:135:0.7);

  \end{tikzpicture}}\\
 \hline
\end{tabular}

\end{figure}
\end{landscape}

\end{document}